\documentclass[dvipsnames]{amsart}
\usepackage[utf8]{inputenc}
\usepackage[dvips]{graphics}
\usepackage{times}
\usepackage{mathrsfs}
\usepackage[T1]{fontenc}
\usepackage{latexsym}
\usepackage{epsfig}
\usepackage{hyperref,tikz}
\usepackage{amsmath,amsfonts,amsthm,amssymb,amscd}
\usepackage{pstricks} 
\usepackage[myheadings]{fullpage}
\usepackage{array}
\usepackage{multirow}
\usepackage{longtable}
\usepackage{xcolor}
\usepackage{graphicx}
\usepackage{multicol}
\makeatletter
\newcommand{\addresseshere}{%
  \enddoc@text\let\enddoc@text\relax
}
\makeatother
\usepackage{amsrefs}
\usepackage{listings}

\definecolor{codegray}{rgb}{0.5,0.5,0.5}
\definecolor{lightdarkgreen}{rgb}{0.65,0.78,0.4}
\lstdefinestyle{mystyle}{
  backgroundcolor=\color{White},   
  commentstyle= \color{codegray},
  keywordstyle=\color{PineGreen},
  numberstyle=\tiny\color{Black},
  stringstyle=\color{lightdarkgreen},
  basicstyle=\ttfamily\fontseries{sb}\footnotesize,
  breakatwhitespace=false,         
  breaklines=true,                 
  captionpos=b,                    
  keepspaces=true,                 
  numbers=left,                    
  numbersep=5pt,                  
  showspaces=false,                
  showstringspaces=false,
  showtabs=false,                  
  tabsize=2
}

\usepackage{mmacells}

\mmaDefineMathReplacement[≤]{<=}{\leq}
\mmaDefineMathReplacement[≥]{>=}{\geq}
\mmaDefineMathReplacement[≠]{!=}{\neq}
\mmaDefineMathReplacement[→]{->}{\to}[2]
\mmaDefineMathReplacement[⧴]{:>}{:\hspace{-.2em}\to}[2]
\mmaDefineMathReplacement{∉}{\notin}
\mmaDefineMathReplacement{∞}{\infty}

\mmaSet{
  morefv={gobble=2},
  linklocaluri=mma/symbol/definition:#1,
  morecellgraphics={yoffset=1.9ex}
}

\newtheorem{theorem}{Theorem}[section]

\newtheorem{definition}[theorem]{Definition}
\newtheorem{example}[theorem]{Example}
\newtheorem{lemma}[theorem]{Lemma}

\newtheorem{innercustomthm}{Theorem}
\newenvironment{customthm}[1]
  {\renewcommand\theinnercustomthm{#1}\innercustomthm}
  {\endinnercustomthm}

\newcommand{\kommentar}[1]{}

\newcommand{\NN}{\mathbb{N}}
\newcommand{\ZZ}{\mathbb{Z}}

\newcommand\be{\begin{equation}}
\newcommand\ee{\end{equation}}
\newcommand\bea{\begin{eqnarray}}
\newcommand\eea{\end{eqnarray}}
\newcommand\bi{\begin{itemize}}
\newcommand\ei{\end{itemize}}
\newcommand\ben{\begin{enumerate}}
\newcommand\een{\end{enumerate}}
\newcommand\bc{\begin{center}}
\newcommand\ec{\end{center}}
\newcommand\ba{\begin{array}}
\newcommand\ea{\end{array}}

\newcommand{\Z}{\ensuremath{\mathbb{Z}}}

\newcommand{\N}{\mathbb{N}}

\newtheorem{thm}{Theorem}[section]

\theoremstyle{definition}
\newtheorem{rek}[thm]{Remark}

\newcommand{\ga}{\alpha}

\newcommand{\gG}{\gamma}

\numberwithin{equation}{section}
\newcommand{\floor}[1]{\left \lfloor #1 \right \rfloor}
\newcommand{\al}[1]{\alpha_{#1}}
\newcommand{\bt}[1]{\beta_{#1}}
\newcommand{\w}[0]{\varpi}

\newcommand{\ccc}[0]{\mathfrak{sp}_6(\mathbb{C})}
\newcommand{\g}[0]{\mathfrak{g}}

\newcommand{\e}[0]{\varepsilon}

\newcommand{\A}{\mathcal{A}}
\newcommand{\highestroot}{\tilde{\alpha}}

\newcommand{\poly}[2]{#1_{#2}}

\title{On Kostant's weight $q$-multiplicity formula for $\ccc$}

\author{Pamela E. Harris}\thanks{P.~E. Harris was supported by a Karen Uhlenbeck EDGE Fellowship. All authors thank Williams College for research funding support throughout the completion of this project.
}
\author{Peter Hollander}
\author{Daniel C. Qin}
\author{Maria Rodriguez-Hertz}
\address[P.~E.~Harris, P.~Hollander, M.~Rodriguez-Hertz]{Williams College,
Department of Mathematics and Statistics, Williamstown, MA, USA}

\email[P.~E.~Harris]{\textcolor{blue}{\href{mailto:peh2@williams.edu}{peh2@williams.edu}}}

\email[P.~Hollander]{\textcolor{blue}{\href{mailto:p.hollander99@gmail.com}{p.hollander99@gmail.com}}}

\email[M.~Rodriguez-Hertz]{\textcolor{blue}{\href{mailto:marurh20@gmail.com}{marurh20@gmail.com}}}

\address[D.~C.~Qin]{Kalamazoo College, Department of Mathematics,
Kalamazoo, MI, USA}

\email[D.~C.~Qin]{\textcolor{blue}{\href{mailto:dcqin17@gmail.com}{dcqin17@gmail.com}}}

\date{July 2021}



\begin{document}

\begin{abstract}
Kostant’s weight $q$-multiplicity formula is an alternating sum over a finite group known as the Weyl group, whose terms involve the $q$-analog of Kostant’s partition function. The $q$-analog of the partition function is a polynomial-valued function defined by $\wp_q(\xi)=\sum_{i=0}^k c_i q^i$, where $c_i$ is the number of ways the weight $\xi$ can be written as a sum of exactly $i$ positive roots of a Lie algebra $\mathfrak{g}$. The evaluation of the $q$-multiplicity formula at $q = 1$ recovers the multiplicity of a weight in an irreducible highest weight representation of $\mathfrak{g}$. In this paper, we specialize to the Lie algebra $\mathfrak{sp}_6(\mathbb{C})$ and we provide a closed formula for the $q$-analog of Kostant's partition function, 
which extends recent results of Shahi, Refaghat, and Marefat. We also describe the supporting sets of the multiplicity formula (known as the Weyl alternation sets of $\mathfrak{sp}_6(\mathbb{C})$), and use these results to provide a closed formula for the $q$-multiplicity for any pair of dominant integral weights of $\mathfrak{sp}_6(\mathbb{C})$. Throughout this work, we provide code to facilitate these computations. 
\end{abstract}

\maketitle


\section{Introduction}\label{sec:intro}
We study a combinatorial problem arising in the representation theory of simple classical Lie algebras. 
To begin, we recall the theorem of the highest weight, which asserts that every irreducible complex representation of a simple Lie algebra $\mathfrak{g}$ arises as a highest weight representation with dominant integral weight $\lambda$. In this paper, we denote such a representation by $L(\lambda)$. In this setting it is of interest to compute the multiplicity of a weight $\mu$ in $L(\lambda)$ which can be obtained by using Kostant's weight multiplicity formula \cite{KMF}:
\begin{align}
    m(\lambda,\mu) &= \sum_{\sigma \in W} (-1)^{\ell(\sigma)} \wp(\sigma(\lambda+\rho)-(\mu+\rho)).\label{eq:KWMF}
\end{align}
In Equation \eqref{eq:KWMF}, $W$ denotes the Weyl group of $\g$, a finite group generated by reflections orthogonal to a set of simple roots of $\g$, $\ell(\sigma)$ denotes the length of $\sigma \in W$, and $\rho = \frac{1}{2} \sum_{\alpha \in \Phi^+} \alpha$ with $\Phi^+$ being a set of positive roots of $\g$. 
The terms of the alternating sum in Equation~\eqref{eq:KWMF} are values of Kostant's partition function, which we denote by $\wp$, and which counts the number of ways to express its input as a nonnegative integral linear combination of the positive roots in $\Phi^+$. 

A well-known generalization of Kostant's weight multiplicity formula, due to Lusztig \cite{LL}, is known as the $q$-analog of Kostant's weight multiplicity formula, and it replaces the partition function $\wp$ with its $q$-analog, denoted $\wp_q$. The $q$-analog of the partition function is defined as follows: For a weight $\xi$, $\wp_q$ is a polynomial-valued function: \[
\wp_q(\xi)=c_0+c_1q+\cdots+c_kq^k,\]
where $c_i$ denotes the number of ways to write $\xi$ as a nonnegative integral sum of exactly $i$ positive roots in $\Phi^+$. In this way $\wp_q(\xi)$ when evaluated at $q=1$ recovers the value of the partition function $\wp(\xi)$.

Finding general closed formulas for Kostant's partition function  is very difficult. Most results in the literature are either formulas for low rank Lie algebras or they provide bijections to other families of combinatorial objects. For example, among the results known include closed formulas for the value of Kostant's partition function for Lie algebras of types $A_2$, $A_3$, $B_2$, $C_2$, $C_3$, and $\mathfrak{g}_2$ \cites{Tarski,SRM,CGHLMRK, HL};
there also exist formulas for the $q$-analog of Kostant's partition function for Lie algebras of types $A_2$, $A_3$, $C_2$, and $\mathfrak{g}_2$ and for the classical Lie algebras when the input is the highest root \cites{GHLMMRKT, HL, HRSS, CGHLMRK, HIO}. In 2020, Benedetti, Hanusa, Harris, Morales, and Simpson established bijections between Kostant's partition function (for all simple algebras) and sets of multiplex juggling sequences \cite{Simpson}. Also in 2020, Harris, Rahmoeller, and Schneider provided an approach to the study of the asymptotic behavior of Kostant's partition function \cite{HRS}.

Thus, given the lack of general formulas for Kostant's partition function and the reliance of the weight multiplicity formula on the value of the partition function, there are also few formulas for weight multiplicities. 
Results in this direction are summarized in Table~\ref{tab:known-multiplicities}.
Moreover, there has been some success in implementing computer programs to compute weight $q$-multiplicities for the simple Lie algebras \cites{HIS}.

In spite of of these complications, it has been computationally observed that the number of terms contributing nontrivially to Kostant's weight multiplicity formula is often much smaller than the size of the entire Weyl group, which is either factorial or exponential-factorial in order. In light of this, one can study the elements of the Weyl group that contribute nontrivially to the multiplicity formula. That is, given weights $\lambda$ and $\mu$ define the Weyl alternation set as \[\A(\lambda, \mu)=\{\sigma\in W\,:\;\wp(\sigma(\lambda+\rho)-\mu-\rho)>0 \},\] with the goal of describing and enumerating the sets $\A(\lambda,\mu)$ as the weights $\lambda$ and $\mu$ vary among the (dominant) fundamental weight lattice.
Work to describe and enumerate the Weyl alternation sets has considered cases where $\lambda$ is the highest root or a sum of the simple roots of a simple Lie algebra of types $A_r$, $B_r$, $C_r$, and $D_r$, see \cites{HIO,ChangHarrisInsko,HIS,HIW,Hthesis,GHLMMRKT}. The Weyl alternation sets of $A_2$, $A_3$, and all other simple Lie algebras of rank 2 have been described in \cites{HLM,LRRRHTU}.

In the current work, we specialize our study to $\mathfrak{sp}_6(\mathbb{C})$ (the Lie algebra of type $C_3$) and extend the results of Shahi, Refaghat, and Marefat~\cite{SRM}*{Theorem 1.2} by first providing a new formula for the $q$-analog of Kostant's partition function, shown below and proven in Section~\ref{sec:KPF}.

\begin{theorem} \label{thm:q-konstants-equation}
Let $\mu = m\alpha_1 + n\alpha_2 + k\alpha_3$, with $m,n,k\in\mathbb{N}$, be a weight of $\ccc$. Then, the $q$-analog of Kostant's partition function, $\wp_q(m\alpha_1 + n\alpha_2 + k\alpha_3)$, can be calculated using the following formula: 

\begin{equation*} 
    q^{m+n+k}
    \sum_{h=0}^{\hat{h}}
    \sum_{g=0}^{\hat{g}}
    \sum_{f=0}^{\hat{f}}
    \sum_{i=0}^{\hat{i}}
    \sum_{d=0}^{\hat{d}}
    \sum_{e=0}^{\hat{e}} \left(\frac{1}{q}\right)^{d+e+2f+3g+4h+2i}
\end{equation*}
\noindent where 
\begingroup
\allowdisplaybreaks
\begin{align*}
    &\hat{d} = \min\{m-2h-g-f, n-2h-2g-f-2i\},\\
    &\hat{e} = \min\{n-2h-2g-f-2i-d, k-h-g-f-i\},\\
    &\hat{f} = \min\{m-2h-g, n-2h-2g, k-h-g\},\\
    &\hat{g} = \min\left\{m-2h, \floor{\frac{n-2h}{2}}, k-h\right\},\\
    &\hat{h} = \min\left\{\floor{\frac{m}{2}}, \floor{\frac{n}{2}}, k\right\}, \text{ and }\\
    &\hat{i} = \min\left\{\floor{\frac{n-2h-2g-f}{2}}, k-h-g-f\right\}.
\end{align*}
\endgroup
\end{theorem}
With the above result at hand, in Section \ref{sec:multiplicity} we define and describe the \emph{forbidden} sets of Weyl group elements as the sets consisting of the elements $\sigma \in W$ for which the corresponding term in the multiplicity formula is guaranteed to be zero. This result allows us to 
establish Theorem \ref{thm:alternationSetsFinal}, in which we present 46 distinct Weyl alternation sets $\A(\lambda,\mu)$ for the Lie algebra $\ccc$ when $\lambda$ and $\mu$ lie in the dominant chamber of the fundamental weight lattice.

Finally, in Section~\ref{sec:closed-formula-q-analog-KWMF} we give a closed formula for the value of the $q$-analog of Kostant's weight multiplicity formula,
as presented in Theorem~\ref{thm:finalThm}. The proof of this result utilizes the Weyl alternation sets of $\mathfrak{sp}_6(\mathbb{C})$ from Section~\ref{sec: alt sets} and Theorem~\ref{thm:q-konstants-equation}, which gives the $q$-analog of Kostant's partition function.
We end by remarking that in light of the heavy computational requirements in our main result, much of the article relies on code that we include in the accompanying appendices. For the use of the reader, the code can be downloaded at \cite{CODE}.
The article concludes in Section \ref{sec:future} where we present a direction for future work.

\section{Background} \label{sec:background}
In this section, we provide some background on known results related to Kostant's partition function. To begin we let $\g$ be a simple Lie algebra, $\Phi$ be the set root system, $\Phi^+ \subset \Phi$ be its set of positive roots, and $\Delta$ be the set of simple roots. If $\mu$ is a weight of $\g$, then \emph{Kostant's partition function (KPF)}, denoted $\wp(\mu)$, evaluates to the number of ways we can write $\mu$ as a nonnegative integral linear combination of the elements of $\Phi^+$.
For a weight $\mu$, the $q$-analog of Kostant's partition function is the polynomial-valued formula \begin{equation}
\wp_q(\mu)=c_0+c_1q+\cdots+c_kq^k    \label{eq:polykpf}
\end{equation}
where $c_i$ denotes the number of ways to write $\mu$ as a nonnegative integral sum of exactly $i$ positive roots in $\Phi^+$. In this way, $\wp_q(\mu)$, when evaluated at $q=1$, recovers the value of the partition function.

We note that the highest exponent appearing in the polynomial \eqref{eq:polykpf} depends solely on the input weight $\mu$. That is, if $\mu=\sum_{i=1}^{\ell}n_i\alpha_i$ with $n_1,\ldots n_\ell\in\mathbb{N}:=\{0,1,2,3,\ldots\}$, then we can write 
\[\wp_q(\mu)=\sum_{j=0}^\infty c_jq^j\]
where $c_j=0$ whenever $j> \sum_{i=1}^\ell n_i$. This is because we would be unable to express $\mu$ as a sum of more positive roots that the sum of the coefficients of the simple roots. 

While there is no general formula for either Kostant's partition function nor its $q$-analog, there has been some success in finding closed formulas for some low rank simple Lie algebras. Table \ref{table:known-results} compiles some of the known formulas along with their original references. 

\begin{center}
\begin{longtable}[c]{| p{0.4in} | p{4.4in} | p{.75in} |}
\caption{Table of Known Results} \label{table:known-results}\\
\hline
Type & Results & References\\ \hline
\hline
\endfirsthead
\hline
Type & Results & References\\ \hline
\hline
\endhead
 
\hline
\endfoot
 
\multirow{2}{0.4in}{$A_2$} & \centering$\wp(m\alpha_1 + n\alpha_2) =\min(m,n)+1$ & \multirow{2}{0.75in}{\cites{HL, Tarski}} \\ 
\cline{2-2}

 & \centering$\wp_q(m\alpha_1 + n\alpha_2) = \sum_{j=\max(m,n)}^{m+n} q^j$ & \\
\hline

\multirow{15}{0.4in}{$A_3$} & If $n \leq m,k$: & \multirow{15}{0.75in}{\cites{SRM, HRSS, Tarski}}\\
 & \centering$\wp(m\alpha_1 + n\alpha_2 + k\alpha_3) = \frac{1}{6}(n+1)(n+2)(n+3)$ & \\
 & If $n \geq m,k$ and $k \leq n-m$: & \\
 & \centering$\wp(m\alpha_1 + n\alpha_2 + k\alpha_3) = \frac{1}{6}(m+1)(m+2)(3k-m+3)$ & \\
 & If $n \geq m,k$ and $k>n-m$: & \\
 & \centering$\wp(m\alpha_1 + n\alpha_2 + k\alpha_3) = \frac{1}{6}(m+1)(m+2)(3k-m+3)-\frac{1}{6}A(A+1)(A+2)$ & \\
 & \centering for $A=k+m-n$ & \\
 & If $m \leq n \leq k$: & \\
 & \centering$\wp(m\alpha_1 + n\alpha_2 + k\alpha_3) = \frac{1}{6}(m+1)(m+2)(3n-2m+3)$ & \\
\cline{2-2}

 & \centering $\wp_q(m\alpha_1 + n\alpha_2 + k\alpha_3) = \sum_{f=0}^{\hat{f}} \sum_{d=0}^{\hat{d}} \sum_{e=0}^{\hat{e}} q^z$ & \\
 & where & \\
 & $\hat{f}=\min(m,n,k)$, & \\
 & $\hat{d}=\min(m-f,n-f)$, & \\
 & $\hat{e}=\min(n-f-d,k-f)$, & \\
 & $z=m+n+k-d-e-2f$. & \\
\hline

\multirow{10}{0.4in}{$B_2$} & If $m \leq n$: & \multirow{10}{0.75in}{\cite{Tarski}} \\
 & \centering $\wp(m\alpha_1 + n\alpha_2) =b(m)$ & \\
 & If $n \leq m \leq 2n$: & \\
 & \centering $\wp(m\alpha_1 + n\alpha_2) =b(m) - q_2(m - n -1) = q_2(n) - b(2n - m - 1)$ & \\
 & If $2n \leq m$: & \\
 & \centering $\wp(m\alpha_1 + n\alpha_2) =q_2(n)$ & \\
 & where & \\
 & $q_2(r) = \frac{1}{2}(r+1)(r+2)$, & \\
 & $b(r) = \frac{1}{4}(r+2)^2$ for $r$ even, and & \\
 & \hspace{0.3in}$= \frac{1}{4}(r+1)(r+3)$ for $r$ odd. & \\
\hline

\multirow{7}{0.4in}{$B_r$ $r\geq3$} & For the $q$-analog of KPF:  & \multirow{7}{0.75in}{\cite{HRSS}} \\
 & If $1 \leq i<j \leq r$: & \\
 & \centering $\wp_q(\al{i}+\al{i+1}+\dots+\al{j}) = q(1+q)^{j-i}$ & \\
 & If $l \geq 3$ and $x, y \in \mathbb{Z}_{>0}$: & \\ 
 & \centering $\wp_q(x\al{1} + y\al{2} + \al{3} + \dots + \al{l}) = q(1+q)^{l-3} [\left<x-1, y-1\right>$ & \\
 & \centering $+ \left<x, y-1\right> + \left<x, y\right>]$ & \\
  & where $\left< m, n\right> := \wp_q(m\al{1} + \al{2})$ for $m, n \in \Z_{\geq0}$ & \\
\hline

\multirow{2}{0.4in}{$C_2$} & If  $n \geq m$: & \multirow{2}{0.75in}{\cites{CGHLMRK, HL}}\\
 & \centering $\wp(m\alpha_1 + n\alpha_2) = \left(\floor{\frac{m}{2}}+1\right)\left(m-\floor{\frac{m}{2}}+1\right)$ & \\
\multirow{7}{0.4in}{$C_2$} & If  $2n-1>m>n$: & \multirow{7}{0.75in}{\cites{CGHLMRK, HL}}\\
  & \centering $\wp(m\alpha_1 + n\alpha_2) = \frac{2mn-m^2-n^2+m+n}{2}+\floor{\frac{m}{2}}\left(m-\floor{\frac{m}{2}}\right)+1$ & \\
  & If  $2n>m\geq2n-1>n$: & \\
  & \centering $\wp(m\alpha_1 + n\alpha_2) = \left(\floor{\frac{m}{2}}+1\right)\left(n-\frac{1}{2}\floor{\frac{m}{2}}+1\right)$ &  \\
 & If  $m \geq 2n$: & \\
 & \centering $\wp(m\alpha_1 + n\alpha_2) = \frac{(n+1)(n+2)}{2}$ & \\
\cline{2-2}

 & \centering$\wp_q(m\alpha_1 + n\alpha_2) = \sum^{\min(\floor{\frac{m}{2}},n)}_{i=0} \sum^{m+n-2i}_{j=\max(m-i,n)} q^j$ & \\
\hline

\multirow{7}{0.4in}{$C_3$} & \centering$\wp(m\alpha_1 + n\alpha_2 + k\alpha_3) = \sum_{a=0}^{\hat{a}} \sum_{b=0}^{\hat{b}} \sum_{c=0}^{\hat{c}} \wp_{\mathfrak{sl}}(\hat{\gG})$ & \multirow{7}{0.75in}{\cite{SRM}} \\
 & where & \\
 & $\hat{a} = \min(m, \floor{\frac{1}{2}n}, k)$, & \\
 & $\hat{b} = \min(\floor{\frac{1}{2}(m-a)}, \floor{\frac{1}{2}n}-a, k-a)$, & \\
 & $\hat{c} = \min(\floor{\frac{1}{2}n}-a-b, k-a-b)$, & \\
 & $\hat{\gG} = (m-a-2b)\al{1} + (n-2a-2b-2c)\al{2}+ (k-a-b-c)\al{3}$, & \\
 & and $\wp_{\mathfrak{sl}}$ refers to KPF with respect to $\mathfrak{sl}_4(\mathbb{C})$. & \\
\hline

\multirow{24}{0.4in}{$\g_2$} & If $m \leq n$: & \multirow{24}{0.75in}{\cites{CGHLMRK, Tarski}} \\
 & \centering$\wp(m\alpha_1 + n\alpha_2) = g(m)$ & \\
 & If $n \leq m \leq \frac{3}{2}n$: & \\
 & \centering $\wp(m\alpha_1 + n\alpha_2) = g(m) - h(m-n-1)$ & \\
 & If $\frac{3}{2}n \leq m \leq 2n$: & \\
 & \centering $\wp(m\alpha_1 + n\alpha_2) = h(n) - g(3n-m-1) + h(2n-m-2)$ & \\
 & If $2n \leq m \leq 3n$: & \\
 & \centering $\wp(m\alpha_1 + n\alpha_2) = h(n) - g(3n - m - 1)$ & \\
 & If $3n \leq m$: & \\
 & \centering $\wp(m\alpha_1 + n\alpha_2) = h(n)$ & \\
 & where & \\
 & $h(r) = \frac{1}{48}(r+2)(r+4)(r^2 + 6r + 6)$ for $r$ even, & \\
 & \hspace{0.3in}$= \frac{1}{48}(r+1)(r+3)^2(r+5)$ for $r$ odd, & \\
 & & \\
 & $g(r) = \frac{1}{432}(r+6)(r^3 + 14r^2 + 54r + 72)$ for $r \equiv 0 \mod 6$, & \\
 & \hspace{0.3in}$= \frac{1}{432}(r+5)^2(r^2 + 10r + 13)$ for $r \equiv 1 \mod 6$, & \\
 & \hspace{0.3in}$= \frac{1}{432}(r+4)(r^3 + 16r^2 + 74r + 68)$ for $r \equiv 2 \mod 6$, & \\
 & \hspace{0.3in}$= \frac{1}{432}(r+3)^2(r+5)(r+9)$ for $r \equiv 3 \mod 6$, & \\
 & \hspace{0.3in}$= \frac{1}{432}(r+2)(r+8)(r^2+10r+22)$ for $r \equiv 4 \mod 6$, and & \\
 & \hspace{0.3in}$= \frac{1}{432}(r+1)(r+5)(r+7)^2$ for $r \equiv 5 \mod 6$, & \\
\cline{2-2}

 & \centering$\wp_q(m\alpha_1 + n\alpha_2) = \sum_{i=0}^{\hat{i}} \sum_{j=0}^{\hat{j}} \sum_{k=0}^{\hat{k}} \sum_{l=0}^{\hat{l}} q^{z}$ &  \\
 & where $\hat{i} = \min(\floor{\frac{m}{3}}, \floor{\frac{n}{2}})$, $\hat{j} = \min(\floor{\frac{m-3i}{3}}, n-2i)$,& \\
 & $\hat{k} = \min(\floor{\frac{m-3i-3j}{2}}, n-2i-j)$, $\hat{l} = \min(m-3i-3j-2k, n-2i-j-k),$& \\
 & and $z = m+n-4i-3j-2k-l$. & \\
\end{longtable}
\end{center}

To make our approach precise, we now give the needed background on $\ccc$. 

\subsection{Background on \texorpdfstring{$\ccc$}{sp\_6(C)}}
We begin by more generally giving the background for $\mathfrak{sp}_{2r}(\mathbb{C})$, which we then specialize to $r=3$. Let $\e_1,\e_2,\ldots,\e_r$ denote the standard basis elements of $\mathbb{R}^r$ and let $\alpha_i = \e_i - \e_{i+1}$ for $1 \leq i \leq r-1$ and $\alpha_r = 2\e_r$.
For $r \geq 3$, the set of simple roots of $\mathfrak{sp}_{2r}(\mathbb{C})$ is $\Delta = \{\alpha_i : 1 \leq i \leq r\}$, and the set of positive roots is \[\Phi^+ = \{\e_i - \e_j, \e_i+\e_j : 1 \leq i < j\leq r\} \cup  \{2\e_i : 1 \leq i \leq r\}.\] 
The fundamental weights of $\mathfrak{sp}_{2r}(\mathbb{C})$ are given by \cite{GW}*{p. 146}:
\[\w_i = \al{1} + 2\al{2} + \dots + (i-1)\alpha_{i-1} + i \left(\al{i} + \al{i+1} + \dots + \al{r-1} + \frac{1}{2} \al{r}\right)\] whenever $1 \leq i \leq r$.

In particular, $\ccc$ has simple roots $\Delta = \{\alpha_1, \alpha_2, \alpha_3\}$, positive roots \[\Phi^+ = \{\alpha_1,
\alpha_2, 
\alpha_3, 
\al{1} + \al{2},
\al{2} + \al{3}, 
\al{1} + \al{2} + \al{3}, 
\al{1} + 2\al{2} + \al{3}, 
2\al{1} + 2\al{2} + \al{3}, 
2\al{2} + \al{3}\},\] 
and fundamental weights
\begin{align*}
    \w_1 &= \al{1} + \al{2} + \frac{1}{2} \al{3}, \\
    \w_2 &= \al{1} + 2\al{2} + \al{3}, \text{ and}\\
    \w_3 &= \al{1} + 2\al{2} + \frac{3}{2} \al{3}.
\end{align*}
Note that $\rho=3\al{1}+5\al{2}+3\al{3}=\w_1+\w_2+\w_3$.
Now recall that the Weyl group of $\ccc$ is a group generated by $s_1,s_2,s_3$, the reflections orthogonal to the simple roots $\alpha_1,\alpha_2,\alpha_3$, respectively. These generators 
act linearly on the simple roots in the following way:

\begin{multicols}{3}
\noindent
\begin{align*}
s_1 :\,  & \al{1} \mapsto -\al{1} \\
& \al{2} \mapsto \al{1}+\al{2} \\
& \al{3} \mapsto \al{3},
\end{align*}
\begin{align*}
s_2 :\,  & \al{1} \mapsto \al{1}+\al{2} \\
& \al{2} \mapsto -\al{2} \\
& \al{3} \mapsto 2\al{2}+\al{3},
\end{align*}
\begin{align*}
s_3 :\,  & \al{1} \mapsto \al{1} \\
& \al{2} \mapsto \al{3}+\al{2} \\
& \al{3} \mapsto -\al{3}.
\end{align*}
\end{multicols}
In addition, for $1\leq i,j\leq 3$ we have that 
\begin{align}\label{eq:simples acting on fun}
s_i(\w_j)&=\begin{cases}\w_j&\mbox{if $i\neq j$}\\
\w_j-\al{j}&\mbox{if $i=j$}
.\end{cases}
\end{align}

\noindent The full set of elements in $W$, as presented in \cite{GW}, is:
\begin{align*}
W = \{&1, s_1, s_2, s_3,
s_1  s_2, s_2  s_1, s_2  s_3, s_3  s_1, s_3  s_2,
s_1  s_2  s_1, s_1  s_2  s_3, s_2  s_3  s_1, s_2  s_3  s_2, s_3  s_2  s_1,\\
&s_3  s_1  s_2, s_3  s_2  s_3,
s_1  s_2  s_3  s_1, s_1  s_2  s_3  s_2,
s_2  s_3  s_2  s_1, s_2  s_3  s_1  s_2,
s_3  s_1  s_2  s_1, s_3  s_2  s_3  s_1, s_3  s_2  s_3  s_2, \\ 
&s_3  s_1  s_2  s_3, 
s_1  s_2  s_3  s_2  s_1, s_1  s_2  s_3  s_1  s_2, 
s_2  s_3  s_1  s_2  s_1, s_2  s_3  s_1  s_2  s_3,
s_3  s_1  s_2  s_3  s_1, s_3  s_1  s_2  s_3  s_2,\\
&s_3  s_2  s_3  s_2  s_1, s_3  s_2  s_3  s_1  s_2,
s_1  s_2  s_3  s_1  s_2  s_1,
s_2  s_3  s_1  s_2  s_3  s_1, s_2  s_3  s_1  s_2  s_3  s_2, 
s_3  s_1  s_2  s_3  s_1  s_2,\\ 
&s_3  s_1  s_2  s_3  s_2  s_1, s_3  s_2  s_3  s_1  s_2  s_1, s_3  s_2  s_3  s_1  s_2  s_3,
s_2  s_3  s_1  s_2  s_3  s_2  s_1, s_2  s_3  s_1  s_2  s_3  s_1  s_2, \\
&s_3  s_2  s_3  s_1  s_2  s_3  s_2, s_3  s_1  s_2  s_3  s_1  s_2  s_1, s_3  s_2  s_3  s_1  s_2  s_3  s_1,
s_2  s_3  s_1  s_2  s_3  s_1  s_2  s_1,\\ 
&s_3  s_2  s_3  s_1  s_2  s_3  s_2  s_1,
s_3  s_2  s_3  s_1  s_2  s_3  s_1  s_2, 
s_3  s_2  s_3  s_1  s_2  s_3  s_1  s_2  s_1\}.
\end{align*}
For $\sigma\in W$, $\ell(\sigma)$ refers to the length of $\sigma$ when it is written in its minimal form, i.e., in a shortest expression as a product of the generators. For example, if $\sigma=s_2  s_3  s_1  s_2  s_3  s_1  s_2  s_1$, then $\sigma$ could have multiple expressions, but among them the one given has minimal length, which is $\ell(\sigma)=8$.

Now recall that since the partition function evaluates to zero for most $\sigma \in W$, we define the Weyl alternation set for a pair of dominant integral weights $\lambda$ and $\mu$ as follows: \[\mathcal{A}(\lambda,\mu):\{\sigma\in W\; :\; \wp_q(\sigma(\lambda+\rho)-\rho-\mu)>0\}.\]

\noindent In other words, $\A(\lambda, \mu)$ contains all $\sigma \in W$ such that $\sigma(\lambda+\rho)-(\mu+\rho)$ can be written as a nonnegative integral combination of $\al{1},\al{2},$ and $\al{3}$.

\begin{rek}
In Table~\ref{tab:known-multiplicities} we  present some Lie algebras and provide the references which give the Weyl alternation sets $\A(\lambda,\mu)$ for varying pairs of weights $\lambda$ and $\mu$. We note that our contributions in this paper provide the Weyl alternation sets for $\ccc$, which were previously unknown.
\begin{table}[h!]
    \centering
    \caption{Table of Known Weyl Alternation Sets}
    \begin{tabular}{|c|c|}
        \hline
         Type & Reference \\
         \hline
         \hline
         $A_2$ &  \cite{HLM}*{p. 2}\\
         \hline
         $A_3$ &  \cite{GHLMMRKT}*{p. 37}\\
         \hline
         $B_2$ &  \cite{LRRRHTU}*{p. 8}\\
         \hline
         $C_2$ &  \cite{LRRRHTU}*{p. 10}\\
         \hline
         $C_3$ &  This paper, Theorem \ref{thm:alternationSetsFinal} \\
         \hline
         $D_2$ &  \cite{LRRRHTU}*{p. 11} \\ 
         \hline
         $\mathfrak{g}_2$ &  \cite{LRRRHTU}*{p. 12-13} \\
         \hline
    \end{tabular}
    \label{tab:known-multiplicities}
\end{table}
\end{rek}

Observe that the Weyl alternation set allows us to reduce the weight multiplicity computation as follows:
\[m_q(\lambda,\mu)=\sum_{\sigma\in W}(-1)^{\ell(\sigma)}\wp_q(\sigma(\lambda+\rho)-\rho-\mu)=\sum_{\sigma\in \mathcal{A}(\lambda,\mu)}(-1)^{\ell(\sigma)}\wp_q(\sigma(\lambda+\rho)-\rho-\mu).\]

\noindent We proceed by letting $\lambda=m\varpi_1+n\varpi_2+k\varpi_3$ and $\mu=x\varpi_1+y\varpi_2+z\varpi_3$ with $m,n,k,x,y,z\in\mathbb{N}$.
Using the change of basis from the fundamental weights to the simple roots we note that 
\begin{align*}
    \lambda &= m\w_1 + n\w_2 + k\w_3 
    = (m+n+k) \al{1} + (m+2n+2k) \al{2} + \left(\frac{m}{2}+n+\frac{3}{2}k\right) \al{3}.
\end{align*}
Similarly, 
\begin{align*}
\mu &= x\w_1 + y\w_2 + z\w_3 = (x+y+z) \al{1} + (x+2y+2z) \al{2} + \left(\frac{x}{2}+y+\frac{3}{2}z\right) \al{3}.
\end{align*}

We now proceed by giving a closed formula for the value of the $q$-analog of Kostant's partition function. This result will be used within the multiplicity formula we derive.

\section{The \texorpdfstring{$q$}{q}-analog of Kostant's partition function for \texorpdfstring{$\ccc$}{sp\_6(C)}}\label{sec:KPF}

In this section, we follow a similar argument to that of \cite{HRSS}*{Proposition 20} to prove Theorem \ref{thm:q-konstants-equation}.
Afterwards, we discuss why simpler closed formulas are not provided, and instead we point the reader to Appendix~\ref{sec:kpf-eval-code} (or to \cite{CODE}) for a computer implementation of this result. 

\begin{customthm}{1.1} 
Let  $\mu = m\alpha_1 + n\alpha_2 + k\alpha_3$, with $m,n,k\in\mathbb{N}$, be a weight of $\ccc$. Then, the $q$-analog of Kostant's partition function, $\wp_q(m\alpha_1 + n\alpha_2 + k\alpha_3)$, can be calculated using the following formula: 
\begin{equation*}
    q^{m+n+k}
    \sum_{h=0}^{\hat{h}}
    \sum_{g=0}^{\hat{g}}
    \sum_{f=0}^{\hat{f}}
    \sum_{i=0}^{\hat{i}}
    \sum_{d=0}^{\hat{d}}
    \sum_{e=0}^{\hat{e}} \left(\frac{1}{q}\right)^{d+e+2f+3g+4h+2i}
\end{equation*}
where 
\begin{align*}
    &\hat{d} = \min\{m-2h-g-f, n-2h-2g-f-2i\},\\
    &\hat{e} = \min\{n-2h-2g-f-2i-d, k-h-g-f-i\},\\
    &\hat{f} = \min\{m-2h-g, n-2h-2g, k-h-g\},\\
    &\hat{g} = \min\{m-2h, \floor{\frac{n-2h}{2}}, k-h\},\\
    &\hat{h} = \min\{\floor{\frac{m}{2}}, \floor{\frac{n}{2}}, k\}, \text{ and}\\
    &\hat{i} = \min\{\floor{\frac{n-2h-2g-f}{2}}, k-h-g-f\}.
\end{align*}
\end{customthm}

\begin{proof}
First, let a partition of the weight $m\alpha_1 + n\alpha_2 + k\alpha_3$ be given such that
\begin{itemize}
    \item $d$ count how many times the root $\bt{4} = \al{1} + \al{2}$ is used,
    \item $e$ count how many times the root $\bt{5} = \al{2} + \al{3}$ is used,
    \item $f$ count how many times the root $\bt{6} = \al{1} + \al{2} + \al{3}$ is used,
    \item $g$ count how many times the root $\bt{7} = \al{1} + 2\al{2} + \al{3}$ is used,
    \item $h$ count how many times the root $\bt{8} = 2\al{1} + 2\al{2} + \al{3}$ is used, and
    \item $i$ count how many times the root $\bt{9} = 2\al{2} + \al{3}$ is used.
\end{itemize}
In other words, let us consider the partition $m\alpha_1 + n\alpha_2 + k\alpha_3 = a\alpha_1 + b\alpha_2 + c\alpha_3 + d\beta_4 + e\beta_5 + f\beta_6 + g\beta_7 + h\beta_8 + i\beta_9$. Note, that $a$, $b$, and $c$ are uniquely determined by our choice of $d$, $e$,$f$,$g$,$h$, and $i$. Namely, $a$ must equal $m-2h-g-f-d$, $b$ must equal $n-2h-2g-f-2i-d-e$, and $c$ must equal $k-h-g-f-i-e$. Now, it suffices to keep track of the total number of positive roots used in this partition, which is given by $h+g+f+i+d+e + (m-2h-g-f-d) + (n-2h-2g-f-2i-d-e) + (k-h-g-f-i-e)$, which we can simplify to $m+n+k-d-e-2f-3g-4h-2i$. 
\end{proof}

Given the appearance of functions involving minimums in the formula of Theorem \ref{thm:q-konstants-equation}, one would ideally consider inequalities involving the parameters $m$, $n$, and $k$ to find closed formulas for $\wp_q(m\alpha_1 + n\alpha_2 + k\alpha_3)$. However, such an  approach is rather intractable. For illustrative purposes we present one such case. 

\begin{example}
Consider the case where $m~\geq~n~\geq~2k$. For the first five summations appearing in Theorem \ref{thm:q-konstants-equation}, after some algebraic manipulations we can determine the expression to which each minimum evaluates. From which, we can then start computing noting that the first sum, whose lower bound is $h=0$ and upper bound is ${\min\{\floor{\frac{m}{2}}, \floor{\frac{n}{2}}, k\}}$. Very quickly, by our assumption, we know that $\min\{\floor{\frac{m}{2}}, \floor{\frac{n}{2}}, k\} = k$, so $0~\leq~h~\leq~k$. 
We can then move onto our second sum, whose lower bound is $g=0$ and upper bound is $\min\{m-2h, \floor{\frac{n-2h}{2}}, k-h\}$. Note that since $2k \geq n$ then $\floor{\frac{n-2h}{2}} \geq k-h$, and since $2k \leq m$ and $k \geq h$ then $m-2h \geq k-h$. So, we know $0 \leq g \leq k-h$. Now, we can look at the third sum with lower bound $f=0$ and upper bound ${\min\{m-2h-g, n-2h-2g, k-h-g\}}$. From our logic in the last summation, we know that $m-2h \geq k-h$, which implies $m-2h-g \geq k-h-g$. And, also from out last summation, we know that $g \leq k-h$ which we can rearrange into $g+h \leq k$. Since $2k \leq n$ (from our assumptions). This means that we know $n-2h-2g \geq k-h-g$. This yields that $0 \leq f \leq k-h-g$.
Our fourth sum's lower bound is $i=0$, and its upper bound is $\min\{\floor{\frac{n-2h-2g-f}{2}}, k-h-g-f\}$. Since $2k \leq n$ and $f \geq \floor{\frac{f}{2}}$, we can quickly see that $\floor{\frac{n-2h-2g-f}{2}} \geq k-h-g-f$, so we know $0 \leq i \leq k-h-g-f$. 

Now, we move onto our fifth sum, whose lower bound is $d=0$ and upper bound is $\min\{m-2h-g-f, n-2h-2g-f-2i\}$. Here, we have to determine an inequality between the values  $m-2h-g-f$  and  $n-2h-2g-f-2i$. This is equivalent to determining an inequality between  $m$  and  $n-g-2i$. Note that, at its highest, the expression on the right hand side, namely $n-g-2i$, is equal to $n$ and that by our assumption $m \geq n$. So, we have that $\min\{m-2h-g-f, n-2h-2g-f-2i\} = n-2h-2g-f-2i$, that is, $0 \leq d \leq n-2h-2g-f-2i$. 

However, once we get to the sixth summation, we are presented with the following problem. We want to determine an inequality  between $n-2h-2g-f-2i-d$ and $k-h-g-f-i$. This is equivalent to finding an inequality between $n-h-g-i-d$  and  $k$. At its highest, the expression on the left hand side, namely $n-h-g-i-d$, is equal to $n$, implying that $n-2h-2g-f-2i-d \geq k-h-g-f-i$. 

However, suppose $d$ is equal to its highest possible value, $n-2h-2g-f-2i$. Then, $n-h-g-i-d = n-h-g-i-d - n+2h+2g+f+2i$, which simplifies to $h+g+f+i$. We know from the fourth sum that $0 \leq i \leq k-h-g-f$, which we can rearrange into $h+g+f \leq h+g+i+f \leq k$ which depending on the values of $h,g$ and $f$ tells us that $h+g+i+f$ can be anywhere from $0$ to $k$. This means that, when $h+g+i+f = 0$ and $d = n-2h-2g-f-2i$, we have $n-2h-2g-f-2i-d \leq k-h-g-f-i$, which contradicts what we had previously found. Note also that this is not dependent on the value of $n$ or $k$.

Even after all of this work, we have reached a point in which we must consider additional subcases (i.e., inequalities between the parameters) in order to actually be able to provide a formula for the value of the $q$-analog of Kostant's partition function.
\end{example}

As illustrated above, although such case-by-case analysis could ensure that closed formulas may be found, the case work needed is overwhelming and can be more aptly found through a computer implementation. 
In light of this, as previously mentioned, in  Appendix~\ref{sec:kpf-eval-code} we provide Wolfram Mathematica Code which evaluates the formula in Theorem \ref{thm:q-konstants-equation} to give exact outputs for the value of the $q$-analog of Kostant's partition function for any input weight of the form $\mu = m\alpha_1 + n\alpha_2 + k\alpha_3$ of $\ccc$.

\section{Forbidden Weyl alternation sets} \label{sec:multiplicity}

Recall from Section~\ref{sec:background} that  Kostant's weight multiplicity formula is difficult to compute due in part to the fact that order of the Weyl group grows  exponential-factorially with respect to the rank of the Lie algebra. 
Moreover, for all $\sigma\in W$, since $\wp(\sigma(\lambda+\rho)-(\mu+\rho)) = 0$ if and only if $\wp_q(\sigma(\lambda+\rho)-(\mu+\rho)) = 0$, we know that the support of the  multiplicity formula is the same as that of its $q$-analog. So, in what follows, we restrict our study of Weyl alternation sets to using Kostant's partition function.
In this section, we aim to give some characterizations of the Weyl alternation sets for the Lie algebra $\ccc$ when we consider $\lambda$ and $\mu$ in the  dominant integral weight lattice. 
We begin by remarking that there exist a possible $2^{48}$ subsets of the Weyl group, since $|W|=48$. However, we find that many of these subsets never arise as Weyl alternation sets for any pair of dominant integral weights $\lambda$ and $\mu$. We call these subsets \textit{forbidden Weyl alternation sets.} We establish that Weyl alternation sets arise from a set of logical statements, and we show that forbidden Weyl alternation sets contain some logical contradictions among their defining logical statements. When it is clear from context, we often omit the word ``Weyl" from ``Weyl alternation sets'' and from ``forbidden Weyl alternation sets.'' Moreover, we remark that these sets are called ``forbidden Weyl alternation sets'' not because they \emph{are} alternation sets, but because they are sets which are \emph{logically forbidden from being} alternation sets.

To find forbidden alternation sets, we begin by considering 1-element subsets of the Weyl group and identify contradictions arising in their defining logical statements when those sets are thought of as subsets of Weyl alternation sets. 
Following this, we then proceed using the same technique by considering 2-element subsets of the Weyl group, and we continue to increment their size as needed until we are able to account for all possible subsets of the Weyl group as either being a forbidden alternation set or in fact an alternation set we can construct (i.e., one for which we can find a $\lambda$ and $\mu$ such that $\A(\lambda,\mu)$ is that set). This approach works because if an $n$-element subset gives a contradiction, any set of Weyl group elements containing that $n-$element subset is a forbidden alternation set. We make this precise next.

\begin{rek}\label{rem:subset}
Given a subset $X$ of the Weyl group, we let $X^c$ be its complement. If $X$ were an alternation set, this would imply that every element $x\in X$ yields $x(\lambda+\rho)-\rho-\mu$ which can be written as a nonnegative integer linear combination of the simple roots, while simultaneously this not being the case for all $y\in X^c$. 
Equivalently, this yields a set of logical statements requiring that for every $x\in X$ the coefficients of the simple roots in the expression $x(\lambda+\rho)-\rho-\mu$ are in $\mathbb{N}$, while for every $y\in X^c$ there exists a coefficient of the simple roots in the expression $y(\lambda+\rho)-\rho-\mu$ which is \textbf{not} in $\mathbb{N}$. 
If in this set of logical statements we can find a subset in which a contradiction arises, then this shows that the set $X$ cannot be an alternation set. In fact, whenever we can find any sublist of logical  statements which give contradictions, they will remain contradictions in any larger list of logical statements containing them. This then guarantees that if we can find ``small'' lists of logical statements in which contradictions arise, then we will be able to eliminate many subsets of the Weyl group from being alternation sets.

In what follows we begin by showing that certain elements can never actually be in any alternation set. Later, we consider pairs of elements being in the alternation set together with a third element not being in the set. Lastly, we consider cases where larger subsets of Weyl group elements are in the alternation set. Together these three types of contradictions will allow us to fully characterize the alternation sets of $\ccc$.
\end{rek}

We begin this analysis by discarding forbidden alternation sets that imply contradictions of three types, which we describe shortly. As it turns out, doing so allows us to fully characterize the Weyl alternation sets of $\ccc$, as we show that the remaining 46 subsets of the Weyl group can be constructed as alternation sets (Theorem~\ref{thm:alternationSetsFinal}).

Recall that in order to describe the set $\mathcal{A}(\lambda,\mu)$ for each pair of dominant integral weights $(\lambda,\mu)$, we need to determine for which $\sigma \in W$ the expression $\sigma(\lambda+\rho)-(\mu+\rho)$ can be written as a nonnegative integral linear combination of the simple roots $\al{1},\al{2},$ and $\al{3}$.
To do so, we consider each $\sigma\in W$ and compute the expression
\begin{equation}
    \sigma(\lambda+\rho)-\rho-\mu=C_{1,\sigma}\alpha_1+C_{2,\sigma}\alpha_2+C_{3,\sigma}\alpha_3,
\end{equation}
where $C_{1,\sigma}, C_{2,\sigma},C_{3,\sigma}$ are the respective coefficients of the simple roots $\alpha_1,\alpha_2,\alpha_3$ when we fully simplify the given expression.
Notice that $\wp(\sigma(\lambda+\rho)-\rho-\mu)>0$ if and only if the corresponding expression $C_{1,\sigma}\alpha_1+C_{2,\sigma}\alpha_2+C_{3,\sigma}\alpha_3$ satisfies the following two conditions:
\begin{enumerate}
    \item $C_{i,\sigma}$ is an integer for all $1\leq i\leq 3$. We call this the \textit{integrality condition}.
    \item $C_{i,\sigma}$ is nonnegative for all $1\leq i\leq 3$. We call this the \textit{nonnegativity condition}.
\end{enumerate}
For fixed $\lambda$ and $\mu$, we want to find for which $\sigma \in W$ these conditions are satisfied.

Recall, from Section~\ref{sec:background}, that we can write $\lambda=m\w_1+n\w_2+k\w_3$ and $\mu=x\w_1+y\w_2+z\w_3$ (with $m,n,k,x,y,z\in\mathbb{N}$) in terms of the simple roots $\alpha_1$, $\alpha_2$, and $\alpha_3$ as follows: \[\lambda = (m+n+k) \al{1} + (m+2n+2k) \al{2} + \left(\frac{m}{2}+n+\frac{3}{2}k\right) \al{3},\] and similarly, \[\mu = (x+y+z) \al{1} + (x+2y+2z) \al{2} + \left(\frac{x}{2}+y+\frac{3}{2}z\right) \al{3}.\]
We also know that $\rho = \frac{1}{2} \sum_{\ga \in \Phi^{+}} \ga$, which can also we expressed as $\rho=3\al{1} + 5\al{2} + 3\al{3}$. 
With these computations, we can then rewrite 
\begin{align*}
    \sigma[\lambda+\rho]-\rho-\mu&=\sigma\left[(m+n+k+3)\al{1} + (m+2n+2k+5)\al{2} + \left(\frac{m}{2}+n+\frac{3}{2}k+3\right)\al{3}\right]\\
    &\qquad\qquad- (x+y+z+3) \al{1} - (x+2y+2z+5) \al{2} - \left(\frac{x}{2}+y+\frac{3}{2}z+3\right) \al{3}\\
    &= C_{1,\sigma}\al{1}+C_{2,\sigma}\al{2}+C_{3,\sigma}\al{3}.
\end{align*}
Next we consider each element of the Weyl group and its associated computation that determines the coefficients $C_{1,\sigma}, C_{2,\sigma}, C_{3,\sigma}$ in terms of the variables $m,n,k,x,y,z$. 
We compile these results in in Table~\ref{table:all-sigmas}.
The interested reader is encouraged to use the accompanying code found in Appendix~\ref{sec:coeff-code} to verify these results. 

{\footnotesize{
\begin{center}
\begin{longtable}[c]{| p{1.23in} | p{1.4in} | p{1.6in} | p{1.6in} | }
\caption{The expressions $C_{1,\sigma}\alpha_1+C_{2,\sigma}\alpha_2+C_{3,\sigma}\alpha_3$ for all $\sigma \in W$.} \label{table:all-sigmas}\\\hline

$\sigma$ & $C_{1, \sigma}$ & $C_{2, \sigma}$ & $C_{3, \sigma}$\\\hline
\hline
\endfirsthead
\hline
$\sigma$ & $C_{1, \sigma}$ & $C_{2, \sigma}$ & $C_{3, \sigma}$\\\hline
\hline
\endhead

\hline
\endfoot
$1$ & $m+n+k-x-y-z$ & $m+2n+2k-x-2y-2z$ & $\frac{m}{2}+n+\frac{3}{2}k-\frac{x}{2}-y-\frac{3}{2}z$\\\hline
$s_1$ & $n+k-x-y-z-1$ & $m+2n+2k-x-2y-2z$ & $\frac{m}{2}+n+\frac{3}{2}k-\frac{x}{2}-y-\frac{3}{2}z$ \\\hline
$s_2$ & $m+n+k-x-y-z$ & $m+n+2k-x-2y-2z-1$ & $\frac{m}{2}+n+\frac{3}{2}k-\frac{x}{2}-y-\frac{3}{2}z$ \\\hline
$s_3$ & $m+n+k-x-y-z$ & $m+2n+2k-x-2y-2z$ & $\frac{m}{2}+n+\frac{k}{2}-\frac{x}{2}-y-\frac{3}{2}z-1$ \\\hline
$s_1  s_2$ & $k-x-y-z-2$ & $m+n+2k-x-2y-2z-1$ & $\frac{m}{2}+n+\frac{3}{2}k-\frac{x}{2}-y-\frac{3}{2}z$ \\\hline
$s_2  s_1$ & $n+k-x-y-z-1$ & $n+2k-x-2y-2z-2$ & $\frac{m}{2}+n+\frac{3}{2}k-\frac{x}{2}-y-\frac{3}{2}z$ \\\hline
$s_2  s_3$ & $m+n+k-x-y-z$ & $m+n-x-2y-2z-3$ & $\frac{m}{2}+n+\frac{k}{2}-\frac{x}{2}-y-\frac{3}{2}z-1$ \\\hline
$s_3  s_1$ & $n+k-x-y-z-1$ & $m+2n+2k-x-2y-2z$ & $\frac{m}{2}+n+\frac{k}{2}-\frac{x}{2}-y-\frac{3}{2}z-1$ \\\hline
$s_3  s_2$ & $m+n+k-x-y-z$ & $m+n+2k-x-2y-2z-1$ & $\frac{m}{2}+\frac{k}{2}-\frac{x}{2}-y-\frac{3}{2}z-2$ \\\hline
$s_1  s_2  s_1$ & $k-x-y-z-2$ & $n+2k-x-2y-2z-2$ & $\frac{m}{2}+n+\frac{3}{2}k-\frac{x}{2}-y-\frac{3}{2}z$ \\\hline
$s_1  s_2  s_3$ & $-k-x-y-z-4$ & $m+n-x-2y-2z-3$ & $\frac{m}{2}+n+\frac{k}{2}-\frac{x}{2}-y-\frac{3}{2}z-1$ \\\hline
$s_2  s_3  s_1$ & $n+k-x-y-z-1$ & $n-x-2y-2z-4$ & $\frac{m}{2}+n+\frac{k}{2}-\frac{x}{2}-y-\frac{3}{2}z-1$ \\\hline
$s_2  s_3  s_2$ & $m+n+k-x-y-z$ & $m-x-2y-2z-4$ & $\frac{m}{2}+\frac{k}{2}-\frac{x}{2}-y-\frac{3}{2}z-2$ \\\hline
$s_3  s_2  s_1$ & $n+k-x-y-z-1$ & $n+2k-x-2y-2z-2$ & $-\frac{m}{2}+\frac{k}{2}-\frac{x}{2}-y-\frac{3}{2}z-3$ \\\hline
$s_3  s_1  s_2$ & $k-x-y-z-2$ & $m+n+2k-x-2y-2z-1$ & $\frac{m}{2}+\frac{k}{2}-\frac{x}{2}-y-\frac{3}{2}z-2$ \\\hline
$s_3  s_2  s_3$ & $m+n+k-x-y-z$ & $m+n-x-2y-2z-3$ & $\frac{m}{2}-\frac{k}{2}-\frac{x}{2}-y-\frac{3}{2}z-3$ \\\hline
$s_1  s_2  s_3  s_1$ & $-k-x-y-z-4$ & $n-x-2y-2z-4$ & $\frac{m}{2}+n+\frac{k}{2}-\frac{x}{2}-y-\frac{3}{2}z-1$ \\\hline
$s_1  s_2  s_3  s_2$ & $-n-k-x-y-z-5$ & $m-x-2y-2z-4$ & $\frac{m}{2}+\frac{k}{2}-\frac{x}{2}-y-\frac{3}{2}z-2$ \\\hline
$s_2  s_3  s_2  s_1$ & $n+k-x-y-z-1$ & $-m-x-2y-2z-6$ & $-\frac{m}{2}+\frac{k}{2}-\frac{x}{2}-y-\frac{3}{2}z-3$ \\\hline
$s_2  s_3  s_1  s_2$ & $k-x-y-z-2$ & $-n-x-2y-2z-6$ & $\frac{m}{2}+\frac{k}{2}-\frac{x}{2}-y-\frac{3}{2}z-2$ \\\hline
$s_3  s_1  s_2  s_1$ & $k-x-y-z-2$ & $n+2k-x-2y-2z-2$ & $-\frac{m}{2}+\frac{k}{2}-\frac{x}{2}-y-\frac{3}{2}z-3$ \\\hline
$s_3  s_2  s_3  s_1$ & $n+k-x-y-z-1$ & $n-x-2y-2z-4$ & $-\frac{m}{2}-\frac{k}{2}-\frac{x}{2}-y-\frac{3}{2}z-4$ \\\hline
$s_3  s_2  s_3  s_2$ & $m+n+k-x-y-z$ & $m-x-2y-2z-4$ & $\frac{m}{2}-\frac{k}{2}-\frac{x}{2}-y-\frac{3}{2}z-3$ \\\hline
$s_3  s_1  s_2  s_3$ & $-k-x-y-z-4$ & $m+n-x-2y-2z-3$ & $\frac{m}{2}-\frac{k}{2}-\frac{x}{2}-y-\frac{3}{2}z-3$ \\\hline
$s_1  s_2  s_3  s_2  s_1$ & $-m-n-k-x-y-z-6$ & $-m-x-2y-2z-6$ & $-\frac{m}{2}+\frac{k}{2}-\frac{x}{2}-y-\frac{3}{2}z-3$ \\\hline
$s_1  s_2  s_3  s_1  s_2$ & $-n-k-x-y-z-5$ & $-n-x-2y-2z-6$ & $\frac{m}{2}+\frac{k}{2}-\frac{x}{2}-y-\frac{3}{2}z-2$ \\\hline
$s_2  s_3  s_1  s_2  s_1$ & $k-x-y-z-2$ & $-m-n-x-2y-2z-7$ & $-\frac{m}{2}+\frac{k}{2}-\frac{x}{2}-y-\frac{3}{2}z-3$ \\\hline
$s_2  s_3  s_1  s_2  s_3$ & $-k-x-y-z-4$ & $-n-2k-x-2y-2z-8$ & $\frac{m}{2}-\frac{k}{2}-\frac{x}{2}-y-\frac{3}{2}z-3$ \\\hline
$s_3  s_1  s_2  s_3  s_1$ & $-k-x-y-z-4$ & $n-x-2y-2z-4$ & $-\frac{m}{2}-\frac{k}{2}-\frac{x}{2}-y-\frac{3}{2}z-4$ \\\hline
$s_3  s_1  s_2  s_3  s_2$ & $-n-k-x-y-z-5$ & $m-x-2y-2z-4$ & $\frac{m}{2}-\frac{k}{2}-\frac{x}{2}-y-\frac{3}{2}z-3$ \\\hline
$s_3  s_2  s_3  s_2  s_1$ & $n+k-x-y-z-1$ & $-m-x-2y-2z-6$ & $-\frac{m}{2}-\frac{k}{2}-\frac{x}{2}-y-\frac{3}{2}z-4$ \\\hline
$s_3  s_2  s_3  s_1  s_2$ & $k-x-y-z-2$ & $-n-x-2y-2z-6$ & $-\frac{m}{2}-n-\frac{k}{2}-\frac{x}{2}-y-\frac{3}{2}z-5$ \\\hline
$s_1  s_2  s_3  s_1  s_2  s_1$ & $-m-n-k-x-y-z-6$ & $-m-n-x-2y-2z-7$ & $-\frac{m}{2}+\frac{k}{2}-\frac{x}{2}-y-\frac{3}{2}z-3$ \\\hline
$s_2  s_3  s_1  s_2  s_3  s_1$ & $-k-x-y-z-4$ & $-m-n-2k-x-2y-2z-9$ & $-\frac{m}{2}-\frac{k}{2}-\frac{x}{2}-y-\frac{3}{2}z-4$ \\\hline
$s_2  s_3  s_1  s_2  s_3  s_2$ & $-n-k-x-y-z-5$ & $-n-2k-x-2y-2z-8$ & $\frac{m}{2}-\frac{k}{2}-\frac{x}{2}-y-\frac{3}{2}z-3$ \\\hline
$s_3  s_1  s_2  s_3  s_1  s_2$ & $-n-k-x-y-z-5$ & $-n-x-2y-2z-6$ & $-\frac{m}{2}-n-\frac{k}{2}-\frac{x}{2}-y-\frac{3}{2}z-5$ \\\hline
$s_3  s_1  s_2  s_3  s_2  s_1$ & $-m-n-k-x-y-z-6$ & $-m-x-2y-2z-6$ & $-\frac{m}{2}-\frac{k}{2}-\frac{x}{2}-y-\frac{3}{2}z-4$ \\\hline
$s_3  s_2  s_3  s_1  s_2  s_1$ & $k-x-y-z-2$ & $-m-n-x-2y-2z-7$ & $-\frac{m}{2}-n-\frac{k}{2}-\frac{x}{2}-y-\frac{3}{2}z-5$ \\\hline
$s_3  s_2  s_3  s_1  s_2  s_3$ & $-k-x-y-z-4$ & $-n-2k-x-2y-2z-8$ & $-\frac{m}{2}-n-\frac{3}{2}k-\frac{x}{2}-y-\frac{3}{2}z-6$ \\\hline
$s_2  s_3  s_1  s_2  s_3  s_2  s_1$ & $-m-n-k-x-y-z-6$ & $-m-n-2k-x-2y-2z-9$ & $-\frac{m}{2}-\frac{k}{2}-\frac{x}{2}-y-\frac{3}{2}z-4$ \\\hline
$s_2  s_3  s_1  s_2  s_3  s_1  s_2$ & $-n-k-x-y-z-5$ & $-m-2n-2k-x-2y-2z-10$ & $-\frac{m}{2}-n-\frac{k}{2}-\frac{x}{2}-y-\frac{3}{2}z-5$ \\\hline
$s_3  s_2  s_3  s_1  s_2  s_3  s_2$ & $-n-k-x-y-z-5$ & $-n-2k-x-2y-2z-8$ & $-\frac{m}{2}-n-\frac{3}{2}k-\frac{x}{2}-y-\frac{3}{2}z-6$ \\\hline
$s_3  s_1  s_2  s_3  s_1  s_2  s_1$ & $-m-n-k-x-y-z-6$ & $-m-n-x-2y-2z-7$ & $-\frac{m}{2}-n-\frac{k}{2}-\frac{x}{2}-y-\frac{3}{2}z-5$ \\\hline
$s_3  s_2  s_3  s_1  s_2  s_3  s_1$ & $-k-x-y-z-4$ & $-m-n-2k-x-2y-2z-9$ & $-\frac{m}{2}-n-\frac{3}{2}k-\frac{x}{2}-y-\frac{3}{2}z-6$ \\\hline
$s_2  s_3  s_1  s_2  s_3  s_1  s_2  s_1$ & $-m-n-k-x-y-z-6$ & $-m-2n-2k-x-2y-2z-10$ & $-\frac{m}{2}-n-\frac{k}{2}-\frac{x}{2}-y-\frac{3}{2}z-5$ \\\hline
$s_3  s_2  s_3  s_1  s_2  s_3  s_2  s_1$ & $-m-n-k-x-y-z-6$ & $-m-n-2k-x-2y-2z-9$ & $-\frac{m}{2}-n-\frac{3}{2}k-\frac{x}{2}-y-\frac{3}{2}z-6$ \\\hline
$s_3  s_2  s_3  s_1  s_2  s_3  s_1  s_2$ & $-n-k-x-y-z-5$ & $-m-2n-2k-x-2y-2z-10$ & $-\frac{m}{2}-n-\frac{3}{2}k-\frac{x}{2}-y-\frac{3}{2}z-6$ \\\hline
$s_3  s_2  s_3  s_1  s_2  s_3  s_1  s_2  s_1$ & $-m-n-k-x-y-z-6$ & $-m-2n-2k-x-2y-2z-10$ & $-\frac{m}{2}-n-\frac{3}{2}k-\frac{x}{2}-y-\frac{3}{2}z-6$
\end{longtable}
\end{center}
}}

Since $m,n,k,x,y,z \in \N$, we now notice that many $\sigma \in W$ give at least one negative coefficient $C_{i,\sigma}$ in the expression $C_{1,\sigma}\alpha_1+C_{2,\sigma}\alpha_2+C_{3,\sigma}\alpha_3$ of Table~\ref{table:all-sigmas}. Namely, of the 48 elements of the Weyl group the following $31$ elements produce at least one coefficient $C_{i,\sigma}$ which is negative:
\begingroup
\allowdisplaybreaks
\begin{align}\label{eq:type I elements}
&s_1s_2s_3, s_1s_2s_3s_1, s_1s_2s_3s_2, s_2s_3s_2s_1, s_2s_3s_1s_2, s_3s_2s_3s_1, s_3s_1s_2s_3, s_1s_2s_3s_2s_1, s_1s_2s_3s_1s_2, \\ 
&s_2s_3s_1s_2s_1, s_2s_3s_1s_2s_3, s_3s_1s_2s_3s_1, s_3s_1s_2s_3s_2, s_3s_2s_3s_2s_1, s_3s_2s_3s_1s_2, s_1s_2s_3s_1s_2s_1, \nonumber\\ 
&s_2s_3s_1s_2s_3s_1, s_2s_3s_1s_2s_3s_2, s_3s_1s_2s_3s_1s_2, s_3s_1s_2s_3s_2s_1, s_3s_2s_3s_1s_2s_1, s_3s_2s_3s_1s_2s_3, \nonumber\\ 
&s_2s_3s_1s_2s_3s_2s_1, s_2s_3s_1s_2s_3s_1s_2, s_3s_2s_3s_1s_2s_3s_2, s_3s_1s_2s_3s_1s_2s_1, s_3s_2s_3s_1s_2s_3s_1, \nonumber\\ 
&s_2s_3s_1s_2s_3s_1s_2s_1, s_3s_2s_3s_1s_2s_3s_2s_1, s_3s_2s_3s_1s_2s_3s_1s_2, s_3s_2s_3s_1s_2s_3s_1s_2s_1.\nonumber
\end{align}
\endgroup

Thus, the above elements of $W$ cannot be in the Weyl alternation set of any dominant integral weights $\lambda=m\w_1+n\w_2+k\w_3$ and $\mu=x\w_1+y\w_2+z\w_3$, since, for these elements, $\wp(\sigma(\lambda+\rho)-\rho-\mu) = 0$ for any choice of parameters $m,n,k,x,y,z \in \mathbb{N}$.  Equivalently, any subset that contains any of the above elements of $W$ is a forbidden alternation set, because as we have established by definition of the Weyl alternation sets, single elements of the Weyl group cannot contribute nontrivially to the multiplicity formula. This is what we refer to as a \textit{Type I contradiction}. For ease of reference we define this formally.

\begin{definition}
Let $\sigma\in W$ and $\lambda$ and $\mu$ be weights of $\ccc$.
If $\sigma(\lambda+\rho)-\rho-\mu=C_{1,\sigma}\al{1}+C_{2,\sigma}\al{2}+C_{3,\sigma}\al{3}$ and $C_{i,\sigma}<0$ for some $1\leq i\leq 3$, then we say $\sigma$ induces a \emph{Type I contradiction}.
\end{definition}

We now remark that based on the Type I contradictions this leaves a possible $17$ Weyl group elements which can appear in Weyl alternation sets, but they must satisfy that $\wp(\sigma(\lambda+\rho)-\rho-\mu) > 0$.
For sake of clarity, in Table \ref{table:18-elements}, we recreate the content of Table \ref{table:all-sigmas} only presenting the remaining $17$ Weyl group elements which may appear in Weyl alternation sets.

\begin{center}
\begin{longtable}[c]{| p{0.6in} | p{1.5in} | p{1.7in} | p{1.6in} | }
\caption{The expressions $C_{1,\sigma}\alpha_1+C_{2,\sigma}\alpha_2+C_{3,\sigma}\alpha_3$ for contributing $\sigma \in W$.} \label{table:18-elements} \\
\hline
$\sigma$ & $C_{1, \sigma}$ & $C_{2, \sigma}$ & $C_{3, \sigma}$\\\hline
\hline
\endfirsthead
\hline
$\sigma$ & $C_{1, \sigma}$ & $C_{2, \sigma}$ & $C_{3, \sigma}$\\\hline
\hline
\endhead

\hline
\endfoot

$1$ & $m+n+k-x-y-z$ & $m+2n+2k-x-2y-2z$ & $\frac{m}{2}+n+\frac{3}{2}k-\frac{x}{2}-y-\frac{3}{2}z$\\\hline
$s_1$ & $n+k-x-y-z-1$ & $m+2n+2k-x-2y-2z$ & $\frac{m}{2}+n+\frac{3}{2}k-\frac{x}{2}-y-\frac{3}{2}z$ \\\hline
$s_2$ & $m+n+k-x-y-z$ & $m+n+2k-x-2y-2z-1$ & $\frac{m}{2}+n+\frac{3}{2}k-\frac{x}{2}-y-\frac{3}{2}z$ \\\hline
$s_3$ & $m+n+k-x-y-z$ & $m+2n+2k-x-2y-2z$ & $\frac{m}{2}+n+\frac{k}{2}-\frac{x}{2}-y-\frac{3}{2}z-1$ \\\hline
$s_1  s_2$ & $k-x-y-z-2$ & $m+n+2k-x-2y-2z-1$ & $\frac{m}{2}+n+\frac{3}{2}k-\frac{x}{2}-y-\frac{3}{2}z$ \\\hline
$s_2  s_1$ & $n+k-x-y-z-1$ & $n+2k-x-2y-2z-2$ & $\frac{m}{2}+n+\frac{3}{2}k-\frac{x}{2}-y-\frac{3}{2}z$ \\\hline
$s_2  s_3$ & $m+n+k-x-y-z$ & $m+n-x-2y-2z-3$ & $\frac{m}{2}+n+\frac{k}{2}-\frac{x}{2}-y-\frac{3}{2}z-1$ \\\hline
$s_3  s_1$ & $n+k-x-y-z-1$ & $m+2n+2k-x-2y-2z$ & $\frac{m}{2}+n+\frac{k}{2}-\frac{x}{2}-y-\frac{3}{2}z-1$ \\\hline
$s_3  s_2$ & $m+n+k-x-y-z$ & $m+n+2k-x-2y-2z-1$ & $\frac{m}{2}+\frac{k}{2}-\frac{x}{2}-y-\frac{3}{2}z-2$ \\\hline
$s_1  s_2  s_1$ & $k-x-y-z-2$ & $n+2k-x-2y-2z-2$ & $\frac{m}{2}+n+\frac{3}{2}k-\frac{x}{2}-y-\frac{3}{2}z$ \\\hline
$s_2  s_3  s_1$ & $n+k-x-y-z-1$ & $n-x-2y-2z-4$ & $\frac{m}{2}+n+\frac{k}{2}-\frac{x}{2}-y-\frac{3}{2}z-1$ \\\hline
$s_2  s_3  s_2$ & $m+n+k-x-y-z$ & $m-x-2y-2z-4$ & $\frac{m}{2}+\frac{k}{2}-\frac{x}{2}-y-\frac{3}{2}z-2$ \\\hline
$s_3  s_2  s_1$ & $n+k-x-y-z-1$ & $n+2k-x-2y-2z-2$ & $-\frac{m}{2}+\frac{k}{2}-\frac{x}{2}-y-\frac{3}{2}z-3$ \\\hline
$s_3  s_1  s_2$ & $k-x-y-z-2$ & $m+n+2k-x-2y-2z-1$ & $\frac{m}{2}+\frac{k}{2}-\frac{x}{2}-y-\frac{3}{2}z-2$ \\\hline
$s_3  s_2  s_3$ & $m+n+k-x-y-z$ & $m+n-x-2y-2z-3$ & $\frac{m}{2}-\frac{k}{2}-\frac{x}{2}-y-\frac{3}{2}z-3$ \\\hline
$s_3  s_1  s_2  s_1$ & $k-x-y-z-2$ & $n+2k-x-2y-2z-2$ & $-\frac{m}{2}+\frac{k}{2}-\frac{x}{2}-y-\frac{3}{2}z-3$ \\\hline
$s_3  s_2  s_3  s_2$ & $m+n+k-x-y-z$ & $m-x-2y-2z-4$ & $\frac{m}{2}-\frac{k}{2}-\frac{x}{2}-y-\frac{3}{2}z-3$
\end{longtable}
\end{center}

Type I contradictions allow us to reduce the number of possible Weyl alternation sets $\mathcal{A}(\lambda,\mu)$ from $2^{48}$ to $2^{17}$.
However, $2^{17} = 131072$ possible subsets is still quite large. To find yet more forbidden alternation sets, we use logic similar to that presented in~\cite{CGHLMRK}. The authors of the article find contradictions within certain subsets of the Weyl group in order to discard them as potential alternation sets. First, for the sake of simplicity, we present the following change of variables:
\begin{center}
\begingroup
\allowdisplaybreaks
\begin{align} 
    a &= m+n+k-x-y-z,  &&h = n-x-2y-2z-4, \nonumber\\
    b &= n+k-x-y-z-1,  &&i = m-x-2y-2z-4, \nonumber\\
    c &= k-x-y-z-2,  &&j = \frac{m}{2}+n+\frac{3}{2}k-\frac{x}{2}-y-\frac{3}{2}z, \nonumber\\
    d &= m+2n+2k-x-2y-2z,  &&l = \frac{m}{2}+n+\frac{k}{2}-\frac{x}{2}-y-\frac{3}{2}z-1, \label{eq:lowecase-eq}\\
    e &= m+n+2k-x-2y-2z-1,  &&o = \frac{m}{2}+\frac{k}{2}-\frac{x}{2}-y-\frac{3}{2}z-2, \nonumber\\
    f &= n+2k-x-2y-2z-2,  &&p = -\frac{m}{2}+\frac{k}{2}-\frac{x}{2}-y-\frac{3}{2}z-3, \nonumber\\
    g &= m+n-x-2y-2z-3,  &&r = \frac{m}{2}-\frac{k}{2}-\frac{x}{2}-y-\frac{3}{2}z-3.\nonumber
\end{align}
\endgroup
\end{center} 
Using the new variables along with the expressions in Table~\ref{table:18-elements}, we rewrite each term of the multiplicity formula as follows: 

\begingroup
\allowdisplaybreaks
\begin{align}\nonumber
    A &= \wp(1(\lambda + \rho)-(\mu + \rho)) = \wp(a \al{1} + d \al{2} + j \al{3}),  \\
    B &= \wp(s_1(\lambda + \rho)-(\mu + \rho)) = \wp(b \al{1} + d \al{2} + j \al{3}), \nonumber\\
    C &= \wp(s_2(\lambda + \rho)-(\mu + \rho)) = \wp(a \al{1} + e \al{2} + j \al{3}), \nonumber\\
    D &= \wp(s_3(\lambda + \rho)-(\mu + \rho)) = \wp(a \al{1} + d \al{2} + l \al{3}), \nonumber\\
    E &= \wp(s_1  s_2(\lambda + \rho)-(\mu + \rho)) = \wp(c \al{1} + e \al{2} + j \al{3}), \nonumber\\
    F &= \wp(s_2  s_1(\lambda + \rho)-(\mu + \rho)) = \wp(b \al{1} + f \al{2} + j \al{3}), \nonumber\\
    G &= \wp(s_2  s_3(\lambda + \rho)-(\mu + \rho)) = \wp(a \al{1} + g \al{2} + l \al{3}), \nonumber\\
    H &= \wp(s_3  s_1(\lambda + \rho)-(\mu + \rho)) = \wp(b \al{1} + d \al{2} + l \al{3}), \nonumber\\
    I &= \wp(s_3  s_2(\lambda + \rho)-(\mu + \rho)) = \wp(a \al{1} + e \al{2} + o\al{3}), \label{eq:so many equations}\\
    J &= \wp(s_1  s_2  s_1(\lambda + \rho)-(\mu + \rho)) = \wp(c \al{1} + f \al{2} + j \al{3}), \nonumber\\
    K &= \wp(s_2  s_3  s_1(\lambda + \rho)-(\mu + \rho)) = \wp(b \al{1} + h \al{2} + l \al{3}), \nonumber\\
    L &= \wp(s_2  s_3  s_2(\lambda + \rho)-(\mu + \rho)) = \wp(a \al{1} + i \al{2} + o\al{3}), \nonumber\\
    M &= \wp(s_3  s_2  s_1(\lambda + \rho)-(\mu + \rho)) = \wp(b \al{1} + f \al{2} + p \al{3}), \nonumber\\
    N &= \wp(s_3  s_1  s_2(\lambda + \rho)-(\mu + \rho)) = \wp(c \al{1} + e \al{2} + o\al{3}), \nonumber\\
    O &= \wp(s_3  s_2  s_3(\lambda + \rho)-(\mu + \rho)) = \wp(a \al{1} + g \al{2} + r \al{3}), \nonumber\\
    P &= \wp(s_3  s_1  s_2  s_1(\lambda + \rho)-(\mu + \rho)) = \wp(c \al{1} + f \al{2} + p \al{3}), \nonumber\\
    Q &= \wp(s_3  s_2  s_3  s_2(\lambda + \rho)-(\mu + \rho)) = \wp(a \al{1} + i \al{2} +r \al{3}). \nonumber
\end{align}
\endgroup

Note that there are still instances where certain values of $m,n,k,x,y,z \in \N$ result in some of the above expressions evaluating to zero. 
In this analysis we always begin by fixing an arbitrary set of values $m,n,k,x,y,z \in \N$, and then we say that the corresponding Weyl group element $\sigma$ contributes \emph{trivially} to $m(\lambda,\mu)$ whenever $\wp(\sigma(\lambda+\rho)-\rho-\mu)=0$. If instead $\wp(\sigma(\lambda+\rho)-\rho-\mu)>0$, we say that $\sigma$ contributes \emph{nontrivially} to $m(\lambda,\mu)$. As the contributions of these terms are always in reference to the multiplicity formula, we often omit the wording ``to $m(\lambda,\mu)$'' whenever we refer to an element contributing trivially or nontrivially. Moreover, 
we remark that in each of the expression listed in \eqref{eq:so many equations} we gather information in three ways: the value of the partition function denoted by a capital letter, the corresponding Weyl group element used in computing $\sigma(\lambda+\rho)-\rho-\mu$, and a triple of lower case letters denoting the coefficients of the simple roots in the expression $\sigma(\lambda+\rho)-\rho-\mu$.
Lastly, as a Weyl group element corresponds uniquely to an expression in \eqref{eq:so many equations}, we also say that the value of the partition function (denoted by capital letters) contributes nontrivially/trivially to denote the same phenomena/behavior.

In the work that follows, we show that of the remaining $2^{17}$ possible subsets of the Weyl group which could be Weyl alternation sets, only $46$ actually are such sets. To establish this, we define two new types of contradictions.  To begin this process, we let $\vee$ denote the Boolean operator `or' and $\wedge$ denote the Boolean operator `and'. Note that since the expressions $a$ through $r$ are always integral (which we establish in Lemma~\ref{lemma:divisible} below), our expressions $A$ through $Q$ contribute nontrivially only when $a$ through $r$ are nonnegative. To discuss this further, we define the following statements:

\begin{multicols}{7}
\noindent
$$a_0: a \geq 0,$$
$$a_1: a < 0,$$
$$b_0: b \geq 0,$$
$$b_1: b < 0,$$
$$c_0: c \geq 0,$$
$$c_1: c < 0,$$
$$d_0: d \geq 0,$$
$$d_1: d < 0,$$
$$e_0: e \geq 0,$$
$$e_1: e < 0,$$
$$f_0: f \geq 0,$$
$$f_1: f < 0,$$
$$g_0: g \geq 0,$$
$$g_1: g < 0,$$
$$h_0: h \geq 0,$$
$$h_1: h < 0,$$
$$i_0: i \geq 0,$$
$$i_1: i < 0,$$
$$j_0: j \geq 0,$$
$$j_1: j < 0,$$
$$l_0: l \geq 0,$$
$$l_1: l < 0,$$
$$o_0: o \geq 0,$$
$$o_1: o < 0,$$
$$p_0: p \geq 0,$$
$$p_1: p < 0,$$
$$r_0: r \geq 0,$$
$$r_1: r < 0.$$
\end{multicols}

Having defined these expressions, we now provide a necessary condition which characterizes when a weight on the fundamental weight lattice $\mathbb{Z}\w_1\oplus\mathbb{Z}\w_2\oplus\mathbb{Z}\w_3$ is also in the root lattice $\mathbb{Z}\alpha_1\oplus\mathbb{Z}\alpha_2\oplus\mathbb{Z}\alpha_3$. 

\begin{lemma}\label{lemma:divisible}
Let $\lambda = m\w_1+n\w_2+k\w_3$ and $\mu = x\w_1+y\w_2+z\w_3$ with $m,n,k,x,y,z\in\Z$ and let $\sigma\in W$. Then, $m+k+x+z$ is divisible by 2 if and only if $\sigma(\lambda)-\mu\in\Z\al{1}\oplus \Z\al{2}\oplus \Z\al{3}$.
\end{lemma}

\begin{proof}
Let $\lambda$ and $\mu$ be defined as above and let $\sigma\in W$.
Assume $m+k+x+z$ is divisible by two. Then, by inspection of Table~\ref{table:18-elements}, we observe that $\sigma(\lambda+\rho)-(\mu+\rho)\in\ZZ\al{1}\oplus\ZZ\al{2}\oplus\ZZ\al{3}.$ Since $\sigma$ is linear, we then have that $\sigma(\lambda+\rho)-(\mu+\rho)=\sigma(\lambda)-\mu+\left(\sigma(\rho)-\rho\right).$ Then, we can equivalently express $\rho=\w_1+\w_2+\w_3$ and, again, since  $\sigma$'s action on $\rho$ is given by subtracting a nonnegative integer linear combination of the simple roots (see \eqref{eq:simples acting on fun}), we know that  $\sigma(\rho)-\rho\in\ZZ\al{1}\oplus\ZZ\al{2}\oplus\ZZ\al{3}$. Thus
$\sigma(\lambda)-\mu\in \ZZ\al{1}\oplus\ZZ\al{2}\oplus\ZZ\al{3}$.

To prove the opposite direction, we illustrate the approach by considering the case where $\sigma=s_3s_2s_3s_2$, as all other cases are analogous. Assume that $\sigma(\lambda)-\mu\in\ZZ\al{1}\oplus\ZZ\al{2}\oplus\ZZ\al{3}$. A direct computation shows
\begin{align}\label{eq:sigma lambda -mu}
    \sigma(\lambda)-\mu&= \left(m+n+k-x-y-z\right)\al{1}+\left(m-x-2y-2z\right)\al{2}+\left(\frac{m}{2}-\frac{k}{2}-\frac{x}{2}-y-\frac{3z}{2}\right)\al{3}.
\end{align}
Note that the coefficients of $\al{1}$ and $\al{2}$ in \eqref{eq:sigma lambda -mu} are both integers since by assumption $m,n,k,x,y,z \in \ZZ$. We can rewrite the  coefficient of $\al{3}$  as $$\frac{m-k-x-z}{2}-(y+z)\in\mathbb{Z}.$$ Note that $y+z\in\ZZ$ since $y,z\in\ZZ$ by assumption. This then implies that $\frac{m-k-x-z}{2}\in\mathbb{Z}.$ Hence, $m-k-x-z$ is divisible by $2$. Further, as $m+k+x+z=m-k-x-z +2(k+x+z)$  and both $m-k-x-z$ and $2(k+x+z)$ are divisible by 2, so is $m+k+x+z$, as claimed. 
\end{proof}

\noindent The following key result follows from Lemma~\ref{lemma:divisible} allows us to identify the remaining forbidden alternation sets.
\begin{lemma}\label{lemma:NT}
Let $\lambda=m\w_1+n\w_2+k\w_3$ and $\mu=x\w_1+y\w_2+z\w_3$ with $m,n,k,x,y,z\in\mathbb{Z}$and $\frac{m+k+x+z}{2} \in \Z$ and let the expressions $A$ through $Q$ be defined as in \eqref{eq:so many equations}. Then, we have the following biconditional statements:
\begin{center}
$A$ contributes nontrivially if and only if $a_0 \wedge d_0 \wedge j_0$ holds true, \\
$B$ contributes nontrivially if and only if $b_0 \wedge d_0 \wedge j_0$ holds true, \\
$C$ contributes nontrivially if and  only if $a_0 \wedge e_0 \wedge j_0$ holds true, \\
$D$ contributes nontrivially if and  only if $a_0 \wedge d_0 \wedge l_0$ holds true, \\
$E$ contributes nontrivially if and  only if $c_0 \wedge e_0 \wedge j_0$ holds true, \\
$F$ contributes nontrivially if and  only if $b_0 \wedge f_0 \wedge j_0$ holds true, \\
$G$ contributes nontrivially if and  only if $a_0 \wedge g_0 \wedge l_0$ holds true, \\
$H$ contributes nontrivially if and  only if $b_0 \wedge d_0 \wedge l_0$ holds true, \\
$I$ contributes nontrivially if and  only if $a_0 \wedge e_0 \wedge o_0$ holds true, \\
$J$ contributes nontrivially if and  only if $c_0 \wedge f_0 \wedge j_0$ holds true, \\
$K$ contributes nontrivially if and  only if $b_0 \wedge h_0 \wedge l_0$ holds true, \\
$L$ contributes nontrivially if and  only if $a_0 \wedge i_0 \wedge o_0$ holds true, \\
$M$ contributes nontrivially if and  only if $b_0 \wedge f_0 \wedge p_0$ holds true, \\
$N$ contributes nontrivially if and  only if $c_0 \wedge e_0 \wedge o_0$ holds true, \\
$O$ contributes nontrivially if and  only if $a_0 \wedge g_0 \wedge r_0$ holds true, \\
$P$ contributes nontrivially if and  only if $c_0 \wedge f_0 \wedge p_0$ holds true, \\
$Q$ contributes nontrivially if and  only if $a_0 \wedge i_0 \wedge r_0$ holds true.
\end{center}
\end{lemma}

\begin{proof}
Since $\lambda,\mu\in \Z\al{1}\oplus\Z\al{2}\oplus\Z\al{3}$ and $\frac{m+k+x+z}{2} \in \Z$, then by Lemma~\ref{lemma:divisible} we know the expressions $a$ through $o$ are integral. Thus, a term in the multiplicity formula contributes nontrivially if and only if its corresponding coefficients are nonnegative.
\end{proof}

\noindent Negating the biconditional statements of Lemma~\ref{lemma:NT}, we arrive at the following.

\begin{lemma}Let the expressions $A$ through $Q$ be defined as above. Then, we have the following biconditional statements:

\begin{center}
$A$ contributes trivially if and only if  $a_1 \vee d_1 \vee j_1$ holds true, \\
$B$ contributes trivially if and only if  $b_1 \vee d_1 \vee j_1$ holds true, \\
$C$ contributes trivially if and only if  $a_1 \vee e_1 \vee j_1$ holds true, \\
$D$ contributes trivially if and only if  $a_1 \vee d_1 \vee l_1$ holds true, \\
$E$ contributes trivially if and only if  $c_1 \vee e_1 \vee j_1$ holds true, \\
$F$ contributes trivially if and only if  $b_1 \vee f_1 \vee j_1$ holds true, \\
$G$ contributes trivially if and only if  $a_1 \vee g_1 \vee l_1$ holds true, \\
$H$ contributes trivially if and only if  $b_1 \vee d_1 \vee l_1$ holds true, \\
$I$ contributes trivially if and only if  $a_1 \vee e_1 \vee o_1$ holds true, \\
$J$ contributes trivially if and only if  $c_1 \vee f_1 \vee j_1$ holds true, \\
$K$ contributes trivially if and only if  $b_1 \vee h_1 \vee l_1$ holds true, \\
$L$ contributes trivially if and only if  $a_1 \vee i_1 \vee o_1$ holds true, \\
$M$ contributes trivially if and only if  $b_1 \vee f_1 \vee p_1$ holds true, \\
$N$ contributes trivially if and only if  $c_1 \vee e_1 \vee o_1$ holds true, \\
$O$ contributes trivially if and only if  $a_1 \vee g_1 \vee r_1$ holds true, \\
$P$ contributes trivially if and only if  $c_1 \vee f_1 \vee p_1$ holds true, \\
$Q$ contributes trivially if and only if  $a_1 \vee i_1 \vee r_1$ holds true.
\end{center}
\end{lemma}

Next, we present an example to illustrate our use of contradictions to logical statements in order to further reduce the number of possible alternation sets. 

\begin{example}\label{ex:type ii contradiction}
Assume $m(\lambda,\mu) = A + F$; this implies that $A$ and $F$ contribute nontrivially while $B$, $C$, $D$, $E$, $G$, $H$, $I$, $J$, $K$, $L$, $M$, $N$, $O$, $P$, and $Q$ contribute trivially. Then, the following logical statement must be true:

\begin{center}
    \((a_0 \wedge d_0 \wedge j_0) \wedge (b_1 \vee d_1 \vee j_1) \wedge (a_1 \vee e_1 \vee j_1) \wedge (a_1 \vee d_1 \vee l_1) \wedge (c_1 \vee e_1 \vee j_1) \wedge (b_0 \wedge f_0 \wedge j_0) \wedge\)\\
    \( (a_1 \vee g_1 \vee l_1) \wedge (b_1 \vee e_1 \vee l_1) \wedge (a_1 \vee e_1 \vee o_1) \wedge (c_1 \vee f_1 \vee j_1) \wedge(b_1 \vee h_1 \vee l_1) \wedge (a_1 \vee i_1 \vee o_1) \wedge\)\\
    \( (b_1 \vee f_1 \vee p_1) \wedge (c_1 \vee e_1 \vee o_1) \wedge(a_1 \vee g_1 \vee r_1) \wedge(c_1 \vee f_1 \vee n_1) \wedge (a_1 \vee i_1 \vee r_1).\)\\
\end{center}

\noindent Notice, however, that this logical statement contains $(b_0 \wedge d_0 \wedge j_0) \wedge (b_1 \vee d_1 \vee j_1)$ which can never be true.
Thus $m(\lambda,\mu) \neq A+F$ for any $m,n,k,x,y,z \in \N$ and, therefore,  $\{1,s_2s_1\}$ is a forbidden Weyl alternation set whenever $\lambda$ and $\mu$ are dominant integral weights. 
\end{example}

Next, we provide a formal definition of the type of contradiction we encountered in Example~\ref{ex:type ii contradiction}.

\begin{definition}\label{def:type ii contradictions}
Given logical statements $X,Y,Z$ and their negations $\lnot X, \lnot Y, \lnot Z$, respectively, the logical statement 
\[(X\wedge Y\wedge Z)\wedge (\lnot X \vee \lnot Y\vee \lnot Z)\] is what we refer to as a \emph{Type II contradiction}.

\end{definition}

A short nonexhaustive list of Type II contradictions that arise in finding forbidden alternation sets is provided in Table~\ref{table:contradictions}. These Type II contradictions allow us to identify most of the forbidden Weyl alternation sets. 

\begin{center}
\begin{longtable}[c]{| p{1.6in} | p{1.4in} | p{2in} | }
\caption{Examples of Type II contradictions.} \label{table:contradictions}\\
\hline
Contribute nontrivially & Contribute trivially & Contradiction\\ \hline
\hline
\endfirsthead
\hline
Contribute nontrivially & Contribute trivially & Contradiction \\ \hline
\hline
\endhead

\hline
\endfoot

$A,E$ & $C$  & $(a_0 \wedge e_0 \wedge k_0) \wedge (a_1 \vee e_1 \vee k_1)$  \\
\hline
$A,G$ & $D$  & $(a_0 \wedge e_0 \wedge l_0) \wedge (a_1 \vee d_1 \vee l_1)$  \\
\hline
$A,H$ & $D$  & $(a_0 \wedge d_0 \wedge l_0) \wedge (a_1 \vee d_1 \vee l_1)$  \\
 & $B$ & $(b_0 \wedge d_0 \wedge k_0) \wedge (b_1 \vee d_1 \vee k_1)$ \\
\hline
$A,I$ & $C$  & $(a_0 \wedge e_0 \wedge k_0) \wedge (a_1 \vee e_1 \vee k_1)$  \\
\hline
$A,K$ & $D$  & $(a_0 \wedge d_0 \wedge l_0) \wedge (a_1 \vee d_1 \vee l_1)$  \\
 & $B$ & $(b_0 \wedge e_0 \wedge k_0) \wedge (b_1 \vee d_1 \vee k_1)$ \\
 & $H$ & $(b_0 \wedge d_0 \wedge l_0) \wedge (b_1 \vee d_1 \vee l_1)$\\
\hline
$A,M$ & $B$ & $(b_0 \wedge d_0 \wedge k_0) \wedge (b_1 \vee d_1 \vee k_1)$ \\
 & $F$ & $(b_0 \wedge f_0 \wedge k_0) \wedge (b_1 \vee f_1 \vee k_1)$ \\
\hline
$A,N$ & $I$ & $(a_0 \wedge e_0 \wedge m_0) \wedge (a_1 \vee e_1 \vee m_1)$ \\
 & $C$ & $(a_0 \wedge e_0 \wedge k_0) \wedge (a_1 \vee e_1 \vee k_1)$ \\
 & $E$ & $(c_0 \wedge e_0 \wedge k_0) \wedge (c_1 \vee e_1 \vee k_1)$ \\
\hline
$A,P$ & $J$ & $(c_0 \wedge f_0 \wedge k_0) \wedge (c_1 \vee f_1 \vee k_1)$ \\
\hline
$B,C$ & $A$ & $(a_0 \wedge d_0 \wedge k_0) \wedge (a_1 \vee d_1 \vee k_1)$ \\
\hline
$B,D$ & $A$ & $(a_0 \wedge d_0 \wedge k_0) \wedge (a_1 \vee d_1 \vee k_1)$ \\
 & $H$ & $(b_0 \wedge d_0 \wedge l_0) \wedge (b_1 \vee d_1 \vee l_1)$ \\
\hline
$B,G$ & $H$ & $(b_0 \wedge d_0 \wedge l_0) \wedge (b_1 \vee d_1 \vee l_1)$ \\
 & $A$ & $(a_0 \wedge d_0 \wedge k_0) \wedge (a_1 \vee d_1 \vee k_1)$ \\
 & $D$ & $(a_0 \wedge d_0 \wedge l_0) \wedge (a_1 \vee d_1 \vee l_1)$ \\
\hline
$A,E$ & $C$  & $(a_0 \wedge f_0 \wedge k_0) \wedge (a_1 \vee f_1 \vee k_1)$  \\
\hline
$A,G$ & $D$  & $(a_0 \wedge e_0 \wedge l_0) \wedge (a_1 \vee e_1 \vee l_1)$  \\
\hline
$A,H$ & $D$  & $(a_0 \wedge e_0 \wedge l_0) \wedge (a_1 \vee e_1 \vee l_1)$  \\
 & $B$ & $(b_0 \wedge e_0 \wedge k_0) \wedge (b_1 \vee e_1 \vee k_1)$ \\
\hline
$A,I$ & $C$  & $(a_0 \wedge f_0 \wedge k_0) \wedge (a_1 \vee f_1 \vee k_1)$  \\
\hline
$A,K$ & $D$  & $(a_0 \wedge e_0 \wedge l_0) \wedge (a_1 \vee e_1 \vee l_1)$  \\
 & $B$ & $(b_0 \wedge e_0 \wedge k_0) \wedge (b_1 \vee e_1 \vee k_1)$ \\
 & $H$ & $(b_0 \wedge e_0 \wedge l_0) \wedge (b_1 \vee e_1 \vee l_1)$\\
\hline
$A,M$ & $B$ & $(b_0 \wedge e_0 \wedge k_0) \wedge (b_1 \vee e_1 \vee k_1)$ \\
 & $F$ & $(b_0 \wedge g_0 \wedge k_0) \wedge (b_1 \vee g_1 \vee k_1)$ \\
\hline
$A,N$ & $I$ & $(a_0 \wedge f_0 \wedge m_0) \wedge (a_1 \vee f_1 \vee m_1)$ \\
 & $C$ & $(a_0 \wedge f_0 \wedge k_0) \wedge (a_1 \vee f_1 \vee k_1)$ \\
 & $E$ & $(c_0 \wedge f_0 \wedge k_0) \wedge (c_1 \vee f_1 \vee k_1)$ \\
\hline
$A,P$ & $J$ & $(c_0 \wedge g_0 \wedge k_0) \wedge (c_1 \vee g_1 \vee k_1)$ \\
\hline
$B,C$ & $A$ & $(a_0 \wedge e_0 \wedge k_0) \wedge (a_1 \vee e_1 \vee k_1)$ \\
\hline
$B,D$ & $A$ & $(a_0 \wedge e_0 \wedge k_0) \wedge (a_1 \vee e_1 \vee k_1)$ \\
 & $H$ & $(b_0 \wedge e_0 \wedge l_0) \wedge (b_1 \vee e_1 \vee l_1)$ \\
\hline
$B,G$ & $H$ & $(b_0 \wedge e_0 \wedge l_0) \wedge (b_1 \vee e_1 \vee l_1)$ \\
 & $A$ & $(a_0 \wedge e_0 \wedge k_0) \wedge (a_1 \vee e_1 \vee k_1)$ \\
 & $D$ & $(a_0 \wedge e_0 \wedge l_0) \wedge (a_1 \vee e_1 \vee l_1)$ \\
\hline
\end{longtable}
\end{center}

Python code which identifies Type II contradictions arising in all but 1124 of the $2^{17}$ remaining possible sets is provided in  Appendix~\ref{sec:contradiction-code-1}. Although the code is not written for efficiency purposes, it does identify the $1124$ sets which remain as possible alternation sets; these are provided in Appendix~\ref{sec:alt-sets-I}. 
    
We remark that Type I contradictions allowed us determine which elements of the Weyl group never appear in an alternation set. On the other hand, Type II contradictions did not eliminate any element of the Weyl group from being in an alternation set altogether. Instead, Type II contradictions allowed us to consider cases when two elements are in the same alternation set, while a third is not. From these assumptions we were able to establish some contradictions and reduce down to 1124 remaining subsets of the Weyl group which might arise as Weyl alternation sets. 
    
Next, we present the last set of contradictions, where we consider certain elements of the Weyl group appearing in the same alternation set. In those cases, we show that the existence of certain elements lying within the same alternation set generates new contradictions in the defining logical statements.   
    
The Type III contradictions are listed in Theorem~\ref{thm:contradictions}. Although the given list is not exhaustive of all remaining possible contradictions, the list is sufficient to identify all remaining forbidden alternation sets.

\begin{theorem}\label{thm:contradictions} For $m,n,k,x,y,z \in \N$, and $a_0, a_1, \ldots, r_0, r_1$ as defined previously, the following conditions can never hold true:
$(a_1 \wedge b_0)$,
$(a_1 \wedge c_0)$,
$(a_1\wedge g_0)$,
$(a_1\wedge h_0)$,
$(b_1 \wedge c_0)$,
$(b_1 \wedge h_0)$,
$(b_1 \wedge p_0)$,
$(c_1 \wedge p_0)$,
$(d_1 \wedge b_0)$,
$(d_1 \wedge c_0)$,
$(d_1 \wedge e_0)$,
$(d_1 \wedge f_0)$,
$(d_1 \wedge g_0)$,
$(d_1 \wedge h_0)$,
$(d_1 \wedge i_0)$,
$(d_1 \wedge l_0)$,
$(d_1 \wedge o_0)$,
$(d_1 \wedge p_0)$,
$(d_1 \wedge r_0)$,
$(e_1 \wedge c_0)$,
$(e_1 \wedge f_0)$,
$(e_1 \wedge g_0)$,
$(e_1 \wedge h_0)$,
$(e_1 \wedge i_0)$,
$(e_1 \wedge o_0)$,
$(e_1 \wedge p_0)$,
$(e_1 \wedge r_0)$,
$(f_1 \wedge c_0)$,
$(f_1 \wedge h_0)$,
$(f_1 \wedge p_0)$,
$(g_1 \wedge h_0)$,
$(g_1 \wedge i_0)$,
$(g_1 \wedge r_0)$,
$(i_1 \wedge r_0)$,
$(j_1 \wedge c_0)$,
$(j_1 \wedge h_0)$,
$(j_1 \wedge l_0)$,
$(j_1 \wedge o_0)$,
$(j_1 \wedge p_0)$,
$(j_1 \wedge r_0)$,
$(l_1 \wedge h_0)$,
$(l_1 \wedge o_0)$,
$(l_1 \wedge p_0)$,
$(l_1 \wedge r_0)$,
$(o_1 \wedge p_0)$,
$(o_1 \wedge r_0)$,
$(p_0 \wedge r_0)$,
$(b_1 \wedge f_0 \wedge h_0)$,
$(j_1 \wedge a_0 \wedge f_0)$,
$(l_1 \wedge b_0 \wedge g_0)$,
$(l_1 \wedge b_0 \wedge i_0)$,
$(l_1 \wedge f_0 \wedge g_0)$,
and
$(o_1 \wedge c_0 \wedge i_0)$.
\end{theorem}

\begin{proof} We prove each of these contradictions by manipulating inequalities until we reach a contradiction to our assumptions. We treat each contradiction as a case and prove each one separately.

\begin{itemize}
\item[Case 1:] Assume $a_1 \wedge b_0$.
Then, $n+k-x-y-z-1 \geq 0$ and $m+n+k-x-y-z < 0$, which implies $n+k-y-z-1 \geq x$ and $m+n+k-y-z < x$. So, we can say that $n+k-y-z-1 > m+n+k-y-z$, and we can simplify that statement to $-1 > m$, but we assumed $m \in \mathbb{N}$.\\

\item[Case 2:]Assume $a_1 \wedge c_0$. Then, $m+n+k-x-y-z<0$ and $k-x-y-z-2 \geq 0$, which implies $m+n+k-y-z<x$ and $k-y-z-2 \geq x$. So, we can say, $m+n+k-y-z < k-y-z-2$, and we can simplify that statement to $m+n < -2$, but we assumed $m,n \in \mathbb{N}$.\\
 
\item[Case 3:]Assume $a_1 \wedge g_0$. Then, $-m-n-k+x+y+z>0$, and $m+n-x-2y-2z-3 \geq 0$. Adding these inequalities, we get $-k-y-z-3>0$. This is equivalent to $-k-y-z>3$, but we assumed $k,y,z \in \mathbb{N}$.\\

\item[Case 4:]Assume $a_1 \wedge h_0$. Then, $-m-n-k+x+y+z>0$, and $n-x-2y-2z-4 \geq 0$. Adding these inequalities, we get $-m-k-y-z-4>0$. This is equivalent to $-m-k-y-z>4$, but we assumed $m,k,y,z \in \mathbb{N}$.\\
 
\item[Case 5:]Assume $b_1 \wedge c_0$. Then, $n+k-x-y-z-1<0$ and $k-x-y-z-2 \geq 0$, which implies $n+k-y-z-1<x$ and $k-y-z-2 \geq x$. So we can say that $n+k-y-z-1 < k-y-z-2$, and we can simplify that statement to $n < -1$, but we assumed $n \in \mathbb{N}$.\\

\item[Case 6:]Assume $b_1 \wedge h_0$. Then, $-n-k+x+y+z+1>0$, and $n-x-2y-2z-4 \geq 0$. Adding these inequalities, we get $-k-y-z-3>0$. This is equivalent to $-k-y-z>3$, but we assumed $k,y,z \in \mathbb{N}$.\\

\item[Case 7:]Assume $b_1 \wedge p_0$. Then, $-n-k+x+y+z+1>0$, and $-m+k-x-2y-3z-6 \geq 0$. Adding these inequalities, we get $-m-n-y-2z-5>0$. This is equivalent to $-m-n-y-2z>5$, but we assumed $m,n,y,z \in \mathbb{N}$.\\

\item[Case 8:]Assume $c_1 \wedge p_0$. Then, $-k + x + y + z + 2 > 0$ and $-m + k - x - 2y - 3z - 6 \geq 0$. Adding these inequalities, we get $-m - y - 2z - 4 > 0$, but we assumed $m, y, z \in \mathbb{N}$.\\

\item[Case 9:]Assume $d_1 \wedge b_0$. Then, $-m-2n-2k+x+2y+2z>0$, and $2(n+k-x-y-z-1) \geq 0$. Adding these inequalities, we get $-m-x-2>0$, but we assumed $m,x \in \mathbb{N}$.\\

\item[Case 10:]Assume $d_1 \wedge c_0$. Then, $-m-2n-2k+x+2y+2z>0$, and $2(k-x-y-z-2) \geq 0$. Adding these inequalities, we get $-m-2n-x-4>0$, but we assumed $m,n,x \in \mathbb{N}$.\\

\item[Case 11:]Assume $d_1 \wedge e_0$. Then, $m+2n+2k-x-2y-2z < 0$ and $m+n+2k-x-2y-2z-1 \geq 0$, which implies $m+2n+2k-2y-2z < x$ and $m+n+2k-2y-2z-1 \geq x$. So, we can say, $m+2n+2k-2y-2z < m+n+2k-2y-2z-1$, and we can simplify that statement to $n < -1$, but we assumed $n \in \mathbb{N}$.\\

\item[Case 12:]Assume $d_1 \wedge f_0$. Then, $m+2n+2k-x-2y-2z < 0$ and $n+2k-x-2y-2z-2 \geq 0$. So, we can say, $m+2n+2k-x-2y-2z < n+2k-x-2y-2z-2$, and we can simplify that statement to $m+n < -2$, but we assumed $m,n \in \mathbb{N}$.\\

\item[Case 13:]Assume $d_1 \wedge g_0$. Then, $m+2n+2k-x-2y-2z < 0$ and $m+n-x-2y-2z-3 \geq 0$. So, we can say, $m+2n+2k-x-2y-2z < m+n-x-2y-2z-3$, and we can simplify that statement to $n+2k < -3$, but we assumed $n,k \in \mathbb{N}$.\\

\item[Case 14:]Assume $d_1 \wedge h_0$. Then, $m+2n+2k-x-2y-2z < 0$ and $n-x-2y-2z-4 \geq 0$. So, we can say, $m+2n+2k-x-2y-2z < n-x-2y-2z-4$, and we can simplify that statement to $m+n+2k<-4$, but we assumed $m,n,k \in \mathbb{N}$.\\

\item[Case 15:]Assume $d_1 \wedge i_0$. Then, $m+2n+2k-x-2y-2z<0$ and $m-x-2y-2z-4 \geq 0$. So, we can say, $m+2n+2k-x-2y-2z < m-x-2y-2z-4$, and we can simplify that statement to $2n + 2k < -4$, but we assumed $n,k \in \mathbb{N}$.\\

\item[Case 16:]Assume $d_1 \wedge l_0$. Then, $-m-2n-2k+x+2y+2z>0$, and $m+2n+k-x-2y-3z-2 \geq 0$. Adding these inequalities, we get $-k-z-2>0$, but we assumed $k,z \in \mathbb{N}$.\\

\item[Case 17:]Assume $d_1 \wedge o_0$. Then, $-m-2n-2k+x+2y+2z>0$, and $m+k-x-2y-3z-4 \geq 0$. Adding these inequalities, we get $-2n-k-z-4>0$, but we assumed $n,k,z \in \mathbb{N}$.\\

\item[Case 18:]Assume $d_1 \wedge p_0$. Then, $-m-2n-2k+x+2y+2z>0$, and $-m+k-x-2y-3z-6 \geq 0$. Adding these inequalities, we get $-2m-2n-k-z-6>0$, but we assumed $m,n,k,z \in \mathbb{N}$.\\

\item[Case 19:]Assume $d_1 \wedge r_0$. Then, $-m -2n -2k + x + 2y + 2z>0$, and $m-k-x-2y-3z-6 \geq 0$. Adding these inequalities, we get $-2n-3k-z-6>0$, but we assumed $n,k,z \in \mathbb{N}$.\\

\item[Case 20:]Assume $e_1 \wedge c_0$. Then, $-m-n-2k+x+2y+2z+1>0$, and $2(k-x-y-z-2) \geq 0$. Adding these inequalities, we get $-m-n-x-3>0$, but we assumed $m,n,x \in \mathbb{N}$.\\

\item[Case 21:]Assume $e_1 \wedge f_0$. Then, $m+n+2k-x-2y-2z-1<0$ and $n+2k-x-2y-2z-2 \geq 0$, which implies $m+n+2k-2y-2z-1<x$ and $n+2k-2y-2z-2 \geq x$. So, we can say, $m+n+2k-2y-2z-1 < n+2k-2y-2z-2$, and we can simplify that statement to $m<-1$, but we assumed $m \in \mathbb{N}$.\\

\item[Case 22:]Assume $e_1 \wedge g_0$. Then, $m+n+2k-x-2y-2z-1<0$ and $m+n-x-2y-2z-3 \geq 0$, which implies $m+n+2k-2y-2z-1<x$ and $m+n-2y-2z-3 \geq x$. So, we can say, $m+n+2k-2y-2z-1 < m+n-2y-2z-3$, and we can simplify that statement to $k<-1$, but we assumed $k \in \mathbb{N}$.\\

\item[Case 23:]Assume $e_1 \wedge h_0$. Then, $m+n+2k-x-2y-2z-1 < 0$ and $n-x-2y-2z-4 \geq 0$. So, we can say, $m+n+2k-x-2y-2z-1 < n-x-2y-2z-4$, and we can simplify that statement to $m+2k < -3$, but we assumed $m,k \in \mathbb{N}$.\\

\item[Case 24:]Assume $e_1 \wedge i_0$. Then, $m+n+2k-x-2y-2z-1 < 0$ and $m-x-2y-2z-4 \geq 0$. So, we can say, $m+n+2k-x-2y-2z-1 < m-x-2y-2z-4$, and we can simplify that statement to $n+2k < -3$, but we assumed $n,k \in \mathbb{N}$.\\

\item[Case 25:]Assume $e_1 \wedge o_0$. Then, $-m-n-2k+x+2y+2z+1>0$, and $m+k-x-2y-3z-4 \geq 0$. Adding these inequalities, we get $-n-k-z-3>0$, but we assumed $n,k,z \in \mathbb{N}$.\\

\item[Case 26:]Assume $e_1 \wedge p_0$. Then, $-m-n-2k+x+2y+2z+1>0$, and $-m+k-x-2y-3z-6 \geq 0$. Adding these inequalities, we get $-2m-n-k-z-5>0$, but we assumed $m,n,k,z \in \mathbb{N}$.\\

\item[Case 27:]Assume $e_1 \wedge r_0$. Then, $-m-n-2k+x+2y+2z+1>0$, and $m-k-x-2y-3z-6 \geq 0$. Adding these inequalities, we get $-n-3k-z-5>0$, but we assumed $n,k,z \in \mathbb{N}$.\\

\item[Case 28:]Assume $f_1 \wedge c_0$. Then, $-n-2k+x+2y+2z+2>0$, and $2(k-x-y-z-2) \geq 0$. Adding these inequalities, we get $-n-x-2>0$, but we assumed $n,x \in \mathbb{N}$.\\

\item[Case 29:]Assume $f_1 \wedge h_0$. Then, $n+2k-x-2y-2z-2<0$ and $n-x-2y-2z-4 \geq 0$, which implies $n+2k-2y-2z-2<x$ and $n-2y-2z-4 \geq x$. So, we can say, $n+2k-2y-2z-2 < n-2y-2z-4$, and we can simplify that statement to $k<-1$, but we assumed $k \in \mathbb{N}$.\\

\item[Case 30:]Assume $f_1 \wedge p_0$. Then, $-n-2k+x+2y+2z+2>0$, and $-m+k-x-2y-3z-6 \geq 0$. Adding these inequalities, we get $-m-n-k-z-4>0$, but we assumed $m,n,k,z \in \mathbb{N}$.\\

\item[Case 31:]Assume $g_1 \wedge h_0$. Then, $m+n-x-2y-2z-3<0$ and $n-x-2y-2z-4 \geq 0$, which implies $m+n-2y-2z-3<x$ and $n-2y-2z-4 \geq x$. So, we can say, $m+n-2y-2z-3 < n-2y-2z-4$, and we can simplify that statement to $m<-1$, but we assumed $m \in \mathbb{N}$.\\

\item[Case 32:]Assume $g_1 \wedge i_0$. Then, $m+n-x-2y-2z-3<0$ and $m-x-2y-2z-4 \geq 0$, which implies $m+n-2y-2z-3<x$ and $m-2y-2z-4 \geq x$. So, we can say, $m+n-2y-2z-3 < m-2y-2z-4$, and we can simplify that statement to $n<-1$, but we assumed $n \in \mathbb{N}$.\\

\item[Case 33:]Assume $g_1 \wedge r_0$. Then, $-m-n+x+2y+2z+3>0$, and $m-k-x-2y-3z-6 \geq 0$. Adding these inequalities, we get $-n-k-z-3>0$, but we assumed $n,z \in \mathbb{N}$.\\ 

\item[Case 34:]Assume $i_1 \wedge r_0$. Then, $-m + x + 2y + 2z + 4 > 0$ and $m - k - x - 2y - 3z - 6 \geq 0$. Adding these inequalities gives $-k - z - 2 > 0$, but we assumed that $k, z \in \mathbb{N}$.\\

\item[Case 35:]Assume $j_1 \wedge c_0$. Then, $-m-2n-3k+x+2y+3z>0$, and $3(k-x-y-z-2) \geq 0$. Adding these inequalities, we get $-m-2n-2x-y-6>0$, but we assumed $m,n,x,y \in \mathbb{N}$.\\

\item[Case 36:]Assume $j_1 \wedge h_0$. Then, $-m-2n-3k+x+2y+3z>0$ and  $2(n-x-2y-2z-4)\geq 0$. Adding these inequalities, we get $-m-3k-x-2y-z-8>0$, but we assumed that $m,k,x,y,z \in \mathbb{N}$.\\

\item[Case 37:]Assume $j_1 \wedge l_0$. Then, $-\frac{m}{2}-n-\frac{3}{2}k+\frac{x}{2}+y+\frac{3}{2}z > 0$ and $\frac{m}{2}+n+\frac{k}{2}-\frac{x}{2}-y-\frac{3}{2}z - 1 \geq 0$. Adding these inequalities $-k-1>0,$ but we assumed that $k\in\N.$\\

\item[Case 38:]Assume $j_1 \wedge o_0$. Then, $\frac{m}{2}+n+\frac{3}{2}k-\frac{x}{2}-y-\frac{3}{2}z < 0$ and $\frac{m}{2}+\frac{k}{2}-\frac{x}{2}-y-\frac{3}{2}z-2 \geq 0$. So, we can say, $\frac{m}{2}+n+\frac{3}{2}k-\frac{x}{2}-y-\frac{3}{2}z < \frac{m}{2}+\frac{k}{2}-\frac{x}{2}-y-\frac{3}{2}z-2$, and we can simplify that statement to $n+k<-2$, but we assumed $n,k \in \mathbb{N}$.\\

\item[Case 39:]Assume $j_1 \wedge p_0$. Then, $\frac{m}{2}+n+\frac{3}{2}k-\frac{x}{2}-y-\frac{3}{2}z < 0$ and $-\frac{m}{2}+\frac{k}{2}-\frac{x}{2}-y-\frac{3}{2}z-3 \geq 0$. So, we can say, $\frac{m}{2}+n+\frac{3}{2}k-\frac{x}{2}-y-\frac{3}{2}z < -\frac{m}{2}+\frac{k}{2}-\frac{x}{2}-y-\frac{3}{2}z-3$, and we can simplify that statement to $m+n+k < -3$, but we assumed $m,n,k \in \mathbb{N}$.\\

\item[Case 40:]Assume $j_1 \wedge r_0$. Then, $-\frac{m}{2}-n-\frac{3}{2}k+\frac{x}{2}+y+\frac{3}{2}z > 0$ and $\frac{m}{2}-\frac{k}{2}-\frac{x}{2}-y-\frac{3}{2}z-3 \geq 0$. Adding these inequalities gives $n+k<-6,$ but we assumed $n,k \in \mathbb{N}$.\\

\item[Case 41:]Assume $l_1 \wedge h_0$. Then, $- m - 2n - k + x + 2y + 3z + 2>0$ and $ 2(n-x-2y-2z-4) \geq 0$. By adding these inequalities, we get $-m-k-x-2y-z-6> 0.$ This is equivalent to $-m-k-x-2y-z> 6,$ but we assumed $m,k,x,y,z \in \mathbb{N}$.\\

\item[Case 42:]Assume $l_1 \wedge o_0$. Then, $\frac{m}{2}+n+\frac{k}{2}-\frac{x}{2}-y-\frac{3}{2}z-1<0$ and $\frac{m}{2}+\frac{k}{2}-\frac{x}{2}-y-\frac{3}{2}z-2 \geq 0$, which implies $\frac{m}{2}+n+\frac{k}{2}-\frac{x}{2}-\frac{3}{2}z-1<y$ and $\frac{m}{2}+\frac{k}{2}-\frac{x}{2}-\frac{3}{2}z-2 \geq y$. So, we can say, $\frac{m}{2}+n+\frac{k}{2}-\frac{x}{2}-\frac{3}{2}z-1 < \frac{m}{2}+\frac{k}{2}-\frac{x}{2}-\frac{3}{2}z-2$, and we can simplify that statement to $n < -1$, but we assumed $n \in \mathbb{N}$.\\

\item[Case 43:]Assume $l_1 \wedge p_0$. Then, $- m - 2n - k + x + 2y + 3z + 2 > 0 $ and $- m + k -x -2y -3z -6\geq 0$. By adding these inequalities, we get $-2m-2n-4>0$. This is equivalent to $-2m-2n> 4,$ but we assumed $m,n \in \mathbb{N}$.\\

\item[Case 44:]Assume $l_1 \wedge r_0$. Then, $\frac{m}{2}+n+\frac{k}{2}-\frac{x}{2}-y-\frac{3}{2}z-1 < 0$ and $\frac{m}{2}-\frac{k}{2}-\frac{x}{2}-y-\frac{3}{2}z-3 \geq 0$, which implies $\frac{m}{2}+n+\frac{k}{2}-\frac{x}{2}-\frac{3}{2}z-1 < y$ and $\frac{m}{2}-\frac{k}{2}-\frac{x}{2}-\frac{3}{2}z-3 \geq y$. So, we can say, $\frac{m}{2}+n+\frac{k}{2}-\frac{x}{2}-\frac{3}{2}z-1 < \frac{m}{2}-\frac{k}{2}-\frac{x}{2}-\frac{3}{2}z-3$, and we can simplify that statement to $n+k < -2$, but we assumed $n,k \in \mathbb{N}$. \\

\item[Case 45:]Assume $o_1 \wedge p_0$. Then, $\frac{m}{2}+\frac{k}{2}-\frac{x}{2}-y-\frac{3}{2}z-2 < 0$ and $-\frac{m}{2}+\frac{k}{2}-\frac{x}{2}-y-\frac{3}{2}z-3 \geq 0$, which implies $\frac{m}{2}+\frac{k}{2}-\frac{x}{2}-\frac{3}{2}z-2 < y$ and $-\frac{m}{2}+\frac{k}{2}-\frac{x}{2}-\frac{3}{2}z-3 \geq y$. So, we can say, $\frac{m}{2}+\frac{k}{2}-\frac{x}{2}-\frac{3}{2}z-2 < -\frac{m}{2}+\frac{k}{2}-\frac{x}{2}-\frac{3}{2}z-3$, and we can simplify that statement to $m < -1$, but we assumed $m \in \mathbb{N}$.\\

\item[Case 46:]Assume $o_1 \wedge r_0$. Then, $\frac{m}{2}+\frac{k}{2}-\frac{x}{2}-y-\frac{3}{2}z-2 < 0$ and $\frac{m}{2}-\frac{k}{2}-\frac{x}{2}-y-\frac{3}{2}z-3 \geq 0$, which implies $\frac{m}{2}+\frac{k}{2}-\frac{x}{2}-\frac{3}{2}z-2 < y$ and $\frac{m}{2}-\frac{k}{2}-\frac{x}{2}-\frac{3}{2}z-3 \geq y$. So, we can say, $\frac{m}{2}+\frac{k}{2}-\frac{x}{2}-\frac{3}{2}z-2 < \frac{m}{2}-\frac{k}{2}-\frac{x}{2}-\frac{3}{2}z-3$, and we can simplify that statement to $k<-1$, but we assumed $k \in \mathbb{N}$.\\

\item[Case 47:]Assume $p_0 \wedge r_0$. Then, $-\frac{m}{2} + \frac{k}{2} - \frac{x}{2} - y - \frac{3}{2}z - 3 \geq 0$ and $\frac{m}{2} - \frac{k}{2} - \frac{x}{2} - y - \frac{3}{2}z - 3 \geq 0$. Adding these inequalities gives $-x - 2y - 3z - 3 \geq 0$, but we assumed that $x, y, z \in \mathbb{N}$.\\

\item[Case 48:]Assume $b_1 \wedge f_0 \wedge h_0$. Then, $2(-n -k + x +y +z +1)>0$, $n+2k-x-2y-2z-2\geq0$, and $n-x-2y-2z-4\geq 0.$ Adding these inequalities, we get $-2y-2z-4 > 0$. This is equivalent to $-2y-2z>4,$ but we assumed that $y,z\in \N$.\\

\item[Case 49:]Assume $j_1 \wedge a_0 \wedge f_0$. Then, $-n - 2k + x + 2y + 2z + 2 \leq 0$. Also, $m + 2n + 3k - x - 2y - 3z < 0$. By adding these inequalities, we get $m + n + k - z + 2 < 0$, or more simply, $m + n + k < z$. However, from $a_0$ we get that $m + n + k - x - y - z > 0$, implying $m + n + k \geq z$, a contradiction.\\

\item[Case 50:]Assume $l_1 \wedge b_0 \wedge g_0$. Then, $m + n - x - 2y - 2z - 3 \geq 0$ and $-m - 2n - k + x + 2y + 3z + 1 > 0$. Adding these inequalities, we get $-n-k+z-2 > 0$, or more simply, $z > n + k + 2$. However, from $b_0$ we get that $n + k - x - y - z - 1 \geq 0$, implying $n + k \geq x + y + z + 1 > z$. This implies $z > n + k + 2 > z$.\\

\item[Case 51:]Assume $l_1 \wedge b_0 \wedge i_0$. Then, $- m - 2n - k + x + 2y + 3z + 2 > 0$, $ n + k - x - y - z - 1\geq 0$, and $m - x -  2y - 2z -4\geq 0.$ By adding these inequalities, we get $-n-x-y-3 > 0,$ but we assumed that $n,x,y\in \N$.\\

\item[Case 52:]Assume $l_1 \wedge f_0 \wedge g_0$. Then, $-n - 2k + x + 2y + 2z + 2 \leq 0$. Also, $m + 2n + 3k - x - 2y - 3z < 0$. Adding these inequalities, we get $m + n + k - z + 2 < 0$, or more simply, $m + n + k < z$. However, from $a_0$ we get that $m + n + k - x - y - z > 0$, implying $m + n + k \geq z$.\\

\item[Case 53:]Assume $o_1 \wedge c_0 \wedge i_0$. Then, $m - x  - 2y - 2z - 4 \geq 0$ and $-m-k+x+2y+3z+4 > 0$. Adding these inequalities, we get $z > k$. However, from $c_0$, we get that $k \geq x + y + z + 2 \geq z$. This implies $z > k \geq z$.\qedhere
\end{itemize}
\end{proof}

We provide Python code which checks each of our $1124$ possible remaining subsets and discards any that has any contradiction of Type III, as delineated in Theorem~\ref{thm:contradictions}.  This code is included in Appendix~\ref{sec:contradiction-code-2}.
Taking these new contradictions into account, we are able to reduce down to only $150$ subsets of the Weyl group that can be alternation sets for some particular dominant integral weights $\lambda$ and $\mu$. These are listed in Appendix~\ref{sec:alt-sets-II}. 

Note, however, that contradictions of Type II and Type III do not exist in isolation. Take for example the case where $m(\lambda,\mu) = -D$, that is, the case where $D$ contributes nontrivially while $A$, $B$, $C$, $E$, $F$, $G$, $H$, $I$, $J$, $K$, $L$, $M$, $N$, $O$, $P$, and $Q$ contribute trivially. Then, the following must necessarily be true:

\begin{center}
    \((a_1 \vee d_1 \vee j_1) \wedge (b_1 \vee d_1 \vee j_1) \wedge (a_1 \vee e_1 \vee j_1) \wedge (a_0 \wedge d_0 \wedge l_0) \wedge (c_1 \vee e_1 \vee j_1) \wedge \)\\
    \((b_1 \vee f_1 \vee j_1) \wedge (a_1 \vee g_1 \vee l_1) \wedge (b_1 \vee d_1 \vee l_1) \wedge (a_1 \vee e_1 \vee o_1) \wedge (c_1 \vee f_1 \vee j_1) \wedge\)\\
    \((b_1 \vee h_1 \vee l_1) \wedge (a_1 \vee i_1 \vee o_1) \wedge (b_1 \vee f_1 \vee p_1) \wedge (c_1 \vee e_1 \vee o_1) \wedge (a_1 \vee g_1 \vee r_1) \wedge\)\\
    \((c_1 \vee f_1 \vee p_1) \wedge (a_1 \vee i_1 \vee r_1).\)\\
\end{center}
If we look through this logical statement for contradictions of Type II or Type III, we would not find a contradiction. However, $(l_0 \wedge j_1)$ is not possible. So, $(a_1 \vee d_1 \vee j_1) \wedge (a_0 \wedge d_0 \wedge l_0)$, which is within the logical statement above, is impossible. 
We again wrote a Python program which checks for both contradictions of Type II and Type III simultaneously (see Appendix \ref{sec:contradiction-code-1-and-2}). Using this, we further reduce to the $46$ sets shown in Appendix~\ref{sec:alt-sets-final}. In the next section, we establish that these are in fact the Weyl alternation sets of $\ccc$, which is the content of Theorem \ref{thm:alternationSetsFinal}.

\section{Weyl alternation sets}\label{sec: alt sets}

Now that we have discarded the forbidden alternation sets, it is left to prove that the $46$ sets in Appendix~\ref{sec:alt-sets-final} arise as alternation sets for a pair of dominant weights $\lambda$ and $\mu$. To do so, for each of the sets $S$ listed,
it is necessary to find a pair of dominant weights $\lambda$ and $\mu$ which induce each of these sets as alternation sets, i.e.,~such that $\mathcal{A}(\lambda,\mu)=S$. This establishes the existence of these sets as Weyl alternation sets. 

In Table~\ref{table:alt-set-examples}, we present pairs $(\lambda,\mu)$ which induce each of the $46$ alternation sets. These were found using a Sage program in which we vary the coefficients $m,n,k,x,y,z\in\mathbb{N}$ defining the weights $\lambda=m\w_1+n\w_2+k\w_3$ and $\mu=x\w_1+y\w_2+z\w_3$, and which computes the corresponding Weyl alternation set $\A(\lambda,\mu)$. This program is presented in Appendix~\ref{sec:calculating-alt-sets}. Note that when we write $(a,b,c)$ under the columns of $\lambda$ and $\mu$ this denotes the coefficients of these weights when they are expressed as a sum of fundamental weights. 

\begin{center}
\begin{longtable}[c]{| p{4.5in} | p{0.5in} | p{0.5in} | }
\caption{Examples of alternation sets with their given weights.} \label{table:alt-set-examples}\\
\hline
Alternation set $\mathcal{A}(\lambda,\mu)$& $\lambda$ & $\mu$ \\ \hline
\hline
\endfirsthead
\hline
Alternation set $\mathcal{A}(\lambda,\mu)$& $\lambda$ & $\mu$\\ \hline
\hline
\endhead
\hline
\endfoot

$\{\}$ & $(0, 0, 0)$ & $(0, 0, 2)$\\
\hline
$\{1\}$ 
& $(0, 0, 0)$ & $(0, 0, 0)$\\
\hline
$\{1$, $ s_1\}$ & $(0, 2, 2)$ & $(0, 1, 2)$\\
\hline
$\{1$, $ s_2\}$ & $(0, 0, 2)$ & $(1, 0, 1)$\\
\hline
$\{1$, $ s_3\}$ & $(0, 2, 0)$ & $(2, 0, 0)$\\
\hline
$\{1$, $ s_1$, $ s_2\}$ & $(0, 1, 2)$ & $(0, 0, 2)$\\
\hline
$\{1$, $s_2$, $s_3\}$ & $(0, 2, 2)$ & $(4, 0, 0)$\\
\hline
$\{1$, $ s_1$, $s_2$, $ s_2s_1\}$ & $(0, 0, 2)$ & $(0, 1, 0)$\\
\hline
$\{1$, $ s_1$, $ s_3$, $ s_3s_1\}$ & $(0, 2, 0)$ & $(0, 1, 0)$\\
\hline
$\{1$, $s_2$, $s_3$, $s_2s_3\}$ & $(4, 1, 0)$ & $(0, 1, 0)$\\
\hline
$\{1$, $s_2$, $s_3$, $ s_3s_2\}$ & $(4, 0, 2)$ & $(2, 0, 0)$\\
\hline
$\{1$, $s_1$, $s_2$, $s_3$, $ s_3s_1\}$ & $(0, 1, 0)$ & $(0, 0, 0)$\\
\hline
$\{1$, $ s_2$, $ s_3$, $ s_2s_3$, $s_3s_2\}$ & $(4, 0, 1)$ & $(1, 0, 0)$\\
\hline
$\{1$, $ s_1$, $ s_2$, $s_3$, $ s_2s_1$, $ s_3s_1\}$ & $(0, 1, 2)$ & $(0, 1, 0)$\\
\hline
$\{1$, $s_1$, $s_2$, $s_3$, $ s_2s_3$, $ s_3s_1\}$ & $(2, 1, 0)$ & $(0, 0, 0)$\\
\hline
$\{1$, $ s_1$, $ s_2$, $ s_3$, $ s_3s_1$, $ s_3s_2\}$ & $(7, 1, 3)$ & $(0, 3, 0)$\\
\hline
$\{1$, $ s_1$, $ s_2$, $ s_1s_2$, $ s_2s_1$ $ s_1s_2s_1\}$ & $(0, 0, 4)$ & $(0, 0, 2)$\\
\hline
$\{1$, $ s_2$, $ s_3$, $ s_2s_3$, $ s_3s_2$, $s_2s_3s_2\}$ & $(4, 0, 0)$ & $(0, 0, 0)$\\
\hline
$\{1$, $s_1$, $s_2$, $s_3$, $s_2s_1$, $ s_2s_3$, $ s_3s_1\}$ & $(0, 3, 0)$ & $(0, 0, 0)$\\
\hline
$\{1$, $ s_1$, $ s_2$, $s_3$, $ s_2s_1$, $ s_3s_1$, $ s_3s_2\}$ & $(3, 0, 2)$ & $(1, 0, 0)$\\
\hline
$\{1$, $ s_1$, $ s_2$, $ s_3$, $ s_2s_3$, $ s_3s_1$, $ s_3s_2\}$ & $(5, 1, 1)$ & $(0, 1, 0)$\\
\hline
$\{1$, $s_1$, $s_2$, $s_3$, $ s_1s_2$, $ s_2s_1$, $ s_3s_1$, $ s_1s_2s_1\}$ & $(0, 0, 2)$ & $(0, 0, 0)$\\
\hline
$\{1$, $ s_1$, $ s_2$, $s_3$, $ s_2s_1$, $ s_2s_3$, $ s_3s_1$, $ s_3s_2\}$ & $(3, 0, 1)$ & $(0, 0, 0)$\\
\hline
$\{1$, $ s_1$, $ s_2$, $ s_3$, $ s_2s_1$, $ s_2s_3$, $ s_3s_1$, $ s_2s_3s_1\}$ & $(0, 4, 0)$ & $(0, 0, 0)$\\
\hline
$\{1$, $ s_1$, $ s_2$, $ s_3$, $ s_2s_3$, $ s_3s_1$, $ s_3s_2$, $s_2s_3s_2\}$ & $(4, 1, 0)$ & $(0, 0, 0)$\\
\hline
$\{1$, $ s_2$, $s_3$, $ s_2s_3$, $ s_3s_2$, $s_2s_3s_2$, $ s_3s_2s_3$, $ s_3s_2s_3s_2\}$ & $(6, 0, 0)$ & $(0, 0, 0)$\\
\hline
$\{1$, $ s_1$, $ s_2$, $s_3$, $ s_1s_2$, $ s_2s_1$, $ s_2s_3$, $ s_3s_1$, $ s_1s_2s_1\}$ & $(0, 3, 2)$ & $(0, 0, 0)$\\
\hline
$\{1$, $ s_1$, $ s_2$, $s_3$, $ s_2s_1$, $ s_2s_3$, $ s_3s_1$, $ s_3s_2$, $ s_2s_3s_1\}$ & $(3, 4, 1)$ & $(0, 0, 0)$\\
\hline
$\{1$, $ s_1$, $ s_2$, $s_3$, $ s_2s_1$, $ s_2s_3$, $ s_3s_1$, $ s_3s_2$, $s_2s_3s_2\}$  & $(4, 2, 0)$ & $(0, 0, 0)$\\
\hline
$\{1$, $ s_1$, $ s_2$, $s_3$, $ s_1s_2$, $ s_2s_1$, $ s_2s_3$, $ s_3s_1$, $ s_1s_2s_1$, $ s_2s_3s_1\}$ & $(0, 4, 2)$ & $(0, 0, 0)$\\
\hline
$\{1$, $ s_1$, $ s_2$, $s_3$, $ s_1s_2$, $ s_2s_1$, $ s_3s_1$, $ s_3s_2$, $ s_1s_2s_1$, $ s_3s_1s_2\}$ & $(0, 0, 4)$ & $(0, 0, 0)$\\
\hline
$\{1$, $ s_1$, $ s_2$, $s_3$, $ s_2s_1$, $ s_2s_3$, $ s_3s_1$, $ s_3s_2$, $ s_2s_3s_1$, $s_2s_3s_2\}$ & $(4, 4, 0)$ & $(0, 0, 0)$\\
\hline
$\{1$, $ s_1$, $ s_2$, $s_3$, $ s_2s_3$, $ s_3s_1$, $ s_3s_2$, $s_2s_3s_2$, $ s_3s_2s_3$, $ s_3s_2s_3s_2\}$ & $(6, 1, 0)$ & $(0, 0, 0)$\\
\hline
$\{1$, $ s_1$, $ s_2$, $s_3$, $ s_1s_2$, $ s_2s_1$, $ s_2s_3$, $ s_3s_1$, $ s_3s_2$, $ s_1s_2s_1$, $ s_3s_1s_2\}$ & $(0, 3, 4)$ & $(0, 0, 0)$\\
\hline
$\{1$, $ s_1$, $ s_2$, $s_3$, $ s_2s_1$, $ s_2s_3$, $ s_3s_1$, $ s_3s_2$, $s_2s_3s_2$, $ s_3s_2s_3$, $ s_3s_2s_3s_2\}$ & $(6, 2, 0)$ & $(0, 0, 0)$\\
\hline
$\{1$, $ s_1$, $ s_2$, $s_3$, $ s_1s_2$, $ s_2s_1$, $ s_2s_3$, $ s_3s_1$, $ s_3s_2$, $ s_1s_2s_1$, $ s_2s_3s_1$, $ s_3s_1s_2\}$ & $(0, 4, 4)$ & $(0, 0, 0)$\\
\hline
$\{1$, $ s_1$, $ s_2$, $s_3$, $ s_1s_2$, $ s_2s_1$, $ s_2s_3$, $ s_3s_1$, $ s_3s_2$, $ s_1s_2s_1$, $s_2s_3s_2$, $ s_3s_1s_2\}$ & $(4, 0, 2)$ & $(0, 0, 0)$\\
\hline
$\{1$, $ s_1$, $ s_2$, $s_3$, $s_1s_2$, $ s_2s_1$, $s_2s_3$, $ s_3s_1$, $ s_3s_2$, $ s_1s_2s_1$, $ s_3s_1s_2$, $ s_3s_2s_1$, $ s_3s_1s_2s_1\}$  & $(0, 3, 6)$ & $(0, 0, 0)$\\
\hline
$\{1$, $ s_1$, $ s_2$, $s_3$, $ s_1s_2$, $ s_2s_1$, $ s_3s_1$, $ s_3s_2$, $ s_1s_2s_1$, $ s_3s_1s_2$, $ s_3s_2s_1$, $ s_3s_1s_2s_1\}$ & $(0, 0, 6)$ & $(0, 0, 0)$\\
\hline
$\{1$, $ s_1$, $ s_2$, $s_3$, $ s_2s_1$, $ s_2s_3$, $ s_3s_1$, $ s_3s_2$, $ s_2s_3s_1$, $s_2s_3s_2$, $ s_3s_2s_3$, $ s_3s_2s_3s_2\}$ & $(6, 4, 0)$ & $(0, 0, 0)$\\
\hline
$\{1$, $ s_1$, $ s_2$, $s_3$, $ s_1s_2$, $ s_2s_1$, $ s_2s_3$, $ s_3s_1$, $ s_3s_2$, $s_1s_2s_1$, $ s_2s_3s_1$, $s_2s_3s_2$, $ s_3s_1s_2\}$ & $(4, 4, 2)$ & $(0, 0, 0)$\\
\hline
$\{1$, $ s_1$, $ s_2$, $s_3$, $ s_1s_2$, $ s_2s_1$, $ s_2s_3$, $ s_3s_1$, $ s_3s_2$, $ s_1s_2s_1$, $ s_2s_3s_1$, $ s_3s_1s_2$, $ s_3s_2s_1$, $ s_3s_1s_2s_1\}$ & $(0, 4, 6)$ & $(0, 0, 0)$\\
\hline
$\{1$, $ s_1$, $ s_2$, $s_3$, $ s_1s_2$, $ s_2s_1$, $ s_2s_3$, $ s_3s_1$, $s_3s_2$, $ s_1s_2s_1$, $s_2s_3s_2$, $ s_3s_1s_2$, $ s_3s_2s_1$, $ s_3s_1s_2s_1\}$ & $(4, 0, 10)$ & $(0, 0, 0)$\\
\hline
$\{1$, $ s_1$, $ s_2$, $s_3$, $ s_1s_2$, $ s_2s_1$, $ s_2s_3$, $ s_3s_1$, $ s_3s_2$, $ s_1s_2s_1$, $s_2s_3s_2$, $ s_3s_1s_2$, $ s_3s_2s_3$, $ s_3s_2s_3s_2\}$ & $(8, 0, 2)$ & $(0, 0, 0)$\\
\hline
$\{1$, $ s_1$, $ s_2$, $s_3$, $ s_1s_2$, $ s_2s_1$, $ s_2s_3$, $ s_3s_1$, $ s_3s_2$, $ s_1s_2s_1$, $ s_2s_3s_1$, $s_2s_3s_2$, $ s_3s_1s_2$,  $ s_3s_2s_1$, $ s_3s_1s_2s_1\}$ & $(4, 4, 10)$ & $(0, 0, 0)$\\
\hline
$\{1$, $ s_1$, $ s_2$, $s_3$, $ s_1s_2$, $ s_2s_1$, $ s_2s_3$, $ s_3s_1$, $s_3s_2$, $ s_1s_2s_1$, $ s_2s_3s_1$, $s_2s_3s_2$, $ s_3s_1s_2$, $ s_3s_2s_3$, $ s_3s_2s_3s_2\}$ & $(8, 4, 2)$ & $(0, 0, 0)$\\

\end{longtable}
\end{center}

Given this characterization of the Weyl alternation sets, we now establish the following result.

\begin{theorem} \label{thm:alternationSetsFinal}
    Let $\lambda = m\w_1 + n\w_2 + k\w_3$  and $\mu = x\w_1 + y\w_2 + z\w_3$ with $m,n,k,x,y,z \in \N$ be two weights of $\ccc$  such that $\frac{m+k+x+z}{2} \in \N$. Then, there are $46$ distinct Weyl alternation sets alternation sets $\mathcal{A}(\lambda,\mu)$, and these are listed in Appendix~\ref{sec:alt-sets-final}. 
\end{theorem}

\begin{proof}
As described above, let $\lambda = m\w_1 + n\w_2 + k\w_3$ and $\mu = x\w_1 + y\w_2 + z\w_3$ with $m,n,k,x,y,z \in \N$ and $\frac{m+k+x+z}{2} \in \N$.
Let $a$ through $r$ be defined as in~\eqref{eq:lowecase-eq} and $A$ through $Q$ be defined as in~\eqref{eq:so many equations}.

Then, from Corollary~\ref{lemma:NT} we have: \begin{center}
$A$ contributes nontrivially if and only if $a_0 \wedge d_0 \wedge j_0$ holds true, \\
$B$ contributes nontrivially if and only if $b_0 \wedge d_0 \wedge j_0$ holds true, \\
$C$ contributes nontrivially if and  only if $a_0 \wedge e_0 \wedge j_0$ holds true, \\
$D$ contributes nontrivially if and  only if $a_0 \wedge d_0 \wedge l_0$ holds true, \\
$E$ contributes nontrivially if and  only if $c_0 \wedge e_0 \wedge j_0$ holds true, \\
$F$ contributes nontrivially if and  only if $b_0 \wedge f_0 \wedge j_0$ holds true, \\
$G$ contributes nontrivially if and  only if $a_0 \wedge g_0 \wedge l_0$ holds true, \\
$H$ contributes nontrivially if and  only if $b_0 \wedge d_0 \wedge l_0$ holds true, \\
$I$ contributes nontrivially if and  only if $a_0 \wedge e_0 \wedge o_0$ holds true, \\
$J$ contributes nontrivially if and  only if $c_0 \wedge f_0 \wedge j_0$ holds true, \\
$K$ contributes nontrivially if and  only if $b_0 \wedge h_0 \wedge l_0$ holds true, \\
$L$ contributes nontrivially if and  only if $a_0 \wedge i_0 \wedge o_0$ holds true, \\
$M$ contributes nontrivially if and  only if $b_0 \wedge f_0 \wedge p_0$ holds true, \\
$N$ contributes nontrivially if and  only if $c_0 \wedge e_0 \wedge o_0$ holds true, \\
$O$ contributes nontrivially if and  only if $a_0 \wedge g_0 \wedge r_0$ holds true, \\
$P$ contributes nontrivially if and  only if $c_0 \wedge f_0 \wedge p_0$ holds true, \\
$Q$ contributes nontrivially if and  only if $a_0 \wedge i_0 \wedge r_0$ holds true.
\end{center}

Intersecting these solution sets on the lattice $\Z\alpha_1 \oplus \Z\alpha_2 \oplus \Z\alpha_3$, as we have done in Section~\ref{sec:multiplicity}, produces the desired results.
\end{proof}

\section{The \texorpdfstring{$q$}{q}-analog of Kostant's weight multiplicity formula for \texorpdfstring{$\ccc$}{sp\_6(C)}}\label{sec:closed-formula-q-analog-KWMF}

Our final result, Theorem~\ref{thm:finalThm} establishes a closed formula for the $q$-analog of Kostant's weight multiplicity formula for the Lie algebra $\ccc$ for dominant integral weights $\lambda$ and $\mu$. The following theorem synthesizes our previous two results, drawing from the closed formulas for the $q$-analog of Kostant’s partition function given by Theorem~\ref{thm:q-konstants-equation} and the Weyl alternation sets $A(\lambda, \mu)$ given by Theorem~\ref{thm:alternationSetsFinal}.

\begin{theorem}\label{thm:finalThm}
Let $\lambda = m\w_1 + n\w_2 + k\w_3$  and $\mu = x\w_1 + y\w_2 + z\w_3$ with $m,n,k,x,y,z \in \N$ and $\frac{m+k+x+z}{2} \in \N$. Let $a$ through $r$ be defined as in~\eqref{eq:lowecase-eq} and let $A_q$ through $Q_q$ be the $q$-analog version of the expressions $A$ through $Q$, respectively, as defined in~\eqref{eq:so many equations}. Then, the following hold:

\begin{enumerate}
\item If $a,b,c,d,e,f,g,h,i,l,r,o,j \in \N$ and $p \notin \N$, then

$m_q(\lambda, \mu) = \poly{A}{q} - \poly{B}{q} - \poly{C}{q} - \poly{D}{q}- \poly{E}{q} + \poly{F}{q} + \poly{G}{q}  + \poly{H}{q} + \poly{I}{q} - \poly{J}{q} - \poly{K}{q} - \poly{L}{q} - \poly{N}{q} - \poly{O}{q} + \poly{Q}{q}$.

\item If $a,b,c,d,e,f,g,h,i,l,o,p,j \in \N$ and $r \notin \N$, then

$m_q(\lambda, \mu) = \poly{A}{q} - \poly{B}{q} - \poly{C}{q} - \poly{D}{q}- \poly{E}{q} + \poly{F}{q} + \poly{G}{q} + \poly{H}{q} + \poly{I}{q} - \poly{J}{q} - \poly{K}{q} - \poly{L}{q} - \poly{M}{q} - \poly{N}{q} + \poly{P}{q}$.

\item If $a,b,c,d,e,f,g,i,l,r,o,j \in \N$ and $h,p \notin \N$, then

$m_q(\lambda, \mu) = \poly{A}{q} - \poly{B}{q} - \poly{C}{q} - \poly{D}{q}- \poly{E}{q} + \poly{F}{q} + \poly{G}{q} + \poly{H}{q} + \poly{I}{q} - \poly{J}{q} - \poly{L}{q} - \poly{N}{q} - \poly{O}{q} + \poly{Q}{q}$.

\item If $a,b,c,d,e,f,g,i,l,o,p,j \in \N$ and $h,r \notin \N$, then

$m_q(\lambda, \mu) = \poly{A}{q} - \poly{B}{q} - \poly{C}{q} - \poly{D}{q}- \poly{E}{q} + \poly{F}{q} + \poly{G}{q} + \poly{H}{q} + \poly{I}{q} - \poly{J}{q} - \poly{L}{q} - \poly{M}{q} - \poly{N}{q} + \poly{P}{q}$. 

\item If $a,b,c,d,e,f,g,h,l,o,p,j \in \N$ and $i,r \notin \N$, then

$m_q(\lambda, \mu) = \poly{A}{q} - \poly{B}{q} - \poly{C}{q} - \poly{D}{q}- \poly{E}{q} + \poly{F}{q} + \poly{G}{q} + \poly{H}{q} + \poly{I}{q} - \poly{J}{q} - \poly{K}{q} - \poly{M}{q} - \poly{N}{q} + \poly{P}{q}$.

\item If $a,b,c,d,e,f,g,l,o,p,j \in \N$ and $h,i,r \notin \N$, then

$m_q(\lambda, \mu) = \poly{A}{q} - \poly{B}{q} - \poly{C}{q} - \poly{D}{q}- \poly{E}{q} + \poly{F}{q} + \poly{G}{q} + \poly{H}{q} + \poly{I}{q} - \poly{J}{q} - \poly{M}{q} - \poly{N}{q} + \poly{P}{q}$.

\item If $a,b,c,d,e,f,g,h,i,l,o,j \in \N$ and $r,p \notin \N$, then

$m_q(\lambda, \mu) = \poly{A}{q} - \poly{B}{q} - \poly{C}{q} - \poly{D}{q}- \poly{E}{q} + \poly{F}{q} + \poly{G}{q} + \poly{H}{q} + \poly{I}{q} - \poly{J}{q} - \poly{K}{q} - \poly{L}{q} - \poly{N}{q}$.

\item If $a,b,c,d,e,f,g,i,l,o,j \in \N$ and $h,r,p \notin \N$, then

$m_q(\lambda, \mu) = \poly{A}{q} - \poly{B}{q} - \poly{C}{q} - \poly{D}{q}- \poly{E}{q} + \poly{F}{q} + \poly{G}{q} + \poly{H}{q} + \poly{I}{q} - \poly{J}{q} - \poly{L}{q} - \poly{N}{q}$.

\item If $a,b,c,d,e,f,l,o,p,j \in \N$ and $g,h,i,r \notin \N$, then

$m_q(\lambda, \mu) = \poly{A}{q} - \poly{B}{q} - \poly{C}{q} - \poly{D}{q}- \poly{E}{q} + \poly{F}{q} + \poly{H}{q} + \poly{I}{q} - \poly{J}{q} - \poly{M}{q} - \poly{N}{q} + \poly{P}{q}$.

\item If $a,b,c,d,e,f,g,h,l,o,j \in \N$ and $i,r,p \notin \N$, then

$m_q(\lambda, \mu) = \poly{A}{q} - \poly{B}{q} - \poly{C}{q} - \poly{D}{q}- \poly{E}{q} + \poly{F}{q} + \poly{G}{q} + \poly{H}{q} + \poly{I}{q} - \poly{J}{q} - \poly{K}{q} - \poly{N}{q}$.

\item If $a,b,d,e,f,g,h,i,l,r,o,j \in \N$ and $c,p \notin \N$, then

$m_q(\lambda, \mu) = \poly{A}{q} - \poly{B}{q} - \poly{C}{q} - \poly{D}{q}+ \poly{F}{q} + \poly{G}{q} + \poly{H}{q} + \poly{I}{q} - \poly{K}{q} - \poly{L}{q} - \poly{O}{q} + \poly{Q}{q}$.

\item If $a,b,c,d,e,f,g,l,o,j \in \N$ and $h,i,r,p \notin \N$, then

$m_q(\lambda, \mu) = \poly{A}{q} - \poly{B}{q} - \poly{C}{q} - \poly{D}{q}- \poly{E}{q} + \poly{F}{q} + \poly{G}{q} + \poly{H}{q} + \poly{I}{q} - \poly{J}{q} - \poly{N}{q}$.

\item If $a,b,d,e,f,g,i,l,r,o,j \in \N$ and $h,c,p \notin \N$, then

$m_q(\lambda, \mu) = \poly{A}{q} - \poly{B}{q} - \poly{C}{q} - \poly{D}{q}+ \poly{F}{q} + \poly{G}{q} + \poly{H}{q} + \poly{I}{q} - \poly{L}{q} - \poly{O}{q} + \poly{Q}{q}$.

\item If $a,b,c,d,e,f,g,h,l,j \in \N$ and $i,r,o,p \notin \N$, then

$m_q(\lambda, \mu) = \poly{A}{q} - \poly{B}{q} - \poly{C}{q} - \poly{D}{q}- \poly{E}{q} + \poly{F}{q} + \poly{G}{q} + \poly{H}{q} - \poly{J}{q} - \poly{K}{q}$.

\item If $a,b,c,d,e,f,l,o,j \in \N$ and $g,h,i,r,p \notin \N$, then

$m_q(\lambda, \mu) = \poly{A}{q} - \poly{B}{q} - \poly{C}{q} - \poly{D}{q}- \poly{E}{q} + \poly{F}{q} + \poly{H}{q} + \poly{I}{q} - \poly{J}{q} - \poly{N}{q}$.

\item If $a,b,d,e,g,i,l,r,o,j \in \N$ and $f,h,c,p \notin \N$, then

$m_q(\lambda, \mu) = \poly{A}{q} - \poly{B}{q} - \poly{C}{q} - \poly{D}{q}+ \poly{G}{q} + \poly{H}{q} + \poly{I}{q} - \poly{L}{q} - \poly{O}{q} + \poly{Q}{q}$.

\item If $a,b,d,e,f,g,h,i,l,o,j \in \N$ and $r,c,p \notin \N$, then

$m_q(\lambda, \mu) = \poly{A}{q} - \poly{B}{q} - \poly{C}{q} - \poly{D}{q}+ \poly{F}{q} + \poly{G}{q} + \poly{H}{q} + \poly{I}{q} - \poly{K}{q} - \poly{L}{q}$.

\item If $a,b,c,d,e,f,g,l,j \in \N$ and $h,i,r,o,p \notin \N$, then

$m_q(\lambda, \mu) = \poly{A}{q} - \poly{B}{q} - \poly{C}{q} - \poly{D}{q}- \poly{E}{q} + \poly{F}{q} + \poly{G}{q} + \poly{H}{q} - \poly{J}{q}$.

\item If $a,b,d,e,f,g,i,l,o,j \in \N$ and $c,h,r,p \notin \N$, then

$m_q(\lambda, \mu) = \poly{A}{q} - \poly{B}{q} - \poly{C}{q} - \poly{D}{q}+ \poly{F}{q} + \poly{G}{q} + \poly{H}{q} + \poly{I}{q} - \poly{L}{q}$.

\item If $a,b,d,e,f,g,h,l,o,j \in \N$ and $c,i,r,p \notin \N$, then

$m_q(\lambda, \mu) = \poly{A}{q} - \poly{B}{q} - \poly{C}{q} - \poly{D}{q}+ \poly{F}{q} + \poly{G}{q} + \poly{H}{q} + \poly{I}{q} - \poly{K}{q}$.

\item If $a,b,c,d,e,f,l,j \in \N$ and $g,h,i,r,o,p \notin \N$, then

$m_q(\lambda, \mu) = \poly{A}{q} - \poly{B}{q} - \poly{C}{q} - \poly{D}{q}- \poly{E}{q} + \poly{F}{q} + \poly{H}{q} - \poly{J}{q}$.

\item If $a,b,d,e,f,g,h,l,j \in \N$ and $c,r,o,p \notin \N$ and $i \in \Z$, then

$m_q(\lambda, \mu) = \poly{A}{q} - \poly{B}{q} - \poly{C}{q} - \poly{D}{q}+ \poly{F}{q} + \poly{G}{q} + \poly{H}{q} - \poly{K}{q}$.

\item If $a,b,d,e,f,g,l,o,j \in \N$ and $c,h,i,r,p \notin \N$, then

$m_q(\lambda, \mu) = \poly{A}{q} - \poly{B}{q} - \poly{C}{q} - \poly{D}{q}+ \poly{F}{q} + \poly{G}{q} + \poly{H}{q} + \poly{I}{q}$.

\item If $a,b,d,e,g,i,l,o,j \in \N$ and $c,f,h,r,p \notin \N$, then

$m_q(\lambda, \mu) = \poly{A}{q} - \poly{B}{q} - \poly{C}{q} - \poly{D}{q}+ \poly{G}{q} + \poly{H}{q} + \poly{I}{q} - \poly{L}{q}$.

\item If $a,d,e,g,i,l,r,o,j \in \N$ and $b,c,h,p \notin \N$ and $f \in \Z$, then

$m_q(\lambda, \mu) = \poly{A}{q} - \poly{C}{q} - \poly{D}{q}+ \poly{G}{q} + \poly{I}{q} - \poly{L}{q} - \poly{O}{q} + \poly{Q}{q}$.

\item If $a,b,d,e,f,g,l,j \in \N$ and $c,h,r,o,p \notin \N$ and $i \in \Z$, then

$m_q(\lambda, \mu) = \poly{A}{q} - \poly{B}{q} - \poly{C}{q} - \poly{D}{q}+ \poly{F}{q} + \poly{G}{q} + \poly{H}{q}$.

\item If $a,b,d,e,f,l,o,j \in \N$ and $c,g,h,i,r,p \notin \N$, then

$m_q(\lambda, \mu) = \poly{A}{q} - \poly{B}{q} - \poly{C}{q} - \poly{D}{q}+ \poly{F}{q} + \poly{H}{q} + \poly{I}{q}$.

\item If $a,b,d,e,g,l,o,j \in \N$ and $c,f,h,i,r,p \notin \N$, then

$m_q(\lambda, \mu) = \poly{A}{q} - \poly{B}{q} - \poly{C}{q} - \poly{D}{q}+ \poly{G}{q} + \poly{H}{q} + \poly{I}{q}$.

\item If $a,b,c,d,e,f,j \in \N$ and $g,h,i,l,r,o,p \notin \N$, then

$m_q(\lambda, \mu) = \poly{A}{q} - \poly{B}{q} - \poly{C}{q} - \poly{E}{q} + \poly{F}{q} - \poly{J}{q}$.

\item If $a,b,d,e,f,l,j \in \N$ and $c,g,h,i,r,o,p \notin \N$, then

$m_q(\lambda, \mu) = \poly{A}{q} - \poly{B}{q} - \poly{C}{q} - \poly{D}{q}+ \poly{F}{q} + \poly{H}{q}$.

\item If $a,b,d,e,g,l,j \in \N$ and $c,f,h,r,o,p \notin \N$ and $i \in \Z$, then

$m_q(\lambda, \mu) = \poly{A}{q} - \poly{B}{q} - \poly{C}{q} - \poly{D}{q}+ \poly{G}{q} + \poly{H}{q}$.

\item If $ a,d,e,g,i,l,o,j \in \N$ and $b,c,h,r,p \notin \N$ and $f \in \Z$, then

$m_q(\lambda, \mu) = \poly{A}{q} - \poly{C}{q} - \poly{D}{q}+ \poly{G}{q} + \poly{I}{q} - \poly{L}{q}$.

\item If $a,b,d,e,l,o,j \in \N$ and $c,f,g,h,i,r,p \notin \N$, then

$m_q(\lambda, \mu) = \poly{A}{q} - \poly{B}{q} - \poly{C}{q} - \poly{D}{q}+ \poly{H}{q} + \poly{I}{q}$.

\item If $a,b,d,e,l,j \in \N$ and $c,f,g,h,i,r,o,p \notin \N$, then

$m_q(\lambda, \mu) = \poly{A}{q} - \poly{B}{q} - \poly{C}{q} - \poly{D}{q}+ \poly{H}{q}$.

\item If $a,d,e,g,l,o,j \in \N$ and $b,c,h,i,r,p \notin \N$ and $f\in \Z$, then

$m_q(\lambda, \mu) = \poly{A}{q} - \poly{C}{q} - \poly{D}{q}+ \poly{G}{q} + \poly{I}{q}$.

\item If $a,b,d,e,f,j \in \N$ and $c,g,h,i,l,r,o,p \notin \N$, then

$m_q(\lambda, \mu) = \poly{A}{q} - \poly{B}{q} - \poly{C}{q} + \poly{F}{q}$.

\item If $a,b,d,l,j \in \N$ and $c,e,f,g,h,i,r,o,p \notin \N$, then

$m_q(\lambda, \mu) = \poly{A}{q} - \poly{B}{q} - \poly{D}{q}+ \poly{H}{q}$.

\item If $a,d,e,g,l,j \in \N$ and $b,c,h,i,r,o,p \notin \N$ and $f \in \Z$, then

$m_q(\lambda, \mu) = \poly{A}{q} - \poly{C}{q} - \poly{D}{q}+ \poly{G}{q}$.

\item If $a,d,e,l,o,j \in \N$ and $b,c,g,h,i,r,p \notin \N$ and $f \in \Z$, then

$m_q(\lambda, \mu) = \poly{A}{q} - \poly{C}{q} - \poly{D}{q}+ \poly{I}{q}$.

\item If $a,b,d,e,j \in \N$ and $c,f,g,h,i,l,r,o,p \notin \N$, then

$m_q(\lambda, \mu) = \poly{A}{q} - \poly{B}{q} - \poly{C}{q}$.

\item If $a,d,e,l,j \in \N$ and $b,c,g,h,i,r,o,p \notin \N$ and $f \in \Z$, then 

$m_q(\lambda, \mu) = \poly{A}{q} - \poly{C}{q} - \poly{D}{q}$.

\item If $a,b,d,j \in \N$ and $c,e,f,g,h,i,l,r,o,p \notin \N$, then

$m_q(\lambda, \mu) = \poly{A}{q} - \poly{B}{q}$.

\item If ($a,d,e,j,g \in \N$ and $b,c,l,o,h,i,p,r,f \notin \N$) or
($a,d,e,j,g,i \in \N$ and $b,c,l,o,h,p,r,f \notin \N$) or
($a,d,e,j,f \in \N$ and $b,c,l,o,g,h,i,p,r \notin \N$) or
($a,d,e,j \in \N$ and $b,c,l,o,g,h,i,p,r,f \notin \N$), then

$m_q(\lambda, \mu) = \poly{A}{q} - \poly{C}{q}$.

\item If $a,d,l,j \in \N$ and $b,c,e,f,g,h,i,r,o,p \notin \N$, then

$m_q(\lambda, \mu) = \poly{A}{q} - \poly{D}{q}$.

\item If $a,d,j \in \N$ and $b,c,e,f,g,h,i,l,r,o,p \notin \N$, then 

$m_q(\lambda, \mu) = \poly{A}{q}$.

\item $m_q(\lambda, \mu) = 0$ otherwise.
\end{enumerate}
\end{theorem}

\begin{proof}
This result follows directly from Theorem~\ref{thm:alternationSetsFinal}.
\end{proof}

To conclude, we implement Theorem~\ref{thm:finalThm} to calculate $m_q(\lambda, \mu)$ in three examples. 
For the first example, we confirm a well-known result of Lusztig (see \cite{LL}*{19, Section 10, p. 226}) which established that $m_q(\tilde{\alpha},0)=q+q^{3}+q^{5}$, where 
 $\tilde{\alpha}$ is the highest root of $\ccc$ and $1$, $3$, and $5$ are the exponents of $\ccc$. In the second example, we let $\lambda=\mu = 0$ and in the final example, we confirm one of our results from Table~\ref{table:alt-set-examples}.

\begin{example}
Recall that the highest root of $\ccc$ is $\lambda=\tilde{\alpha} = 2\alpha_1 + 2\alpha_2 + \alpha_3=2\w_1$. Letting $\mu=0$ and evaluating the variables $a$ through $o$, we find that {only $a,d,e,j,l\in\NN$ with $a=d=2$, $e=j=1$, and $l=0$.} Then by case 41 in Theorem \ref{thm:finalThm}, we have that $$m_q(\tilde{\alpha},0)=A_q-C_q-D_q.$$ 
Computing the associated values of the $q$-analog of the partition function yields
\begin{align*}
    A_q&=q + 2 q^2 + 4 q^3 + 2 q^4 + q^5,\\
    C_q&=q^2+2q^3+q^4,\mbox{ and}\\
    D_q&=q + q^3 + q^4.
\end{align*}
Thus, 
$$m_q(\highestroot,0)=q+q^3+q^5,$$
which when evaluated at $q=1$ yields $m(\tilde{\alpha}, 0)=3$, as expected.
\end{example}

\begin{example}
Now, we let $\lambda = \mu = 0$. In this case, the only coefficients which are nonnegative are $a = d = j = 0$. This implies $$m_q = A = \wp_q(a\alpha_1 + d\alpha_2 + j\alpha_3).$$ We know $\wp_q(0\alpha_1 + 0\alpha_2 + 0\alpha_3) = 1$, so $$m_q(0, 0) = 1.$$ Similarly, when we let $q = 1$, we obtain $m(0,0) = 1$.  
\end{example}

\begin{example}
Let $\lambda = 2\w_3$ and $\mu = \w_1+\w_3$. Computing for the nonnegative coefficients, we have that $a,d,e,j\in\NN$ with $a=e=0,$ and $d = j =1 $. Then, by case 43 in Theorem \ref{thm:finalThm}, this implies that $$m_q= A_q-C_q,$$ as expected from Table~\ref{table:alt-set-examples}. Then, computing the value of the $q$-analog of the partition function, we find that $A_q=q+q^2$ and $C_q=q$. Hence $$m_q = q^2,$$ when evaluated at $q = 1$ yields $m(2\w_3,\w_1+\w_3) = 1$.
\end{example}

\section{Future work}\label{sec:future}

In the work of Harris, Lescinsky, and Mabie \cite{HLM}, they considered the Lie algebra $\mathfrak{sl}_{3}(\mathbb{C})$ and provided characterizations of the Weyl alternation sets for $\lambda,\mu\in\mathbb{Z}\w_1\oplus\mathbb{Z}\w_2$. 
Moreover, these Weyl alternation sets are described via patterns on the root lattice. This work fixes a particular weight $\mu $ and color codes all weights $\lambda$ and $\beta$ in the root lattice which satisfy $\A(\lambda,\mu)=\A(\beta,\mu)$. 
In light of this, one direction of future work is to provide the same analysis of these lattice patterns in the case of $\mathfrak{sp}_6(\mathbb{C})$. A first step in doing this would require one to find all Weyl alternation sets for weights $\lambda$ and $\mu$ being dominant integral weights.

\bibliography{Bibliography}

\begin{bibdiv}
\begin{biblist}

\bib{Simpson}{article}{
      author={Benedetti, C.},
      author={Hanusa, C. R.~H.},
      author={Harris, P.~E.},
      author={Morales, A.},
      author={Simpson, A.},
       title={{K}ostant's partition function and magic multiplex juggling
  sequences},
        date={2020},
        note={Preprint arXiv:2001.03219},
}

\bib{ChangHarrisInsko}{article}{
      author={Chang, K.},
      author={Harris, P.~E.},
      author={Insko, E.},
       title={{K}ostant's weight multiplicity formula and the {F}ibonacci and
  {L}ucas numbers},
        date={2020},
        ISSN={2156-3527},
     journal={Journal of Combinatorics},
      volume={11},
      number={1},
       pages={141\ndash 167},
  url={https://doi-org.ezproxy2.williams.edu/10.4310/JOC.2020.v11.n1.a7},
}

\bib{CGHLMRK}{article}{
      author={Cockerham, J.},
      author={Guti\'errez~Gonz\'alez, M.},
      author={Harris, P.~E.},
      author={Loving, M.},
      author={Mini\~no, A.~V.},
      author={Rennie, J.},
      author={Kirby, G.~R.},
       title={Weight $q$-multiplicities for representations of the exceptional
  {L}ie algebra $\mathfrak{g}_2$},
        date={2020},
     journal={arXiv},
}

\bib{GHLMMRKT}{article}{
      author={Garcia, R.~E.},
      author={Harris, P.~E.},
      author={Loving, M.},
      author={Martinez, L.},
      author={Melendez, D.},
      author={Rennie, J.},
      author={Kirby, G.~R.},
      author={Tinoco, D.},
       title={On {K}ostant’s weight $q$-multiplicity formula for
  $\mathfrak{sl}_4(\mathbb{C})$},
        date={2020},
     journal={AAECC},
         url={https://link.springer.com/article/10.1007/s00200-020-00454-8},
}

\bib{GW}{book}{
      author={Goodman, R.},
      author={Wallach, N.~R.},
       title={Symmetry, representations and invariants},
   publisher={Springer},
     address={New York, NY},
        date={2009},
        ISBN={978-0-387-79851-6},
}

\bib{Hthesis}{thesis}{
      author={Harris, P.~E.},
       title={Combinatorial problems related to {K}ostant's weight multiplicity
  formula},
        type={Ph.D. Thesis},
        date={2012},
}

\bib{HIO}{article}{
      author={Harris, P.~E.},
      author={Insko, E.},
      author={Omar, M.},
       title={The q-analog of {K}ostant’s partition function and the highest
  root of the simple {L}ie algebras},
        date={201508},
     journal={Australasian Journal Of Combinatorics},
      volume={71},
      number={1},
       pages={68\ndash 91},
}

\bib{HIS}{misc}{
      author={Harris, P.~E.},
      author={Insko, E.},
      author={Simpson, A.},
       title={Computing weight $q$-multiplicities for the representations of
  the simple {L}ie algebras},
        date={2018},
      volume={29},
      number={4},
        ISBN={1432-0622},
        note={https://doi.org/10.1007/s00200-017-0346-7},
}

\bib{HIW}{article}{
      author={Harris, P.~E.},
      author={Insko, E.},
      author={Williams, L.},
       title={The adjoint representation of a {L}ie algebra and the support of
  {K}ostant's weight multiplicity formula},
        date={201312},
     journal={Journal of Combinatorics},
      volume={7},
      number={1},
}

\bib{HL}{article}{
      author={Harris, P.~E.},
      author={Lauber, E.~L.},
       title={Weight $q$-multiplicities for representations of
  $\mathfrak{sp}_4(\mathbb{C})$},
        date={2017},
     journal={Journal of Siberian Federal University. Mathematics \& Physics},
      volume={10},
      number={4},
       pages={494\ndash 502},
}

\bib{HLM}{article}{
      author={Harris, P.~E.},
      author={Lescinsky, H.},
      author={Mabie, G.},
       title={Lattice patterns for the support of {K}ostant's weight
  multiplicity formula on $\mathfrak{sl}_3(\mathbb{C})$},
        date={2018},
     journal={Minnesota Journal of Undergraduate Mathematics},
      volume={4},
      number={1},
         url={https://pubs.lib.umn.edu/index.php/mjum/article/view/4149},
}

\bib{LRRRHTU}{article}{
      author={Harris, P.~E.},
      author={Loving, M.},
      author={Ram{\'i}rez, J.},
      author={Rennie, J.},
      author={Kirby, G.},
      author={Davila, E.},
      author={Ulysse, F.~O.},
       title={Visualizing the support of {K}ostant's weight multiplicity
  formula for the rank two {L}ie algebras.},
        date={2019},
     journal={arXiv: Combinatorics},
         url={https://arxiv.org/pdf/1908.08405.pdf},
}

\bib{HRSS}{article}{
      author={Harris, P.~E.},
      author={Rahmoeller, M.},
      author={Schneider, L.},
      author={Simpson, A.},
       title={When is the $q$-multiplicity of a weight a power of $q$?},
        date={2019},
     journal={The Electronic Journal of Combinatorics},
      volume={26},
      number={4},
}

\bib{CODE}{manual}{
      author={Harris, P.~E.},
      author={Rodriguez-Hertz, M.},
      author={Qin, D.},
      author={Hollander, P.},
       title={Calculating weight multiplicities for {L}ie algebra of type
  ${C}_3$},
   publisher={GitHub},
        date={2021. [Online].},
        note={Available at
  \textit{https://github.com/21mdr1/Weight-Multiplicities}},
}

\bib{HRS}{article}{
      author={Harris, P.~E.},
      author={Schinder, L.},
      author={Rahmoeller, M.},
       title={On the asymptotic behavior of the q-analog of {K}ostant's
  partition function},
        date={2019},
     journal={To appear in Journal of Combinatorics},
         url={https://arxiv.org/pdf/1912.02266.pdf},
}

\bib{KMF}{article}{
      author={Kostant, B.},
       title={A formula for the multiplicity of a weight},
        date={1958},
     journal={Proc. Nat. Acad. Sci. U.S.A.},
      volume={44},
       pages={588\ndash 589},
}

\bib{LL}{article}{
      author={Lusztig, G.},
       title={Singularities, character formulas, and a $q$-analog of weight
  multiplicities},
        date={1983},
     journal={Ast$\acute{\text{e}}$risque},
      volume={101-102},
       pages={208\ndash 229},
}

\bib{SRM}{article}{
      author={Shahi, E.},
      author={Refaghat, H.},
      author={Marefat, Y.},
       title={{K}ostant partition function for $\mathfrak{sl}_4(\mathbb{C})$
  and $\mathfrak{sp}_6(\mathbb{C})$},
        date={2020},
     journal={Unpublished},
}

\bib{Tarski}{article}{
      author={Tarski, J.},
       title={The partition function for certain simple {L}ie algebras},
        date={1957},
     journal={Technical Report No. 7 Prepared under Contract A7 49(638)-79
  Division File No. 3. 22, United States Air Force, Office of Scientific
  Research},
  url={https://play.google.com/books/reader?id=osDh_OXtG7sC&hl=en&pg=GBS.PP3},
        note={Retrieved February 24, 2020},
}

\end{biblist}
\end{bibdiv}
\bibliographystyle{plain}

\addresseshere

\appendix 

\newpage

\section{Mathematica code evaluating the \texorpdfstring{$q$}{q}-analog of Kostant's partition function} \label{sec:kpf-eval-code}

This program takes as input the coefficients $m,n,k$ of a weight $\mu$ and outputs the $q$-analog of Kostant's partition function for $\mu$. To run the program, input the values of $m,n,k$ in the function $P[m,n,k]$.

\noindent $P[m,n,k]$ evaluated to the $q$-analog of Kostant's partition function for a weight $\mu = m * \alpha_1 + n * \alpha_2 + k * \alpha_3$, just plug in $m$, $n$, and $k$.

\begin{mmaCell}[pattern={m,n,k,m_,n_,k_}, undefined={P,q}, functionlocal={h,g,f,i,d,e}]{Code}
    P[m_, n_, k_] := 
    Sum[q^(m + n + k - d - e - 2*f - 3*g - 4*h - 2*i), 
        {h, 0, Min[Floor[m/2], Floor[n/2], k]}, 
        {g, 0, Min[m - 2*h, Floor[(n - 2*h)/2], k - h]}, 
        {f, 0, Min[m - 2*h - g, n - 2*h - 2*g, k - h - g]}, 
        {i, 0, Min[Floor[(n - 2*h - 2*g - f)/2], k - h - g - f]}, 
        {d, 0, Min[m - 2*h - g - f, n - 2*h - 2*g - f - 2*i]}, 
        {e, 0, Min[n - 2*h - 2*g - f - 2*i - d, k - h - g - f - i]}]
    P[\mmaUnd{m}, \mmaUnd{n}, \mmaUnd{k}]
\end{mmaCell}

\section{Mathematica code evaluating \texorpdfstring{$\sigma(\lambda+\rho)-\rho-\mu$}{sigma(lambda+rho) - rho - mu}} \label{sec:coeff-code}
This program takes as input a $\sigma \in W$ and uses it to determine the coefficients $C_{i, \sigma}$ of the expression $\sigma(\lambda+\rho)-\rho-\mu=C_{1,\sigma}\alpha_1+C_{2,\sigma}\alpha_2+C_{3,\sigma}\alpha_3$. To use, run the first cell, then input the desired $\sigma \in W$ in place of the word `Sigma' written with every generator separated by a comma.

\begin{mmaCell}[label={In[1]:=},undefined={alpha1,alpha2,alpha3,m,n,k,x,y,z},pattern={u_Plus,u,a_,r_,a,r,}]{Code}
    s1[alpha1] = -alpha1
    s1[alpha2] = alpha1 + alpha2
    s1[alpha3] = alpha3
    
    s2[alpha1] = alpha1 + alpha2
    s2[alpha2] = -alpha2
    s2[alpha3] = 2 alpha2 + alpha3
    
    s3[alpha1] = alpha1
    s3[alpha2] = alpha2 + alpha3
    s3[alpha3] = -alpha3
    
    numberq[_] := False
    numberq[_?NumericQ] := True
    numericq[m | n | k | x | y | z] := True;
    numberq[m | n | k | x | y | z] := True;
    s1[u_Plus] := s1 /@ u
    s1[a_?numberq r_] := a s1[r]
    s2[u_Plus] := s2 /@ u
    s2[a_?numberq r_] := a s2[r]
    s3[u_Plus] := s3 /@ u
    s3[a_?numberq r_] := a s3[r]
\end{mmaCell}

\noindent Run the previous cell to send the definitions to the Kernel. Input $\sigma \in W$ in place of `Sigma' (on the first line) with the $s_i$ separated by commas. (ie. $s_1s_2s_3$ would be written as s1, s2, s3)

\begin{mmaCell}[label={In[2]:=},undefined={alpha1,alpha2,alpha3,m,n,k,x,y,z,Sigma}]{Code}
    Simplify[Composition[Simplify, Sigma][
        Distribute[(m + n + k + 3)*alpha1] + 
        Distribute[(m + 2 n + 2 k + 5)*alpha2] + 
        Distribute[(m/2 + n + (3/2)*k + 3)*alpha3]] + 
        Distribute[-(x + y + z + 3)*alpha1] + 
        Distribute[-(x + 2 y + 2 z + 5)*alpha2] + 
        Distribute[-(x/2 + y + (3/2)*z + 3)*alpha3]]
\end{mmaCell}

\section{Python code checking for contradictions of Type II} \label{sec:contradiction-code-1}
This program generates every subset of the remaining $17$ Weyl group elements which can be in an alternation set. It then checks the subset and discards if it contains a Type II contradiction. 
\begin{lstlisting}[language=Python]
import itertools
from updated_alt_sets import currentAltSets2

#getting coefficients
bigToSmall = {
    "A" : {"a", "d", "j"}, "B" : {"b", "d", "j"}, "C" : {"a", "e", "j"},
    "D" : {"a", "d", "l"}, "E" : {"c", "e", "j"}, "F" : {"b", "f", "j"},
    "G" : {"a", "g", "l"}, "H" : {"b", "d", "l"}, "I" : {"a", "e", "o"},
    "J" : {"c", "f", "j"}, "K" : {"b", "h", "l"},
    "L" : {"a", "i", "o"}, "M" : {"b", "f", "p"}, "N" : {"c", "e", "o"},
    "O" : {"a", "g", "q"}, "P" : {"c", "f", "p"}, "Q" : {"a", "i", "q"},
}

#contradictions of type III
zeroToOne = {
    "b":{"a"}, "c":{"a", "b", "f"}, "e":{"d"}, "f":{"d", "e"},
    "g":{"d","e"}, "h":{"d","e","f","g"}, "i":{"d","e","g"}, "l":{"j"},
    "o":{"j","l"}, "p":{"j","l","o"}, "q":{"j","l","o","i"}, "a":set(),
    "d": set(), "j": set(), "p": set(),
}

allElementList = ["A", "B", "C", "D", "E", "F", "G", "H", "I", "J", "K", "L", "M", "N",    "O", "P", "Q"]
allElementSet = set(allElementList)

#turns any imputed set of big letters into small ones
def takeBigToSmall(combination):
    allSmall = set()
    for element in combination:
        allSmall = allSmall.union(bigToSmall[element])
        allSmall.union(bigToSmall[element]))
    return allSmall

#gives small letters that need to be >=0 or there's a type II contradiction
def giveContradictionsII(combination):
    return takeBigToSmall(combination)

#gives small letters that need to be >= 0 or there's a type III contradiction
def giveContradictionsIII(combination):
    allSmall = takeBigToSmall(combination)

    #we are assuming this combination is included
    smallContradictions = set() #these also need to be >=0
    for small in allSmall:
        smallContradictions = smallContradictions.union(zeroToOne[small])
    if {"a","g"}.issubset(allSmall):
        smallContradictions = smallContradictions.union("t")
        smallContradictions.union(zeroToOne[small]))

    return smallContradictions

def giveNeeded(combination): #combination = an alternation set
    #getting togher all small letters that need to be >= 0
    all_contradictions = giveContradictionsII(combination).union(giveContradictionsIII(combination))
    
    #gettin all 3-element subsets of all_contradictions
    contraSubsets = list(itertools.combinations(all_contradictions,3))

    needed = []
    #turning tuples into sets
    contraChangedSubsets = list(set(subset) for subset in contraSubsets)

    #check what big letters need to be in the alt set
    for subset in contraChangedSubsets:
        for element in allElementList:
            if subset == bigToSmall[element] and element not in combination:
                needed.append(element)
    #account for s_0 and o_0
    if all(x in takeBigToSmall(combination) for x in ['o','s']):
        needed.append('K') #because K is a no-longer used letter
    return needed

def removeSubsets(altSets):
    newAltSet = []
    for subset in altSets:
        if len(giveNeeded(subset)) == 0:
            if not(all(x in subset for x in ['A','N']) and 'J' not in subset):
                newAltSet.append(subset)
    print(len(newAltSet))
    return newAltSet

print(removeSubsets(currentAltSets2))
\end{lstlisting}

\section{Alternation sets without contradictions of Type II} \label{sec:alt-sets-I}
\begin{multicols}{3}
\footnotesize
\begin{enumerate}
\item \{\}, \item \{N\}, \item \{J\}, \item \{C\}, \item \{E\}, \item \{F\}, \item \{A\}, \item \{K\}, \item \{B\}, \item \{P\}, \item \{H\}, \item \{O\}, \item \{D\}, \item \{Q\}, \item \{I\}, \item \{G\}, \item \{M\}, \item \{L\}, \item \{N, E\}, \item \{N, K\}, \item \{N, P\}, \item \{N, H\}, \item \{N, I\}, \item \{J, E\}, \item \{J, F\}, \item \{J, A\}, \item \{J, P\}, \item \{J, O\}, \item \{J, Q\}, \item \{J, G\}, \item \{J, L\}, \item \{C, E\}, \item \{C, F\}, \item \{C, A\}, \item \{C, K\}, \item \{C, O\}, \item \{C, Q\}, \item \{C, I\}, \item \{C, G\}, \item \{E, K\}, \item \{E, B\}, \item \{F, K\}, \item \{F, B\}, \item \{F, O\}, \item \{F, Q\}, \item \{F, G\}, \item \{F, M\}, \item \{F, L\}, \item \{A, B\}, \item \{A, O\}, \item \{A, D\}, \item \{A, Q\}, \item \{A, L\}, \item \{K, H\}, \item \{K, Q\}, \item \{K, I\}, \item \{K, G\}, \item \{K, M\}, \item \{K, L\}, \item \{B, H\}, \item \{P, O\}, \item \{P, D\}, \item \{P, Q\}, \item \{P, G\}, \item \{P, M\}, \item \{P, L\}, \item \{H, D\}, \item \{H, M\}, \item \{O, Q\}, \item \{O, I\}, \item \{O, G\}, \item \{O, M\}, \item \{D, Q\}, \item \{D, I\}, \item \{D, G\}, \item \{D, L\}, \item \{Q, M\}, \item \{Q, L\}, \item \{I, G\}, \item \{I, M\}, \item \{I, L\}, \item \{G, M\}, \item \{G, L\}, \item \{M, L\}, \item \{N, J, E\}, \item \{N, E, K\}, \item \{N, E, B\}, \item \{N, K, H\}, \item \{N, K, I\}, \item \{N, P, I\}, \item \{N, P, M\}, \item \{N, O, I\}, \item \{N, D, I\}, \item \{N, I, G\}, \item \{N, I, L\}, \item \{J, C, E\}, \item \{J, E, F\}, \item \{J, E, P\}, \item \{J, F, K\}, \item \{J, F, B\}, \item \{J, F, O\}, \item \{J, F, Q\}, \item \{J, F, G\}, \item \{J, F, L\}, \item \{J, A, P\}, \item \{J, A, O\}, \item \{J, A, D\}, \item \{J, A, Q\}, \item \{J, A, L\}, \item \{J, P, O\}, \item \{J, P, Q\}, \item \{J, P, G\}, \item \{J, P, L\}, \item \{J, O, Q\}, \item \{J, O, G\}, \item \{J, Q, L\}, \item \{J, G, L\}, \item \{C, E, A\}, \item \{C, E, K\}, \item \{C, E, O\}, \item \{C, E, Q\}, \item \{C, E, G\}, \item \{C, F, K\}, \item \{C, F, O\}, \item \{C, F, Q\}, \item \{C, F, I\}, \item \{C, F, G\}, \item \{C, F, M\}, \item \{C, A, B\}, \item \{C, A, O\}, \item \{C, A, D\}, \item \{C, A, Q\}, \item \{C, A, I\}, \item \{C, K, Q\}, \item \{C, K, I\}, \item \{C, K, G\}, \item \{C, O, Q\}, \item \{C, O, I\}, \item \{C, O, G\}, \item \{C, I, G\}, \item \{C, I, L\}, \item \{E, B, H\}, \item \{F, A, B\}, \item \{F, K, Q\}, \item \{F, K, G\}, \item \{F, K, M\}, \item \{F, K, L\}, \item \{F, B, H\}, \item \{F, B, M\}, \item \{F, O, Q\}, \item \{F, O, G\}, \item \{F, O, M\}, \item \{F, Q, M\}, \item \{F, Q, L\}, \item \{F, G, M\}, \item \{F, G, L\}, \item \{F, M, L\}, \item \{A, B, O\}, \item \{A, B, Q\}, \item \{A, B, L\}, \item \{A, O, Q\}, \item \{A, D, Q\}, \item \{A, D, G\}, \item \{A, D, L\}, \item \{A, Q, L\}, \item \{K, B, H\}, \item \{K, P, M\}, \item \{K, H, D\}, \item \{K, H, M\}, \item \{K, O, G\}, \item \{K, Q, M\}, \item \{K, Q, L\}, \item \{K, I, G\}, \item \{K, I, M\}, \item \{K, I, L\}, \item \{K, G, M\}, \item \{K, G, L\}, \item \{K, M, L\}, \item \{P, H, M\}, \item \{P, O, Q\}, \item \{P, O, G\}, \item \{P, O, M\}, \item \{P, D, Q\}, \item \{P, D, G\}, \item \{P, D, L\}, \item \{P, Q, M\}, \item \{P, Q, L\}, \item \{P, G, M\}, \item \{P, G, L\}, \item \{P, M, L\}, \item \{H, D, Q\}, \item \{H, D, I\}, \item \{H, D, G\}, \item \{H, D, M\}, \item \{H, D, L\}, \item \{O, D, G\}, \item \{O, Q, G\}, \item \{O, Q, M\}, \item \{O, Q, L\}, \item \{O, I, G\}, \item \{O, I, M\}, \item \{O, G, M\}, \item \{D, Q, L\}, \item \{D, I, G\}, \item \{D, I, L\}, \item \{D, G, L\}, \item \{Q, I, L\}, \item \{Q, M, L\}, \item \{I, G, M\}, \item \{I, G, L\}, \item \{I, M, L\}, \item \{G, M, L\}, \item \{N, J, E, F\}, \item \{N, J, E, P\}, \item \{N, C, E, I\}, \item \{N, E, B, H\}, \item \{N, K, P, M\}, \item \{N, K, I, G\}, \item \{N, K, I, L\}, \item \{N, P, H, M\}, \item \{N, P, O, I\}, \item \{N, P, D, I\}, \item \{N, P, I, G\}, \item \{N, P, I, M\}, \item \{N, P, I, L\}, \item \{N, H, D, I\}, \item \{N, O, I, G\}, \item \{N, D, I, G\}, \item \{N, D, I, L\}, \item \{N, Q, I, L\}, \item \{N, I, G, L\}, \item \{J, C, E, F\}, \item \{J, C, E, A\}, \item \{J, C, E, P\}, \item \{J, C, E, O\}, \item \{J, C, E, Q\}, \item \{J, C, E, G\}, \item \{J, E, F, K\}, \item \{J, E, F, B\}, \item \{J, F, A, B\}, \item \{J, F, K, Q\}, \item \{J, F, K, G\}, \item \{J, F, K, L\}, \item \{J, F, B, H\}, \item \{J, F, P, M\}, \item \{J, F, O, Q\}, \item \{J, F, O, G\}, \item \{J, F, Q, L\}, \item \{J, F, G, L\}, \item \{J, A, P, O\}, \item \{J, A, P, D\}, \item \{J, A, P, Q\}, \item \{J, A, P, L\}, \item \{J, A, O, Q\}, \item \{J, A, D, Q\}, \item \{J, A, D, G\}, \item \{J, A, D, L\}, \item \{J, A, Q, L\}, \item \{J, P, O, Q\}, \item \{J, P, O, G\}, \item \{J, P, Q, L\}, \item \{J, P, G, L\}, \item \{J, O, Q, G\}, \item \{J, O, Q, L\}, \item \{C, E, A, B\}, \item \{C, E, A, O\}, \item \{C, E, A, D\}, \item \{C, E, A, Q\}, \item \{C, E, K, Q\}, \item \{C, E, K, G\}, \item \{C, E, O, Q\}, \item \{C, E, O, G\}, \item \{C, F, A, B\}, \item \{C, F, K, Q\}, \item \{C, F, K, I\}, \item \{C, F, K, G\}, \item \{C, F, K, M\}, \item \{C, F, O, Q\}, \item \{C, F, O, I\}, \item \{C, F, O, G\}, \item \{C, F, O, M\}, \item \{C, F, Q, M\}, \item \{C, F, I, G\}, \item \{C, F, I, M\}, \item \{C, F, I, L\}, \item \{C, F, G, M\}, \item \{C, A, B, O\}, \item \{C, A, B, Q\}, \item \{C, A, B, I\}, \item \{C, A, O, Q\}, \item \{C, A, O, I\}, \item \{C, A, D, Q\}, \item \{C, A, D, I\}, \item \{C, A, D, G\}, \item \{C, A, I, L\}, \item \{C, K, O, G\}, \item \{C, K, I, G\}, \item \{C, K, I, L\}, \item \{C, O, Q, G\}, \item \{C, O, I, G\}, \item \{C, Q, I, L\}, \item \{C, I, G, L\}, \item \{E, K, B, H\}, \item \{F, A, B, O\}, \item \{F, A, B, Q\}, \item \{F, A, B, M\}, \item \{F, A, B, L\}, \item \{F, K, B, H\}, \item \{F, K, O, G\}, \item \{F, K, Q, M\}, \item \{F, K, Q, L\}, \item \{F, K, G, M\}, \item \{F, K, G, L\}, \item \{F, K, M, L\}, \item \{F, B, H, M\}, \item \{F, O, Q, G\}, \item \{F, O, Q, M\}, \item \{F, O, Q, L\}, \item \{F, O, G, M\}, \item \{F, Q, M, L\}, \item \{F, G, M, L\}, \item \{A, B, H, D\}, \item \{A, B, O, Q\}, \item \{A, B, Q, L\}, \item \{A, O, D, G\}, \item \{A, O, Q, L\}, \item \{A, D, Q, L\}, \item \{A, D, G, L\}, \item \{K, P, H, M\}, \item \{K, P, Q, M\}, \item \{K, P, G, M\}, \item \{K, P, M, L\}, \item \{K, H, D, Q\}, \item \{K, H, D, I\}, \item \{K, H, D, G\}, \item \{K, H, D, M\}, \item \{K, H, D, L\}, \item \{K, O, Q, G\}, \item \{K, O, I, G\}, \item \{K, O, G, M\}, \item \{K, Q, I, L\}, \item \{K, Q, M, L\}, \item \{K, I, G, M\}, \item \{K, I, G, L\}, \item \{K, I, M, L\}, \item \{K, G, M, L\}, \item \{P, H, D, M\}, \item \{P, O, D, G\}, \item \{P, O, Q, G\}, \item \{P, O, Q, M\}, \item \{P, O, Q, L\}, \item \{P, O, G, M\}, \item \{P, D, Q, L\}, \item \{P, D, G, L\}, \item \{P, Q, M, L\}, \item \{P, G, M, L\}, \item \{H, O, D, G\}, \item \{H, D, Q, M\}, \item \{H, D, Q, L\}, \item \{H, D, I, G\}, \item \{H, D, I, M\}, \item \{H, D, I, L\}, \item \{H, D, G, M\}, \item \{H, D, G, L\}, \item \{H, D, M, L\}, \item \{O, D, Q, G\}, \item \{O, D, I, G\}, \item \{O, Q, I, L\}, \item \{O, Q, G, M\}, \item \{O, Q, G, L\}, \item \{O, Q, M, L\}, \item \{O, I, G, M\}, \item \{D, Q, I, L\}, \item \{D, I, G, L\}, \item \{Q, I, M, L\}, \item \{I, G, M, L\}, \item \{N, J, C, E, I\}, \item \{N, J, E, F, K\}, \item \{N, J, E, F, B\}, \item \{N, C, E, A, I\}, \item \{N, C, E, K, I\}, \item \{N, C, E, O, I\}, \item \{N, C, E, I, G\}, \item \{N, C, E, I, L\}, \item \{N, E, K, B, H\}, \item \{N, K, P, H, M\}, \item \{N, K, P, I, M\}, \item \{N, K, H, D, I\}, \item \{N, K, O, I, G\}, \item \{N, K, Q, I, L\}, \item \{N, K, I, G, L\}, \item \{N, P, O, I, G\}, \item \{N, P, O, I, M\}, \item \{N, P, D, I, G\}, \item \{N, P, D, I, L\}, \item \{N, P, Q, I, L\}, \item \{N, P, I, G, M\}, \item \{N, P, I, G, L\}, \item \{N, P, I, M, L\}, \item \{N, H, D, I, G\}, \item \{N, H, D, I, L\}, \item \{N, O, D, I, G\}, \item \{N, O, Q, I, L\}, \item \{N, D, Q, I, L\}, \item \{N, D, I, G, L\}, \item \{J, C, E, F, K\}, \item \{J, C, E, F, O\}, \item \{J, C, E, F, Q\}, \item \{J, C, E, F, G\}, \item \{J, C, E, A, P\}, \item \{J, C, E, A, O\}, \item \{J, C, E, A, D\}, \item \{J, C, E, A, Q\}, \item \{J, C, E, P, O\}, \item \{J, C, E, P, Q\}, \item \{J, C, E, P, G\}, \item \{J, C, E, O, Q\}, \item \{J, C, E, O, G\}, \item \{J, E, F, B, H\}, \item \{J, E, F, P, M\}, \item \{J, F, A, B, O\}, \item \{J, F, A, B, Q\}, \item \{J, F, A, B, L\}, \item \{J, F, K, B, H\}, \item \{J, F, K, P, M\}, \item \{J, F, K, O, G\}, \item \{J, F, K, Q, L\}, \item \{J, F, K, G, L\}, \item \{J, F, B, P, M\}, \item \{J, F, P, O, M\}, \item \{J, F, P, Q, M\}, \item \{J, F, P, G, M\}, \item \{J, F, P, M, L\}, \item \{J, F, O, Q, G\}, \item \{J, F, O, Q, L\}, \item \{J, A, P, O, Q\}, \item \{J, A, P, D, Q\}, \item \{J, A, P, D, G\}, \item \{J, A, P, D, L\}, \item \{J, A, P, Q, L\}, \item \{J, A, O, D, G\}, \item \{J, A, O, Q, L\}, \item \{J, A, D, Q, L\}, \item \{J, A, D, G, L\}, \item \{J, P, O, Q, G\}, \item \{J, P, O, Q, L\}, \item \{J, O, Q, G, L\}, \item \{C, E, A, B, O\}, \item \{C, E, A, B, Q\}, \item \{C, E, A, O, Q\}, \item \{C, E, A, D, Q\}, \item \{C, E, A, D, G\}, \item \{C, E, K, O, G\}, \item \{C, E, O, Q, G\}, \item \{C, F, A, B, O\}, \item \{C, F, A, B, Q\}, \item \{C, F, A, B, I\}, \item \{C, F, A, B, M\}, \item \{C, F, K, O, G\}, \item \{C, F, K, Q, M\}, \item \{C, F, K, I, G\}, \item \{C, F, K, I, M\}, \item \{C, F, K, I, L\}, \item \{C, F, K, G, M\}, \item \{C, F, O, Q, G\}, \item \{C, F, O, Q, M\}, \item \{C, F, O, I, G\}, \item \{C, F, O, I, M\}, \item \{C, F, O, G, M\}, \item \{C, F, Q, I, L\}, \item \{C, F, I, G, M\}, \item \{C, F, I, G, L\}, \item \{C, F, I, M, L\}, \item \{C, A, B, H, D\}, \item \{C, A, B, O, Q\}, \item \{C, A, B, O, I\}, \item \{C, A, B, I, L\}, \item \{C, A, O, D, G\}, \item \{C, A, D, I, G\}, \item \{C, A, D, I, L\}, \item \{C, A, Q, I, L\}, \item \{C, K, O, Q, G\}, \item \{C, K, O, I, G\}, \item \{C, K, Q, I, L\}, \item \{C, K, I, G, L\}, \item \{C, O, Q, I, L\}, \item \{F, A, B, H, D\}, \item \{F, A, B, O, Q\}, \item \{F, A, B, O, M\}, \item \{F, A, B, Q, M\}, \item \{F, A, B, Q, L\}, \item \{F, A, B, M, L\}, \item \{F, K, B, H, M\}, \item \{F, K, O, Q, G\}, \item \{F, K, O, G, M\}, \item \{F, K, Q, M, L\}, \item \{F, K, G, M, L\}, \item \{F, O, Q, G, M\}, \item \{F, O, Q, G, L\}, \item \{F, O, Q, M, L\}, \item \{A, K, B, H, D\}, \item \{A, B, H, D, Q\}, \item \{A, B, H, D, G\}, \item \{A, B, H, D, L\}, \item \{A, B, O, Q, L\}, \item \{A, O, D, Q, G\}, \item \{K, P, H, D, M\}, \item \{K, P, O, G, M\}, \item \{K, P, Q, M, L\}, \item \{K, P, G, M, L\}, \item \{K, H, O, D, G\}, \item \{K, H, D, Q, M\}, \item \{K, H, D, Q, L\}, \item \{K, H, D, I, G\}, \item \{K, H, D, I, M\}, \item \{K, H, D, I, L\}, \item \{K, H, D, G, M\}, \item \{K, H, D, G, L\}, \item \{K, H, D, M, L\}, \item \{K, O, Q, G, M\}, \item \{K, O, Q, G, L\}, \item \{K, O, I, G, M\}, \item \{K, Q, I, M, L\}, \item \{K, I, G, M, L\}, \item \{P, H, D, Q, M\}, \item \{P, H, D, G, M\}, \item \{P, H, D, M, L\}, \item \{P, O, D, Q, G\}, \item \{P, O, Q, G, M\}, \item \{P, O, Q, G, L\}, \item \{P, O, Q, M, L\}, \item \{H, O, D, Q, G\}, \item \{H, O, D, I, G\}, \item \{H, O, D, G, M\}, \item \{H, D, Q, I, L\}, \item \{H, D, Q, M, L\}, \item \{H, D, I, G, M\}, \item \{H, D, I, G, L\}, \item \{H, D, I, M, L\}, \item \{H, D, G, M, L\}, \item \{O, D, Q, G, L\}, \item \{O, Q, I, G, L\}, \item \{O, Q, I, M, L\}, \item \{O, Q, G, M, L\}, \item \{N, J, C, E, F, I\}, \item \{N, J, C, E, A, I\}, \item \{N, J, C, E, P, I\}, \item \{N, J, C, E, O, I\}, \item \{N, J, C, E, I, G\}, \item \{N, J, C, E, I, L\}, \item \{N, J, E, F, B, H\}, \item \{N, J, E, F, P, M\}, \item \{N, C, E, A, B, I\}, \item \{N, C, E, A, O, I\}, \item \{N, C, E, A, D, I\}, \item \{N, C, E, A, I, L\}, \item \{N, C, E, K, I, G\}, \item \{N, C, E, K, I, L\}, \item \{N, C, E, O, I, G\}, \item \{N, C, E, Q, I, L\}, \item \{N, C, E, I, G, L\}, \item \{N, K, P, I, G, M\}, \item \{N, K, P, I, M, L\}, \item \{N, K, H, D, I, G\}, \item \{N, K, H, D, I, L\}, \item \{N, P, H, D, I, M\}, \item \{N, P, O, D, I, G\}, \item \{N, P, O, Q, I, L\}, \item \{N, P, O, I, G, M\}, \item \{N, P, D, Q, I, L\}, \item \{N, P, D, I, G, L\}, \item \{N, P, Q, I, M, L\}, \item \{N, P, I, G, M, L\}, \item \{N, H, O, D, I, G\}, \item \{N, H, D, Q, I, L\}, \item \{N, H, D, I, G, L\}, \item \{N, O, Q, I, G, L\}, \item \{J, C, E, F, A, B\}, \item \{J, C, E, F, K, Q\}, \item \{J, C, E, F, K, G\}, \item \{J, C, E, F, P, M\}, \item \{J, C, E, F, O, Q\}, \item \{J, C, E, F, O, G\}, \item \{J, C, E, A, P, O\}, \item \{J, C, E, A, P, D\}, \item \{J, C, E, A, P, Q\}, \item \{J, C, E, A, O, Q\}, \item \{J, C, E, A, D, Q\}, \item \{J, C, E, A, D, G\}, \item \{J, C, E, P, O, Q\}, \item \{J, C, E, P, O, G\}, \item \{J, C, E, O, Q, G\}, \item \{J, E, F, K, B, H\}, \item \{J, E, F, K, P, M\}, \item \{J, E, F, B, P, M\}, \item \{J, F, A, B, P, M\}, \item \{J, F, A, B, H, D\}, \item \{J, F, A, B, O, Q\}, \item \{J, F, A, B, Q, L\}, \item \{J, F, K, P, Q, M\}, \item \{J, F, K, P, G, M\}, \item \{J, F, K, P, M, L\}, \item \{J, F, K, O, Q, G\}, \item \{J, F, B, P, H, M\}, \item \{J, F, P, O, Q, M\}, \item \{J, F, P, O, G, M\}, \item \{J, F, P, Q, M, L\}, \item \{J, F, P, G, M, L\}, \item \{J, F, O, Q, G, L\}, \item \{J, A, P, O, D, G\}, \item \{J, A, P, O, Q, L\}, \item \{J, A, P, D, Q, L\}, \item \{J, A, P, D, G, L\}, \item \{J, A, O, D, Q, G\}, \item \{J, P, O, Q, G, L\}, \item \{C, E, A, B, H, D\}, \item \{C, E, A, B, O, Q\}, \item \{C, E, A, O, D, G\}, \item \{C, E, K, O, Q, G\}, \item \{C, F, A, B, H, D\}, \item \{C, F, A, B, O, Q\}, \item \{C, F, A, B, O, I\}, \item \{C, F, A, B, O, M\}, \item \{C, F, A, B, Q, M\}, \item \{C, F, A, B, I, M\}, \item \{C, F, A, B, I, L\}, \item \{C, F, K, O, Q, G\}, \item \{C, F, K, O, I, G\}, \item \{C, F, K, O, G, M\}, \item \{C, F, K, Q, I, L\}, \item \{C, F, K, I, G, M\}, \item \{C, F, K, I, G, L\}, \item \{C, F, K, I, M, L\}, \item \{C, F, O, Q, I, L\}, \item \{C, F, O, Q, G, M\}, \item \{C, F, O, I, G, M\}, \item \{C, F, Q, I, M, L\}, \item \{C, F, I, G, M, L\}, \item \{C, A, K, B, H, D\}, \item \{C, A, B, H, D, Q\}, \item \{C, A, B, H, D, I\}, \item \{C, A, B, H, D, G\}, \item \{C, A, B, Q, I, L\}, \item \{C, A, O, D, Q, G\}, \item \{C, A, O, D, I, G\}, \item \{C, A, O, Q, I, L\}, \item \{C, A, D, Q, I, L\}, \item \{C, A, D, I, G, L\}, \item \{C, O, Q, I, G, L\}, \item \{F, A, K, B, H, D\}, \item \{F, A, B, H, D, Q\}, \item \{F, A, B, H, D, G\}, \item \{F, A, B, H, D, M\}, \item \{F, A, B, H, D, L\}, \item \{F, A, B, O, Q, M\}, \item \{F, A, B, O, Q, L\}, \item \{F, A, B, Q, M, L\}, \item \{F, K, O, Q, G, M\}, \item \{F, K, O, Q, G, L\}, \item \{F, O, Q, G, M, L\}, \item \{A, K, B, H, D, Q\}, \item \{A, K, B, H, D, G\}, \item \{A, K, B, H, D, L\}, \item \{A, B, H, O, D, G\}, \item \{A, B, H, D, Q, L\}, \item \{A, B, H, D, G, L\}, \item \{A, O, D, Q, G, L\}, \item \{K, P, H, D, Q, M\}, \item \{K, P, H, D, G, M\}, \item \{K, P, H, D, M, L\}, \item \{K, P, O, Q, G, M\}, \item \{K, H, O, D, Q, G\}, \item \{K, H, O, D, I, G\}, \item \{K, H, O, D, G, M\}, \item \{K, H, D, Q, I, L\}, \item \{K, H, D, Q, M, L\}, \item \{K, H, D, I, G, M\}, \item \{K, H, D, I, G, L\}, \item \{K, H, D, I, M, L\}, \item \{K, H, D, G, M, L\}, \item \{K, O, Q, I, G, L\}, \item \{K, O, Q, G, M, L\}, \item \{P, H, O, D, G, M\}, \item \{P, H, D, Q, M, L\}, \item \{P, H, D, G, M, L\}, \item \{P, O, D, Q, G, L\}, \item \{P, O, Q, G, M, L\}, \item \{H, O, D, Q, G, M\}, \item \{H, O, D, Q, G, L\}, \item \{H, O, D, I, G, M\}, \item \{H, D, Q, I, M, L\}, \item \{H, D, I, G, M, L\}, \item \{O, D, Q, I, G, L\}, \item \{O, Q, I, G, M, L\}, \item \{N, J, C, E, F, K, I\}, \item \{N, J, C, E, F, O, I\}, \item \{N, J, C, E, F, I, G\}, \item \{N, J, C, E, F, I, L\}, \item \{N, J, C, E, A, P, I\}, \item \{N, J, C, E, A, O, I\}, \item \{N, J, C, E, A, D, I\}, \item \{N, J, C, E, A, I, L\}, \item \{N, J, C, E, P, O, I\}, \item \{N, J, C, E, P, I, G\}, \item \{N, J, C, E, P, I, L\}, \item \{N, J, C, E, O, I, G\}, \item \{N, J, C, E, Q, I, L\}, \item \{N, J, C, E, I, G, L\}, \item \{N, J, E, F, K, B, H\}, \item \{N, J, E, F, K, P, M\}, \item \{N, J, E, F, B, P, M\}, \item \{N, C, E, A, B, O, I\}, \item \{N, C, E, A, B, I, L\}, \item \{N, C, E, A, D, I, G\}, \item \{N, C, E, A, D, I, L\}, \item \{N, C, E, A, Q, I, L\}, \item \{N, C, E, K, O, I, G\}, \item \{N, C, E, K, Q, I, L\}, \item \{N, C, E, K, I, G, L\}, \item \{N, C, E, O, Q, I, L\}, \item \{N, K, P, H, D, I, M\}, \item \{N, K, P, O, I, G, M\}, \item \{N, K, P, Q, I, M, L\}, \item \{N, K, P, I, G, M, L\}, \item \{N, K, H, O, D, I, G\}, \item \{N, K, H, D, Q, I, L\}, \item \{N, K, H, D, I, G, L\}, \item \{N, K, O, Q, I, G, L\}, \item \{N, P, H, D, I, G, M\}, \item \{N, P, H, D, I, M, L\}, \item \{N, P, O, Q, I, G, L\}, \item \{N, P, O, Q, I, M, L\}, \item \{N, O, D, Q, I, G, L\}, \item \{J, C, E, F, A, B, O\}, \item \{J, C, E, F, A, B, Q\}, \item \{J, C, E, F, K, P, M\}, \item \{J, C, E, F, K, O, G\}, \item \{J, C, E, F, P, O, M\}, \item \{J, C, E, F, P, Q, M\}, \item \{J, C, E, F, P, G, M\}, \item \{J, C, E, F, O, Q, G\}, \item \{J, C, E, A, P, O, Q\}, \item \{J, C, E, A, P, D, Q\}, \item \{J, C, E, A, P, D, G\}, \item \{J, C, E, A, O, D, G\}, \item \{J, C, E, P, O, Q, G\}, \item \{J, E, F, B, P, H, M\}, \item \{J, F, A, K, B, H, D\}, \item \{J, F, A, B, P, O, M\}, \item \{J, F, A, B, P, Q, M\}, \item \{J, F, A, B, P, M, L\}, \item \{J, F, A, B, H, D, Q\}, \item \{J, F, A, B, H, D, G\}, \item \{J, F, A, B, H, D, L\}, \item \{J, F, A, B, O, Q, L\}, \item \{J, F, K, B, P, H, M\}, \item \{J, F, K, P, O, G, M\}, \item \{J, F, K, P, Q, M, L\}, \item \{J, F, K, P, G, M, L\}, \item \{J, F, K, O, Q, G, L\}, \item \{J, F, P, O, Q, G, M\}, \item \{J, F, P, O, Q, M, L\}, \item \{J, A, P, O, D, Q, G\}, \item \{J, A, O, D, Q, G, L\}, \item \{C, E, A, K, B, H, D\}, \item \{C, E, A, B, H, D, Q\}, \item \{C, E, A, B, H, D, G\}, \item \{C, E, A, O, D, Q, G\}, \item \{C, F, A, K, B, H, D\}, \item \{C, F, A, B, H, D, Q\}, \item \{C, F, A, B, H, D, I\}, \item \{C, F, A, B, H, D, G\}, \item \{C, F, A, B, H, D, M\}, \item \{C, F, A, B, O, Q, M\}, \item \{C, F, A, B, O, I, M\}, \item \{C, F, A, B, Q, I, L\}, \item \{C, F, A, B, I, M, L\}, \item \{C, F, K, O, Q, G, M\}, \item \{C, F, K, O, I, G, M\}, \item \{C, F, K, Q, I, M, L\}, \item \{C, F, K, I, G, M, L\}, \item \{C, F, O, Q, I, G, L\}, \item \{C, F, O, Q, I, M, L\}, \item \{C, A, K, B, H, D, Q\}, \item \{C, A, K, B, H, D, I\}, \item \{C, A, K, B, H, D, G\}, \item \{C, A, B, H, O, D, G\}, \item \{C, A, B, H, D, I, G\}, \item \{C, A, B, H, D, I, L\}, \item \{C, A, B, O, Q, I, L\}, \item \{C, K, O, Q, I, G, L\}, \item \{F, A, K, B, H, D, Q\}, \item \{F, A, K, B, H, D, G\}, \item \{F, A, K, B, H, D, M\}, \item \{F, A, K, B, H, D, L\}, \item \{F, A, B, H, O, D, G\}, \item \{F, A, B, H, D, Q, M\}, \item \{F, A, B, H, D, Q, L\}, \item \{F, A, B, H, D, G, M\}, \item \{F, A, B, H, D, G, L\}, \item \{F, A, B, H, D, M, L\}, \item \{F, A, B, O, Q, M, L\}, \item \{F, K, O, Q, G, M, L\}, \item \{A, K, B, H, O, D, G\}, \item \{A, K, B, H, D, Q, L\}, \item \{A, K, B, H, D, G, L\}, \item \{A, B, H, O, D, Q, G\}, \item \{K, P, H, O, D, G, M\}, \item \{K, P, H, D, Q, M, L\}, \item \{K, P, H, D, G, M, L\}, \item \{K, P, O, Q, G, M, L\}, \item \{K, H, O, D, Q, G, M\}, \item \{K, H, O, D, Q, G, L\}, \item \{K, H, O, D, I, G, M\}, \item \{K, H, D, Q, I, M, L\}, \item \{K, H, D, I, G, M, L\}, \item \{K, O, Q, I, G, M, L\}, \item \{P, H, O, D, Q, G, M\}, \item \{H, O, D, Q, I, G, L\}, \item \{H, O, D, Q, G, M, L\}, \item \{N, J, C, E, F, A, B, I\}, \item \{N, J, C, E, F, K, I, G\}, \item \{N, J, C, E, F, K, I, L\}, \item \{N, J, C, E, F, P, I, M\}, \item \{N, J, C, E, F, O, I, G\}, \item \{N, J, C, E, F, Q, I, L\}, \item \{N, J, C, E, F, I, G, L\}, \item \{N, J, C, E, A, P, O, I\}, \item \{N, J, C, E, A, P, D, I\}, \item \{N, J, C, E, A, P, I, L\}, \item \{N, J, C, E, A, D, I, G\}, \item \{N, J, C, E, A, D, I, L\}, \item \{N, J, C, E, A, Q, I, L\}, \item \{N, J, C, E, P, O, I, G\}, \item \{N, J, C, E, P, Q, I, L\}, \item \{N, J, C, E, P, I, G, L\}, \item \{N, J, C, E, O, Q, I, L\}, \item \{N, J, E, F, B, P, H, M\}, \item \{N, C, E, A, B, H, D, I\}, \item \{N, C, E, A, B, Q, I, L\}, \item \{N, C, E, A, O, D, I, G\}, \item \{N, C, E, A, O, Q, I, L\}, \item \{N, C, E, A, D, Q, I, L\}, \item \{N, C, E, A, D, I, G, L\}, \item \{N, C, E, O, Q, I, G, L\}, \item \{N, K, P, H, D, I, G, M\}, \item \{N, K, P, H, D, I, M, L\}, \item \{N, P, H, O, D, I, G, M\}, \item \{N, P, H, D, Q, I, M, L\}, \item \{N, P, H, D, I, G, M, L\}, \item \{N, P, O, D, Q, I, G, L\}, \item \{N, P, O, Q, I, G, M, L\}, \item \{N, H, O, D, Q, I, G, L\}, \item \{J, C, E, F, A, B, P, M\}, \item \{J, C, E, F, A, B, H, D\}, \item \{J, C, E, F, A, B, O, Q\}, \item \{J, C, E, F, K, P, Q, M\}, \item \{J, C, E, F, K, P, G, M\}, \item \{J, C, E, F, K, O, Q, G\}, \item \{J, C, E, F, P, O, Q, M\}, \item \{J, C, E, F, P, O, G, M\}, \item \{J, C, E, A, P, O, D, G\}, \item \{J, C, E, A, O, D, Q, G\}, \item \{J, E, F, K, B, P, H, M\}, \item \{J, F, A, K, B, H, D, Q\}, \item \{J, F, A, K, B, H, D, G\}, \item \{J, F, A, K, B, H, D, L\}, \item \{J, F, A, B, P, H, D, M\}, \item \{J, F, A, B, P, O, Q, M\}, \item \{J, F, A, B, P, Q, M, L\}, \item \{J, F, A, B, H, O, D, G\}, \item \{J, F, A, B, H, D, Q, L\}, \item \{J, F, A, B, H, D, G, L\}, \item \{J, F, K, P, O, Q, G, M\}, \item \{J, F, P, O, Q, G, M, L\}, \item \{J, A, P, O, D, Q, G, L\}, \item \{C, E, A, K, B, H, D, Q\}, \item \{C, E, A, K, B, H, D, G\}, \item \{C, E, A, B, H, O, D, G\}, \item \{C, F, A, K, B, H, D, Q\}, \item \{C, F, A, K, B, H, D, I\}, \item \{C, F, A, K, B, H, D, G\}, \item \{C, F, A, K, B, H, D, M\}, \item \{C, F, A, B, H, O, D, G\}, \item \{C, F, A, B, H, D, Q, M\}, \item \{C, F, A, B, H, D, I, G\}, \item \{C, F, A, B, H, D, I, M\}, \item \{C, F, A, B, H, D, I, L\}, \item \{C, F, A, B, H, D, G, M\}, \item \{C, F, A, B, O, Q, I, L\}, \item \{C, F, A, B, Q, I, M, L\}, \item \{C, F, K, O, Q, I, G, L\}, \item \{C, F, O, Q, I, G, M, L\}, \item \{C, A, K, B, H, O, D, G\}, \item \{C, A, K, B, H, D, I, G\}, \item \{C, A, K, B, H, D, I, L\}, \item \{C, A, B, H, O, D, Q, G\}, \item \{C, A, B, H, O, D, I, G\}, \item \{C, A, B, H, D, Q, I, L\}, \item \{C, A, B, H, D, I, G, L\}, \item \{C, A, O, D, Q, I, G, L\}, \item \{F, A, K, B, H, O, D, G\}, \item \{F, A, K, B, H, D, Q, M\}, \item \{F, A, K, B, H, D, Q, L\}, \item \{F, A, K, B, H, D, G, M\}, \item \{F, A, K, B, H, D, G, L\}, \item \{F, A, K, B, H, D, M, L\}, \item \{F, A, B, H, O, D, Q, G\}, \item \{F, A, B, H, O, D, G, M\}, \item \{F, A, B, H, D, Q, M, L\}, \item \{F, A, B, H, D, G, M, L\}, \item \{A, K, B, H, O, D, Q, G\}, \item \{A, B, H, O, D, Q, G, L\}, \item \{K, P, H, O, D, Q, G, M\}, \item \{K, H, O, D, Q, I, G, L\}, \item \{K, H, O, D, Q, G, M, L\}, \item \{P, H, O, D, Q, G, M, L\}, \item \{H, O, D, Q, I, G, M, L\}, \item \{N, J, C, E, F, A, B, O, I\}, \item \{N, J, C, E, F, A, B, I, L\}, \item \{N, J, C, E, F, K, P, I, M\}, \item \{N, J, C, E, F, K, O, I, G\}, \item \{N, J, C, E, F, K, Q, I, L\}, \item \{N, J, C, E, F, K, I, G, L\}, \item \{N, J, C, E, F, P, O, I, M\}, \item \{N, J, C, E, F, P, I, G, M\}, \item \{N, J, C, E, F, P, I, M, L\}, \item \{N, J, C, E, F, O, Q, I, L\}, \item \{N, J, C, E, A, P, D, I, G\}, \item \{N, J, C, E, A, P, D, I, L\}, \item \{N, J, C, E, A, P, Q, I, L\}, \item \{N, J, C, E, A, O, D, I, G\}, \item \{N, J, C, E, A, O, Q, I, L\}, \item \{N, J, C, E, A, D, Q, I, L\}, \item \{N, J, C, E, A, D, I, G, L\}, \item \{N, J, C, E, P, O, Q, I, L\}, \item \{N, J, C, E, O, Q, I, G, L\}, \item \{N, J, E, F, K, B, P, H, M\}, \item \{N, C, E, A, K, B, H, D, I\}, \item \{N, C, E, A, B, H, D, I, G\}, \item \{N, C, E, A, B, H, D, I, L\}, \item \{N, C, E, A, B, O, Q, I, L\}, \item \{N, C, E, K, O, Q, I, G, L\}, \item \{N, K, P, H, O, D, I, G, M\}, \item \{N, K, P, H, D, Q, I, M, L\}, \item \{N, K, P, H, D, I, G, M, L\}, \item \{N, K, P, O, Q, I, G, M, L\}, \item \{N, K, H, O, D, Q, I, G, L\}, \item \{J, C, E, F, A, K, B, H, D\}, \item \{J, C, E, F, A, B, P, O, M\}, \item \{J, C, E, F, A, B, P, Q, M\}, \item \{J, C, E, F, A, B, H, D, Q\}, \item \{J, C, E, F, A, B, H, D, G\}, \item \{J, C, E, F, K, P, O, G, M\}, \item \{J, C, E, F, P, O, Q, G, M\}, \item \{J, C, E, A, P, O, D, Q, G\}, \item \{J, F, A, K, B, P, H, D, M\}, \item \{J, F, A, K, B, H, O, D, G\}, \item \{J, F, A, K, B, H, D, Q, L\}, \item \{J, F, A, K, B, H, D, G, L\}, \item \{J, F, A, B, P, H, D, Q, M\}, \item \{J, F, A, B, P, H, D, G, M\}, \item \{J, F, A, B, P, H, D, M, L\}, \item \{J, F, A, B, P, O, Q, M, L\}, \item \{J, F, A, B, H, O, D, Q, G\}, \item \{J, F, K, P, O, Q, G, M, L\}, \item \{C, E, A, K, B, H, O, D, G\}, \item \{C, E, A, B, H, O, D, Q, G\}, \item \{C, F, A, K, B, H, O, D, G\}, \item \{C, F, A, K, B, H, D, Q, M\}, \item \{C, F, A, K, B, H, D, I, G\}, \item \{C, F, A, K, B, H, D, I, M\}, \item \{C, F, A, K, B, H, D, I, L\}, \item \{C, F, A, K, B, H, D, G, M\}, \item \{C, F, A, B, H, O, D, Q, G\}, \item \{C, F, A, B, H, O, D, I, G\}, \item \{C, F, A, B, H, O, D, G, M\}, \item \{C, F, A, B, H, D, Q, I, L\}, \item \{C, F, A, B, H, D, I, G, M\}, \item \{C, F, A, B, H, D, I, G, L\}, \item \{C, F, A, B, H, D, I, M, L\}, \item \{C, F, A, B, O, Q, I, M, L\}, \item \{C, F, K, O, Q, I, G, M, L\}, \item \{C, A, K, B, H, O, D, Q, G\}, \item \{C, A, K, B, H, O, D, I, G\}, \item \{C, A, K, B, H, D, Q, I, L\}, \item \{C, A, K, B, H, D, I, G, L\}, \item \{F, A, K, B, H, O, D, Q, G\}, \item \{F, A, K, B, H, O, D, G, M\}, \item \{F, A, K, B, H, D, Q, M, L\}, \item \{F, A, K, B, H, D, G, M, L\}, \item \{F, A, B, H, O, D, Q, G, M\}, \item \{F, A, B, H, O, D, Q, G, L\}, \item \{A, K, B, H, O, D, Q, G, L\}, \item \{K, P, H, O, D, Q, G, M, L\}, \item \{K, H, O, D, Q, I, G, M, L\}, \item \{N, J, C, E, F, A, B, P, I, M\}, \item \{N, J, C, E, F, A, B, H, D, I\}, \item \{N, J, C, E, F, A, B, Q, I, L\}, \item \{N, J, C, E, F, K, P, I, G, M\}, \item \{N, J, C, E, F, K, P, I, M, L\}, \item \{N, J, C, E, F, P, O, I, G, M\}, \item \{N, J, C, E, F, P, Q, I, M, L\}, \item \{N, J, C, E, F, P, I, G, M, L\}, \item \{N, J, C, E, F, O, Q, I, G, L\}, \item \{N, J, C, E, A, P, O, D, I, G\}, \item \{N, J, C, E, A, P, O, Q, I, L\}, \item \{N, J, C, E, A, P, D, Q, I, L\}, \item \{N, J, C, E, A, P, D, I, G, L\}, \item \{N, J, C, E, P, O, Q, I, G, L\}, \item \{N, C, E, A, K, B, H, D, I, G\}, \item \{N, C, E, A, K, B, H, D, I, L\}, \item \{N, C, E, A, B, H, O, D, I, G\}, \item \{N, C, E, A, B, H, D, Q, I, L\}, \item \{N, C, E, A, B, H, D, I, G, L\}, \item \{N, C, E, A, O, D, Q, I, G, L\}, \item \{N, P, H, O, D, Q, I, G, M, L\}, \item \{J, C, E, F, A, K, B, H, D, Q\}, \item \{J, C, E, F, A, K, B, H, D, G\}, \item \{J, C, E, F, A, B, P, H, D, M\}, \item \{J, C, E, F, A, B, P, O, Q, M\}, \item \{J, C, E, F, A, B, H, O, D, G\}, \item \{J, C, E, F, K, P, O, Q, G, M\}, \item \{J, F, A, K, B, P, H, D, Q, M\}, \item \{J, F, A, K, B, P, H, D, G, M\}, \item \{J, F, A, K, B, P, H, D, M, L\}, \item \{J, F, A, K, B, H, O, D, Q, G\}, \item \{J, F, A, B, P, H, O, D, G, M\}, \item \{J, F, A, B, P, H, D, Q, M, L\}, \item \{J, F, A, B, P, H, D, G, M, L\}, \item \{J, F, A, B, H, O, D, Q, G, L\}, \item \{C, E, A, K, B, H, O, D, Q, G\}, \item \{C, F, A, K, B, H, O, D, Q, G\}, \item \{C, F, A, K, B, H, O, D, I, G\}, \item \{C, F, A, K, B, H, O, D, G, M\}, \item \{C, F, A, K, B, H, D, Q, I, L\}, \item \{C, F, A, K, B, H, D, I, G, M\}, \item \{C, F, A, K, B, H, D, I, G, L\}, \item \{C, F, A, K, B, H, D, I, M, L\}, \item \{C, F, A, B, H, O, D, Q, G, M\}, \item \{C, F, A, B, H, O, D, I, G, M\}, \item \{C, F, A, B, H, D, Q, I, M, L\}, \item \{C, F, A, B, H, D, I, G, M, L\}, \item \{C, A, B, H, O, D, Q, I, G, L\}, \item \{F, A, K, B, H, O, D, Q, G, M\}, \item \{F, A, K, B, H, O, D, Q, G, L\}, \item \{F, A, B, H, O, D, Q, G, M, L\}, \item \{N, J, C, E, F, A, K, B, H, D, I\}, \item \{N, J, C, E, F, A, B, P, O, I, M\}, \item \{N, J, C, E, F, A, B, P, I, M, L\}, \item \{N, J, C, E, F, A, B, H, D, I, G\}, \item \{N, J, C, E, F, A, B, H, D, I, L\}, \item \{N, J, C, E, F, A, B, O, Q, I, L\}, \item \{N, J, C, E, F, K, P, O, I, G, M\}, \item \{N, J, C, E, F, K, P, Q, I, M, L\}, \item \{N, J, C, E, F, K, P, I, G, M, L\}, \item \{N, J, C, E, F, K, O, Q, I, G, L\}, \item \{N, J, C, E, F, P, O, Q, I, M, L\}, \item \{N, J, C, E, A, O, D, Q, I, G, L\}, \item \{N, C, E, A, K, B, H, O, D, I, G\}, \item \{N, C, E, A, K, B, H, D, Q, I, L\}, \item \{N, C, E, A, K, B, H, D, I, G, L\}, \item \{N, K, P, H, O, D, Q, I, G, M, L\}, \item \{J, C, E, F, A, K, B, P, H, D, M\}, \item \{J, C, E, F, A, K, B, H, O, D, G\}, \item \{J, C, E, F, A, B, P, H, D, Q, M\}, \item \{J, C, E, F, A, B, P, H, D, G, M\}, \item \{J, C, E, F, A, B, H, O, D, Q, G\}, \item \{J, F, A, K, B, P, H, O, D, G, M\}, \item \{J, F, A, K, B, P, H, D, Q, M, L\}, \item \{J, F, A, K, B, P, H, D, G, M, L\}, \item \{J, F, A, K, B, H, O, D, Q, G, L\}, \item \{J, F, A, B, P, H, O, D, Q, G, M\}, \item \{C, F, A, K, B, H, O, D, Q, G, M\}, \item \{C, F, A, K, B, H, O, D, I, G, M\}, \item \{C, F, A, K, B, H, D, Q, I, M, L\}, \item \{C, F, A, K, B, H, D, I, G, M, L\}, \item \{C, F, A, B, H, O, D, Q, I, G, L\}, \item \{C, A, K, B, H, O, D, Q, I, G, L\}, \item \{F, A, K, B, H, O, D, Q, G, M, L\}, \item \{N, J, C, E, F, A, K, B, H, D, I, G\}, \item \{N, J, C, E, F, A, K, B, H, D, I, L\}, \item \{N, J, C, E, F, A, B, P, H, D, I, M\}, \item \{N, J, C, E, F, A, B, P, Q, I, M, L\}, \item \{N, J, C, E, F, A, B, H, O, D, I, G\}, \item \{N, J, C, E, F, A, B, H, D, Q, I, L\}, \item \{N, J, C, E, F, A, B, H, D, I, G, L\}, \item \{N, J, C, E, F, P, O, Q, I, G, M, L\}, \item \{N, J, C, E, A, P, O, D, Q, I, G, L\}, \item \{N, C, E, A, B, H, O, D, Q, I, G, L\}, \item \{J, C, E, F, A, K, B, P, H, D, Q, M\}, \item \{J, C, E, F, A, K, B, P, H, D, G, M\}, \item \{J, C, E, F, A, K, B, H, O, D, Q, G\}, \item \{J, C, E, F, A, B, P, H, O, D, G, M\}, \item \{J, F, A, K, B, P, H, O, D, Q, G, M\}, \item \{J, F, A, B, P, H, O, D, Q, G, M, L\}, \item \{C, F, A, K, B, H, O, D, Q, I, G, L\}, \item \{C, F, A, B, H, O, D, Q, I, G, M, L\}, \item \{N, J, C, E, F, A, K, B, P, H, D, I, M\}, \item \{N, J, C, E, F, A, K, B, H, O, D, I, G\}, \item \{N, J, C, E, F, A, K, B, H, D, Q, I, L\}, \item \{N, J, C, E, F, A, K, B, H, D, I, G, L\}, \item \{N, J, C, E, F, A, B, P, H, D, I, G, M\}, \item \{N, J, C, E, F, A, B, P, H, D, I, M, L\}, \item \{N, J, C, E, F, A, B, P, O, Q, I, M, L\}, \item \{N, J, C, E, F, K, P, O, Q, I, G, M, L\}, \item \{N, C, E, A, K, B, H, O, D, Q, I, G, L\}, \item \{J, C, E, F, A, K, B, P, H, O, D, G, M\}, \item \{J, C, E, F, A, B, P, H, O, D, Q, G, M\}, \item \{J, F, A, K, B, P, H, O, D, Q, G, M, L\}, \item \{C, F, A, K, B, H, O, D, Q, I, G, M, L\}, \item \{N, J, C, E, F, A, K, B, P, H, D, I, G, M\}, \item \{N, J, C, E, F, A, K, B, P, H, D, I, M, L\}, \item \{N, J, C, E, F, A, B, P, H, O, D, I, G, M\}, \item \{N, J, C, E, F, A, B, P, H, D, Q, I, M, L\}, \item \{N, J, C, E, F, A, B, P, H, D, I, G, M, L\}, \item \{N, J, C, E, F, A, B, H, O, D, Q, I, G, L\}, \item \{J, C, E, F, A, K, B, P, H, O, D, Q, G, M\}, \item \{N, J, C, E, F, A, K, B, P, H, O, D, I, G, M\}, \item \{N, J, C, E, F, A, K, B, P, H, D, Q, I, M, L\}, \item \{N, J, C, E, F, A, K, B, P, H, D, I, G, M, L\}, \item \{N, J, C, E, F, A, K, B, H, O, D, Q, I, G, L\}, \item \{N, J, C, E, F, A, B, P, H, O, D, Q, I, G, M, L\}, \item \{N, J, C, E, F, A, K, B, P, H, O, D, Q, I, G, M, L\}
\end{enumerate}
\end{multicols}

\section{Python code checking for contradictions of Type III} \label{sec:contradiction-code-2}

This code takes the $1124$ sets which can still be alternation sets (`CurrentAltSets1') and discards any sets with contradictions of Type III.
\begin{lstlisting}[language=python]
import itertools
from updated_alt_sets import currentAltSets1

bigToSmall = {
    "A" : {"a", "d", "j"}, "B" : {"b", "d", "j"}, "C" : {"a", "e", "j"},
    "D" : {"a", "d", "l"}, "E" : {"c", "e", "j"}, "F" : {"b", "f", "j"},
    "G" : {"a", "g", "l"}, "H" : {"b", "d", "l"}, "I" : {"a", "e", "o"},
    "J" : {"c", "f", "j"}, "K" : {"b", "h", "l"},
    "L" : {"a", "i", "o"}, "M" : {"b", "f", "p"}, "N" : {"c", "e", "o"},
    "O" : {"a", "g", "q"}, "P" : {"c", "f", "p"}, "Q" : {"a", "i", "q"},
}
#contradictions
zeroToOne = {
    "b":{"a"}, "c":{"a", "b", "f"}, "e":{"d"}, "f":{"d", "e"},
    "g":{"d","e"}, "h":{"d","e","f","g"}, "i":{"d","e","g"}, "l":{"j"},
    "o":{"j","l"}, "p":{"j","l","o"}, "q":{"j","l","o","i"}, "a":set(),
    "d": set(), "j": set(), "p": set(),
    }
allElementList = ["A", "B", "C", "D", "E", "F", "G", "H", "I", "J", "K", "L", "M", "N",    "O", "P", "Q"]

#turns any imputed set of big letters (summation term) into small ones (conditions)
def takeBigToSmall(combination):
    allSmall = set()
    for element in combination:
        allSmall = allSmall.union(bigToSmall[element])
    return allSmall

#gives small letters (conditions) that need to be >= 0 or there's a contradiction
def giveContradictions(combination):
    allSmall = takeBigToSmall(combination)
    #we are assuming this combination is included in the alt set
    smallContradictions = set() #these also need to be >=0
    for small in allSmall:
        smallContradictions = smallContradictions.union(zeroToOne[small])
    if {"a","g"}.issubset(allSmall):
        smallContradictions = smallContradictions.union("t")
    return smallContradictions

def giveNeeded(combination):
    #gettin all 3-element subsets of the 'contradicting' small leters (conditions)
    contraSubsets = list(itertools.combinations(giveContradictions(combination),3))
    needed = []
    #turning tuples into sets
    contraChangedSubsets = list(set(subset) for subset in contraSubsets)
    #check what big letters (summation terms) are needed
    for subset in contraChangedSubsets:
        for element in allElementList:
            if subset == bigToSmall[element] and element not in combination:
                needed.append(element)
    #account for s_0 and o_0
    if all(x in takeBigToSmall(combination) for x in ['o','s']):
        needed.append('K') #because K is a no-longer used letter
    return needed

#put this all together to discard contradicting alt sets
def removeSubsets(altSets):
    newAltSet = []
    for subset in altSets:
        if len(giveNeeded(subset)) == 0:
            if not(all(x in subset for x in ['A','N']) and 'J' not in subset):
                newAltSet.append(subset)
    print(len(newAltSet))
    return newAltSet

print(removeSubsets(currentAltSets1))
\end{lstlisting}

\section{Alternation sets without contradictions of Type III} \label{sec:alt-sets-II}
\begin{multicols}{3}
\footnotesize
\begin{enumerate}
\item \{\}, \item \{J\}, \item \{C\}, \item \{E\}, \item \{F\}, \item \{A\}, \item \{B\}, \item \{H\}, \item \{O\}, \item \{D\}, \item \{Q\}, \item \{I\}, \item \{G\}, \item \{L\}, \item \{J, E\}, \item \{J, F\}, \item \{C, E\}, \item \{C, A\}, \item \{C, O\}, \item \{C, Q\}, \item \{C, I\}, \item \{C, G\}, \item \{E, B\}, \item \{F, B\}, \item \{A, B\}, \item \{A, O\}, \item \{A, D\}, \item \{A, Q\}, \item \{A, L\}, \item \{B, H\}, \item \{H, D\}, \item \{O, Q\}, \item \{O, I\}, \item \{O, G\}, \item \{D, Q\}, \item \{D, I\}, \item \{D, G\}, \item \{D, L\}, \item \{Q, L\}, \item \{I, G\}, \item \{I, L\}, \item \{G, L\}, \item \{J, E, F\}, \item \{J, F, B\}, \item \{C, E, A\}, \item \{C, A, B\}, \item \{C, A, O\}, \item \{C, A, D\}, \item \{C, A, Q\}, \item \{C, A, I\}, \item \{C, O, Q\}, \item \{C, O, I\}, \item \{C, O, G\}, \item \{C, I, G\}, \item \{C, I, L\}, \item \{A, O, Q\}, \item \{A, D, Q\}, \item \{A, D, G\}, \item \{A, D, L\}, \item \{A, Q, L\}, \item \{O, D, G\}, \item \{O, Q, G\}, \item \{O, Q, L\}, \item \{O, I, G\}, \item \{D, Q, L\}, \item \{D, I, G\}, \item \{D, I, L\}, \item \{D, G, L\}, \item \{Q, I, L\}, \item \{I, G, L\}, \item \{J, E, F, B\}, \item \{C, E, A, B\}, \item \{C, F, A, B\}, \item \{C, A, O, Q\}, \item \{C, A, O, I\}, \item \{C, A, D, Q\}, \item \{C, A, D, I\}, \item \{C, A, D, G\}, \item \{C, A, I, L\}, \item \{C, O, Q, G\}, \item \{C, O, I, G\}, \item \{C, Q, I, L\}, \item \{C, I, G, L\}, \item \{A, B, H, D\}, \item \{A, O, D, G\}, \item \{A, O, Q, L\}, \item \{A, D, Q, L\}, \item \{A, D, G, L\}, \item \{O, D, Q, G\}, \item \{O, D, I, G\}, \item \{O, Q, I, L\}, \item \{O, Q, G, L\}, \item \{D, Q, I, L\}, \item \{D, I, G, L\}, \item \{C, A, B, H, D\}, \item \{C, A, O, D, G\}, \item \{C, A, D, I, G\}, \item \{C, A, D, I, L\}, \item \{C, A, Q, I, L\}, \item \{C, O, Q, I, L\}, \item \{A, O, D, Q, G\}, \item \{O, D, Q, G, L\}, \item \{O, Q, I, G, L\}, \item \{J, C, E, F, A, B\}, \item \{C, F, A, B, H, D\}, \item \{C, A, K, B, H, D\}, \item \{C, A, B, H, D, I\}, \item \{C, A, B, H, D, G\}, \item \{C, A, O, D, Q, G\}, \item \{C, A, O, D, I, G\}, \item \{C, A, O, Q, I, L\}, \item \{C, A, D, Q, I, L\}, \item \{C, A, D, I, G, L\}, \item \{C, O, Q, I, G, L\}, \item \{A, O, D, Q, G, L\}, \item \{O, D, Q, I, G, L\}, \item \{C, F, A, K, B, H, D\}, \item \{C, F, A, B, H, D, I\}, \item \{C, F, A, B, H, D, G\}, \item \{C, A, K, B, H, D, G\}, \item \{C, A, B, H, D, I, G\}, \item \{J, C, E, F, A, B, H, D\}, \item \{C, F, A, K, B, H, D, G\}, \item \{C, F, A, B, H, D, I, G\}, \item \{C, A, K, B, H, D, I, G\}, \item \{C, A, B, H, D, I, G, L\}, \item \{C, A, O, D, Q, I, G, L\}, \item \{J, C, E, F, A, K, B, H, D\}, \item \{J, C, E, F, A, B, H, D, G\}, \item \{C, F, A, K, B, H, D, I, G\}, \item \{C, F, A, B, H, D, I, G, L\}, \item \{C, A, K, B, H, D, I, G, L\}, \item \{N, J, C, E, F, A, B, H, D, I\}, \item \{J, C, E, F, A, K, B, H, D, G\}, \item \{C, F, A, K, B, H, D, I, G, L\}, \item \{C, A, B, H, O, D, Q, I, G, L\}, \item \{N, J, C, E, F, A, B, H, D, I, G\}, \item \{C, F, A, B, H, O, D, Q, I, G, L\}, \item \{C, A, K, B, H, O, D, Q, I, G, L\}, \item \{N, J, C, E, F, A, K, B, H, D, I, G\}, \item \{N, J, C, E, F, A, B, P, H, D, I, M\}, \item \{N, J, C, E, F, A, B, H, D, I, G, L\}, \item \{C, F, A, K, B, H, O, D, Q, I, G, L\}, \item \{N, J, C, E, F, A, K, B, H, D, I, G, L\}, \item \{N, J, C, E, F, A, B, P, H, D, I, G, M\}, \item \{N, J, C, E, F, A, K, B, P, H, D, I, G, M\}, \item \{N, J, C, E, F, A, B, P, H, D, I, G, M, L\}, \item \{N, J, C, E, F, A, B, H, O, D, Q, I, G, L\}, \item \{N, J, C, E, F, A, K, B, P, H, D, I, G, M, L\}, \item \{N, J, C, E, F, A, K, B, H, O, D, Q, I, G, L\}
\end{enumerate}
\end{multicols}

\section{Python code checking for contradictions of Type II and III} \label{sec:contradiction-code-1-and-2}

This program takes the $150$ sets which can still be alternation sets (`CurrentAltSets2') and discarsd any sets with contradictions of Type II and III.

\begin{lstlisting}[language=Python]
import itertools
from updated_alt_sets import currentAltSets2

#getting coefficients
bigToSmall = {
    "A" : {"a", "d", "j"}, "B" : {"b", "d", "j"}, "C" : {"a", "e", "j"},
    "D" : {"a", "d", "l"}, "E" : {"c", "e", "j"}, "F" : {"b", "f", "j"},
    "G" : {"a", "g", "l"}, "H" : {"b", "d", "l"}, "I" : {"a", "e", "o"},
    "J" : {"c", "f", "j"}, "K" : {"b", "h", "l"},
    "L" : {"a", "i", "o"}, "M" : {"b", "f", "p"}, "N" : {"c", "e", "o"},
    "O" : {"a", "g", "q"}, "P" : {"c", "f", "p"}, "Q" : {"a", "i", "q"},
}

#contradictions of type III
zeroToOne = {
    "b":{"a"}, "c":{"a", "b", "f"}, "e":{"d"}, "f":{"d", "e"},
    "g":{"d","e"}, "h":{"d","e","f","g"}, "i":{"d","e","g"}, "l":{"j"},
    "o":{"j","l"}, "p":{"j","l","o"}, "q":{"j","l","o","i"}, "a":set(),
    "d": set(), "j": set(), "p": set(),
}

allElementList = ["A", "B", "C", "D", "E", "F", "G", "H", "I", "J", "K", "L", "M", "N",    "O", "P", "Q"]
allElementSet = set(allElementList)

#turns any imputed set of big letters into small ones
def takeBigToSmall(combination):
    allSmall = set()
    for element in combination:
        allSmall = allSmall.union(bigToSmall[element])
        allSmall.union(bigToSmall[element]))
    return allSmall

#gives small letters that need to be >=0 or there's a type II contradiction
def giveContradictionsII(combination):
    return takeBigToSmall(combination)

#gives small letters that need to be >= 0 or there's a type III contradiction
def giveContradictionsIII(combination):
    allSmall = takeBigToSmall(combination)

    #we are assuming this combination is included
    smallContradictions = set() #these also need to be >=0
    for small in allSmall:
        smallContradictions = smallContradictions.union(zeroToOne[small])
    if {"a","g"}.issubset(allSmall):
        smallContradictions = smallContradictions.union("t")
        smallContradictions.union(zeroToOne[small]))

    return smallContradictions

def giveNeeded(combination): #combination = an alternation set
    #getting togher all small letters that need to be >= 0
    all_contradictions = giveContradictionsII(combination).union(giveContradictionsIII(combination))
    
    #gettin all 3-element subsets of all_contradictions
    contraSubsets = list(itertools.combinations(all_contradictions,3))

    needed = []
    #turning tuples into sets
    contraChangedSubsets = list(set(subset) for subset in contraSubsets)

    #check what big letters need to be in the alt set
    for subset in contraChangedSubsets:
        for element in allElementList:
            if subset == bigToSmall[element] and element not in combination:
                needed.append(element)
    #account for s_0 and o_0
    if all(x in takeBigToSmall(combination) for x in ['o','s']):
        needed.append('K') #because K is a no-longer used letter
    return needed

def removeSubsets(altSets):
    newAltSet = []
    for subset in altSets:
        if len(giveNeeded(subset)) == 0:
            if not(all(x in subset for x in ['A','N']) and 'J' not in subset):
                newAltSet.append(subset)
    print(len(newAltSet))
    return newAltSet

print(removeSubsets(currentAltSets2))
\end{lstlisting}

\section{Complete list of Weyl alternation sets of \texorpdfstring{$\ccc$}{sp\_6(C)}}
\label{sec:alt-sets-final}
\begin{multicols}{3}
\footnotesize
\begin{enumerate}
\item \{\}, \item \{A\}, \item \{C, A\}, \item \{A, B\}, \item \{A, D\}, \item \{C, A, B\}, \item \{C, A, D\}, \item \{C, F, A, B\}, \item \{C, A, D, I\}, \item \{C, A, D, G\}, \item \{A, B, H, D\}, \item \{C, A, B, H, D\}, \item \{C, A, D, I, G\}, \item \{J, C, E, F, A, B\}, \item \{C, F, A, B, H, D\}, \item \{C, A, B, H, D, I\}, \item \{C, A, B, H, D, G\}, \item \{C, A, D, I, G, L\}, \item \{C, F, A, B, H, D, I\}, \item \{C, F, A, B, H, D, G\}, \item \{C, A, B, H, D, I, G\}, \item \{J, C, E, F, A, B, H, D\}, \item \{C, F, A, K, B, H, D, G\}, \item \{C, F, A, B, H, D, I, G\}, \item \{C, A, B, H, D, I, G, L\}, \item \{C, A, O, D, Q, I, G, L\}, \item \{J, C, E, F, A, B, H, D, G\}, \item \{C, F, A, K, B, H, D, I, G\}, \item \{C, F, A, B, H, D, I, G, L\}, \item \{N, J, C, E, F, A, B, H, D, I\}, \item \{J, C, E, F, A, K, B, H, D, G\}, \item \{C, F, A, K, B, H, D, I, G, L\}, \item \{C, A, B, H, O, D, Q, I, G, L\}, \item \{N, J, C, E, F, A, B, H, D, I, G\}, \item \{C, F, A, B, H, O, D, Q, I, G, L\}, \item \{N, J, C, E, F, A, K, B, H, D, I, G\}, \item \{N, J, C, E, F, A, B, P, H, D, I, M\}, \item \{N, J, C, E, F, A, B, H, D, I, G, L\}, \item \{C, F, A, K, B, H, O, D, Q, I, G, L\}, \item \{N, J, C, E, F, A, K, B, H, D, I, G, L\}, \item \{N, J, C, E, F, A, B, P, H, D, I, G, M\}, \item \{N, J, C, E, F, A, K, B, P, H, D, I, G, M\}, \item \{N, J, C, E, F, A, B, P, H, D, I, G, M, L\}, \item \{N, J, C, E, F, A, B, H, O, D, Q, I, G, L\}, \item \{N, J, C, E, F, A, K, B, P, H, D, I, G, M, L\}, \item \{N, J, C, E, F, A, K, B, H, O, D, Q, I, G, L\}
\end{enumerate}
\end{multicols}

\section{Sage code for calculating Weyl alternation sets}
\label{sec:calculating-alt-sets}

In this program, 
\begin{itemize}
\item {\ttfamily weyl\_alternation\_set(lam, mu)} returns the Weyl alternation set for $\lambda = $ {\ttfamily lam} and $\mu = $ {\ttfamily mu}, 
\item {\ttfamily alt\_sets\_with\_mus(lamnum, munum)} varies the coefficients of $\lambda$ from $0$ to {\ttfamily lamnum} and the coefficients of $\mu$ from $0$ to {\ttfamily munum} and returns a list of the alternation sets which appeared along with a $\lambda$ and $\mu$ pair which induces each alternation set,
\item and, similarly, {\ttfamily alt\_sets\_new\_lattice(lamnum,munum)} varies the coefficients of $\lambda$ from $0$ to {\ttfamily lamnum} and the coefficients of $\mu$ from $0$ to {\ttfamily munum} for $\frac{m+k+x+z}{2} \in \Z$ and returns a list of the alternation sets which appeared along with a $\lambda$ and $\mu$ pair which induces each alternation set.
\end{itemize}

\begin{lstlisting}[language=Python]
def weyl_alternation_set(lam, mu=(0,0,0)):
    """
    Finds the Weyl alternation set for lam = mw_1 + nw_2 + kw_3
    and mu = c_1w_1 + c_2w_2 + c_3w+3
    """
    (m_, n_, k_) = lam
    (c1_, c2_, c3_) = mu

    alt_set = set()
    sigma_results_dict = get_sigma_results_dict(lam,mu)

    for s in W:
        if all( [ sigma_results_dict[s][j] >= 0 for j in range(0,3) ] ):
            alt_set.add(s)
    #print('alternation set')
    #print(alt_set)
    return alt_set

def weyl_alternation_set_17(lam, mu=(0,0,0)):
    """
    Finds the Weyl alternation set for lam = mw_1 + nw_2 + kw_3
    and mu = c_1w_1 + c_2w_2 + c_3w+3
    doesn't look at sigmas we know are never in the alternation set to make program more efficient
    """
    (m_, n_, k_) = lam
    (c1_, c2_, c3_) = mu

    alt_set = set()
    sigma_results_dict = get_sigma_results_dict(lam,mu)

    for s in W_17: 
        if all( [ sigma_results_dict[s][j] >= 0 for j in range(0,3) ] ):
            alt_set.add(s)
    #print('alternation set')
    #print(alt_set)
    return alt_set

'''DEFINITIONS'''

# Takes ambient space vector and converts to column matrix over SR
def ambient_to_list(v):
    return [v[i] for i in range(0,3)]

# Takes vector and converts to a list
def vector_to_list(v):
    return [v[i] for i in range(0,3)]

# Take R3 vector and gives coordinates in terms of alpha's
# v = c1 * a1 + c2 * a2 + c3 * a3
def ambient_to_alpha_coords(v):
    return root_matrix.solve_right(ambient_to_vector(v))

def vector_to_alpha_coords(v):
    return root_matrix.solve_right(v)

def ambient_to_vector(v):
    return matrix([v[i] for i in range(0,3)],ring=SR).transpose()

var('m n k c1 c2 c3')
W = WeylGroup(['C', 3], prefix='s')
a = W.domain().simple_roots()
P = W.domain().positive_roots()
[s1, s2, s3] = W.simple_reflections()
e=s1*s1
W_17 = [e, s1, s2, s3, s1*s2, s2*s1, s2*s3, s3*s1, s3*s2, s1*s2*s1, s2*s3*s1, s2*s3*s2, s3*s2*s1, s3*s1*s2, s3*s2*s3, s3*s1*s2*s1, s3*s2*s3*s2] 

# Simple root vectors for the Lie algebra of type C_3
a1 = ambient_to_vector(a[1]); #print('a1'); print(a1)
a2 = ambient_to_vector(a[2]); #print('a1'); print(a2)
a3 = ambient_to_vector(a[3]); #print('a1'); print(a3)

root_matrix = matrix([ambient_to_list(a[i]) for i in range(1,4)],ring=SR).transpose()

# Fundamental weight vectors for the Lie algebra of type C_3
w1 = (1/2) * (2*a1 + 2*a2 + 1*a3)
w2 = (1*a1 + 2*a2 + 1*a3)
w3 = (1/2) * (2*a1 + 4*a2 + 3*a3)

lam = m*w1 + n*w2 + k*w3
mu = c1*w1 + c2*w2 + c3*w3
rho = ambient_to_vector((1/2)*sum(P))

''' OTHER FUNCTIONS WE NEED '''

def get_sigma_results_dict(lam, mu):
    # lam + mu = (m,n,k) + (c1,c2,c3) = (m,n,k,c1,c2,c3)
    #     ^ This is NOT vector addition! We're concatenating the tuples
    #       to pass into the callable.
    return dict(
    [ (s,[component(*(lam+mu)) for component in weyl_action_callable_dict[s]] ) for s in W
    ]
    )

def weyl_actions():
    """
    Returns a dict with key,value pairs: (s, (b_1, b_2, b_3)) where s is an
    element of the Weyl group and b_1, b_2, b_3 are the coefficients
    of a_1, a_2, and a_3, respectively, of s(lambda + rho) - (rho + mu)
    """
    return [(s,vector_to_alpha_coords(s.matrix() * (lam + rho) - (rho + mu))) for s in W]

weyl_action_callable_dict = dict([ (p[0],[fast_callable(p[1][i][0], vars=[m,n,k,c1,c2,c3]) for i in range(0,3)]) for p in weyl_actions()])

def alt_sets_with_mus(lamnum, munum):
    '''
    iterates through lam=(0,0,0) to (lamnum, lamnum, lamnum) and mu=(0,0,0) to (munum, munum, munum)
    prints alternation sets which show up with a lambda and mu that induce it
    '''
    alt_sets = []
    alt_set_exp = []
    
    #for lam = (m,n,k), mu=(x,y,z)
    for m in range(lamnum+1):
        for n in range(lamnum+1):
            for k in range(lamnum+1):
                for x in range(munum+1):
                    for y in range(munum+1):
                        for z in range(munum+1):
                            alt_set = weyl_alternation_set_17(lam=(m,n,k), mu=(x,y,z))
                            if alt_set not in alt_sets: #if a new set
                                alt_sets.append(alt_set) #record alt set
                                alt_set_exp.append((alt_set,(m,n,k),(x,y,z))) #record how we got it
                                print((alt_set,(m,n,k),(x,y,z)))
    return alt_set_exp


def alt_sets_new_lattice(lamnum, munum):
    '''
    iterates through lam=(0,0,0) to (lamnum, lamnum, lamnum) and mu=(0,0,0) to (munum, munum, munum)
    For m+k+x+z divisible by 2,
    prints alternation sets which show up with a lambda and mu that induce it
    '''
    alt_sets = []
    alt_set_exp = []
    
    #for lam = (m,n,k), mu=(x,y,z)
    for m in range(lamnum+1):
        for n in range(lamnum+1):
            for k in range(lamnum+1):
                for x in range(munum+1):
                    for y in range(munum+1):
                        for z in range(munum+1):
                            if (m+k)%2 == 0 and (x+z)%2 == 0: #if its in our lattice
                                alt_set = weyl_alternation_set_17(lam=(m,n,k), mu=(x,y,z))
                                if alt_set not in alt_sets: #if a new set
                                    alt_sets.append(alt_set) #record alt set
                                    alt_set_exp.append((alt_set,(m,n,k),(x,y,z))) #record how we got it
                                    print((alt_set,(m,n,k),(x,y,z)))
    return alt_set_exp

#gives alternation set for a lambda and a mu
#weyl_alternation_set(lam=(2,1,0), mu=(0,0,0))
#weyl_alternation_set_17(lam=(2,1,0), mu=(0,0,0))

#gives all alt sets for lam going up to lamnum and mu going up to munum
#alt_sets_with_mus(lamnum=0,munum=50)

#gives all alt sets for lam going up to lamnum and mu going up to munum for m+k+x+z divisible by 2
alt_sets_new_lattice(lamnum=10, munum=10)

\end{lstlisting}
\end{document}